\documentclass[12pt]{amsart} \usepackage{amssymb} \usepackage{mathabx}

\newcommand{\wrlab}[1]{\label{#1}}

\newtheorem{thm}{THEOREM}[section]

\newtheorem{lem}[thm]{LEMMA} 
\newtheorem{cor}[thm]{COROLLARY} 
 
 \newtheorem{thm*}{THEOREM}[]

\newcommand{\secref}[1]{\S\ref{#1}}

\newcommand{\tref}[1]{Theorem~\ref{#1}}
\newcommand{\cref}[1]{Corollary~\ref{#1}}

\newcommand{\lref}[1]{Lemma~\ref{#1}}

\def\N{{\mathbb N}} \def\Z{{\mathbb Z}} \def\Q{{\mathbb Q}}
\def\R{{\mathbb R}}

\def\scrd{{\mathcal D}} \def\scrf{{\mathcal F}} 
\def\scri{{\mathcal I}}  
\def\scrm{{\mathcal M}} \def\scrp{{\mathcal P}} 
  \def\scrt{{\mathcal T}}
\def\scru{{\mathcal U}} \def\scrv{{\mathcal V}} 
 \def\scrx{{\mathcal X}} 
  \def\scrc{{\mathcal C}}

\def\bfzero{{\bf 0}}

\font\tenolde=eufm10 at 10pt
\font\sevenolde=eufm7
\font\fiveolde=eufm5
\newfam\oldefam
\textfont\oldefam=\tenolde
\scriptfont\oldefam=\sevenolde
\scriptscriptfont\oldefam=\fiveolde

\def\barscru{{\overline{\scru}}}

\def\barA{{\overline{A}}}
\def\barB{{\overline{B}}}
\def\barC{{\overline{C}}}

\def\barI{{\overline{I}}}
\def\barJ{{\overline{J}}}

\def\barN{{\overline{N}}}
\def\barP{{\overline{P}}}

\def\barV{{\overline{V}}}
\def\barW{{\overline{W}}}
\def\barX{{\overline{X}}}

\def\bargamma{{\overline{\gamma}}}

\def\hatI{{\widehat{I}}}

\def\cksigma{{\widecheck{\sigma}}}

\def\Stab{\hbox{Stab}}

\def\barI{{\overline{I}}}
\def\Id{\hbox{Id}}

\begin{document}
\bibliographystyle{alpha}


\title[Freeness in higher order frame bundles]
{Freeness in higher order frame bundles}


\author{Scot Adams}
\address{School of Mathematics\\ University of Minnesota\\Minneapolis, MN 55455
\\ adams@math.umn.edu}

\date{September 4, 2015\qquad Printout date: \today}

\begin{abstract}
  We provide counterexamples to P.~Olver's freeness conjecture for
  $C^\infty$ actions. In fact, we show that a counterexample exists
  for any connected real Lie group with noncompact center, as well as
  for the additive group of the integers.
\end{abstract}

\maketitle

 

\section{Introduction\wrlab{sect-intro}}

P.~Olver's freeness conjecture (in his words) asserts: ``If a Lie
group acts effectively on a manifold, then, for some $n<\infty$, the
action is free on [a nonempty] open subset of the jet bundle of order $n$.''
There is some ambiguity in this wording:
No mention is made of connectedness of the group or manifold,
the particular choice of jet bundle isn't made precise and
the smoothness of the action is left unspecified.

In this note, we provide counterexamples to one interpretation of~the
freeness conjecture for $C^\infty$ actions and higher order frame bundles.
Those who know Olver's work will understand that there are a family
of~associated jet bundles to which he generally refers,
and work in frame bundles then informs results in these jet bundles,
through the associated bundle construction.
In the $C^\infty$ context, Olver has noted that,
to avoid elementary counterexamples,
``effective'' must be strengthened to ``fixpoint rare'',
which we define in \secref{sect-notation} below.
In \tref{thm-induction-gives-ctrx} and \lref{lem-noncpt-center-gives-inf-cyclic},
we show that a counterexample exists for any connected real Lie group
with noncompact center, as well as for the additive group of the integers.
We also prove (\cite{adams:fAbcentstab}) the validity of the conjecture for
connected real Lie groups with compact center.
Finally, in \cite{ao:prolgpactsfree},
we describe a certain ``meager'' modification of the $C^\omega$ conjecture,
and prove it holds for all connected real Lie groups.

For any fixed group, the $C^\infty$ conjecture implies the $C^\omega$ conjecture,
so the $C^\omega$ conjecture (on frame bundles) is now proved for connected real Lie groups with compact center.
In \cite{adams:comegactx}, we offer a $\Z$-action on a manifold with infinitely generated fundamental group
which, after induction of actions, provides a $C^\omega$ counterexample for any connected Lie group with
{\it noncompact} center.
There is the possibility that the construction in
\tref{thm-induction-gives-ctrx} could be modified to make
counterexamples to the $C^\omega$ conjecture on a contractible manifold, {\it e.g.}, $\R^4$.
The main difficulty in~such an extension appears to be technical, and revolves around
developing a~good understanding of convergence of sequences in $C^\omega$ with~respect to~some well-chosen topology.
For this, D.~Morris' unpublished note~\cite{morris:tosandp} may be useful.

The present writeup, also, is not intended for publication.

\section{Miscellaneous notation and terminology\wrlab{sect-notation}}

A subset of a topological space is {\bf meager} (a.k.a.~{\bf of first category}) if it is a countable union of nowhere dense sets.
A subset of a topological space is {\bf nonmeager} (a.k.a.~{\bf of second category}) if it is not meager.
A subset of a topological space is {\bf comeager} (a.k.a.~{\bf residual}) if its complement is meager.

Let $\N:=\{1,2,3,\ldots\}$.
Let $\scri:=\{(-a,a)\subseteq\R\,|\,a\in\N\}$.
For every $I\in\scri$, let $a_I:=\sup I$, so $a_I\in\N$ and $I=(-a_I,a_I)$.
For every $I\in\scri$, for every integer $n\ge1$, we define $nI:=(-na_I,na_I)\in\scri$;
then $a_{nI}=na_I$.
For every $I\in\scri$, let $\barI:=[-a_I,a_I]$ be the closure in $\R$ of $I$.
For every $I,J\in\scri$, let $I+J:=(-a_I-a_J,a_I+a_J)\in\scri$;
then $a_{I+J}=a_I+a_J$.
We define $I_0:=(-1,1)\in\scri$; then $a_{I_0}=1$.
For all $I,J\in\scri$, we have:
\begin{itemize}
\item[]$[\,(I\subseteq J)\,\,\Leftrightarrow\,\,(a_I\le a_J)\,]$
\quad and \quad
$[\,(\,\barI\subseteq J)\,\,\Leftrightarrow\,\,(a_I<a_J)\,]$.
\end{itemize}

For this entire note, fix a $C^\infty$ function $\zeta:\R\to\R$ such that
\begin{itemize}
\item$\zeta=1$ on $(-\infty,0]$, \qquad $\zeta=0$ on $[1,\infty)$ \qquad and
\item$\zeta'<0$ on $(0,1)$.
\end{itemize}
Then $0<\zeta<1$ on $(0,1)$
and $\zeta$ is decreasing on $(0,1)$.
Moreover, $\zeta$ is nonincreasing on $\R$.
For every $I\in\scri$, define $\zeta_I:\R\to\R$ by
\begin{itemize}
\item$\zeta_I(x)=\zeta(-x-a_I)$ on $x\le0$ \qquad and
\item$\zeta_I(x)=\zeta(x-a_I)$ on $x\ge0$.
\end{itemize}
Then, for all $I\in\scri$, we have
\begin{itemize}
\item$\zeta_I=1$ on $\barI$, \qquad $\zeta_I=0$ on $\R\backslash(I+I_0)$,
\item$0<\zeta_I<1$ on $(I+I_0)\backslash\barI$,
\item$\zeta_I'>0$ on $(-a_I-1,-a_I)$ \qquad and
\item$\zeta_I'<0$ on $(a_I,a_I+1)$.
\end{itemize}

For all integers $d\ge1$, let $\Id_d:\R^d\to\R^d$ be the identity map, defined by $\Id_d(\sigma)=\sigma$.
For every subset $S\subseteq\R$, for every integer $d\ge1$, we define
$S^d:= S\times S\times\cdots\times S\subseteq\R^d$.

Let $d\ge1$ be an integer.
A function $V:\R^d\to\R^d$ will be said to be {\bf complete}
if it is $C^\infty$ and represents a complete vector field on $\R^d$.
For any complete $V:\R^d\to\R^d$, we will use the notation $\Phi_t^V:\R^d\to\R^d$
to denote the time~$t$ flow of $V$,
defined by the ODE $(d/dt)(\Phi_t^V(\sigma))=V(\Phi_t^V(\sigma))$
and by the initial value condition $\Phi_0^V=\Id_d$.

Let $d\ge1$ be an integer.
Let $V:\R^d\to\R^d$ be complete.
For any~$A\subseteq\R$, for any $B\subseteq\R^d$, let $\Phi_A^V(B):=\{\Phi_a^V(b)\,|\,a\in A,b\in B\}$.
For any $A\subseteq\R$, for any $b\in\R^d$, let $\Phi_A^V(b):=\{\Phi_a^V(b)\,|\,a\in A\}$.
For any $a\in\R$, for any $B\subseteq\R^d$, let $\Phi_a^V(B):=\{\Phi_a^V(b)\,|\,b\in B\}$.

Let $d\ge1$ be an integer.
Let $V:\R^d\to\R^d$ be complete.
Let $\sigma\in\R^d$.
We say that $(V,\sigma)$ is {\bf periodic} if there exists an integer $n\ne0$
such that $\Phi_n^V(\sigma)=\sigma$.
For any integer $k\ge0$,
we say that $(V,\sigma)$ is {\bf periodic to order $k$} if there is an integer $n\ne0$
such that the map $\Phi_n^V:\R^d\to\R^d$ agrees with the identity $\Id_d:\R^d\to\R^d$ to order $k$ at~$\sigma$.
We say that $(V,\sigma)$ is {\bf periodic to all orders} if there is an integer $n\ne0$
such that the map $\Phi_n^V:\R^d\to\R^d$ agrees with $\Id_d:\R^d\to\R^d$ to all orders at $\sigma$.

Let $d\ge1$ be an integer.
Let $V:\R^d\to\R^d$ be complete.
Let $S\subseteq\R^d$.
We say $S$ is {\bf$V$-invariant} if,
for all $\sigma\in S$, $\Phi_\R^V(\sigma)\subseteq S$.
We say $S$~is {\bf locally $V$-invariant} if,
for all $\sigma\in S$, there is an open neighborhood $N$ in $\R$ of $0$
such that $\Phi_N^V(\sigma)\subseteq S$.

For all integers $j\in[1,4]$, let $\Pi_j:\R^4\to\R$ be projection onto the $j$th coordinate,
defined by $\Pi_j(x_1,x_2,x_3,x_4)=x_j$.
Let $\scrc$ be the set of~$C^\infty$ maps $V:\R^4\to\R^4$
such that $V(\R^4)\subseteq\barI_0^4$.
Note that, for all $V\in\scrc$, the function $V:\R^4\to\R^4$ is complete.
We define $V_0:\R^4\to\R^4$ by~the rule: for all $\sigma\in\R^4$,
$V_0(\sigma)=(0,0,0,1)$.
Then $V_0\in\scrc$.
For all $a\in\R$, we define
$\scrv(a):=\{V\in\scrc\,|\,V=V_0\hbox{ on }\R^3\times(-\infty,-a)\}$.

For all $I\in\scri$, we define
\begin{eqnarray*}
T(I)&:=&\barI^3\times\{a_I\}\quad\subseteq\quad\barI^4,\\
B(I)&:=&\barI^3\times\{-a_I\}\quad\subseteq\quad\barI^4,\\
T_\circ(I)&:=&I^3\times\{a_I\}\quad\subseteq\quad I^3\times\barI,\\
B_\circ(I)&:=&I^3\times\{-a_I\}\quad\subseteq\quad I^3\times\barI\qquad\hbox{and}\\
\xi_I&:=&(0,0,0,-a_I)\quad\in\quad B_\circ(I).
\end{eqnarray*}
The {\bf straight up map for $I$} is the bijection $SU_I:B(I)\to T(I)$
defined by $SU_I(w,x,y,-a_I)=(w,x,y,a_I)$.
Note that $SU_I(B_\circ(I))=T_\circ(I)$.

Let $V\in\scrc$.
We say $\sigma\in\R^4$ is {\bf undeterred by $V$} if
$\Pi_4(\Phi_\R^V(\sigma))=\R$.
The {\bf undeterred set for $V$} is
the set $\scru(V)$ of all $\sigma\in\R^4$ such that $\sigma$~is undeterred by $V$.
This set $\scru(V)$ is $V$-invariant.
We say $V$ is {\bf porous} if $\scru(V)$ is dense in $\R^4$.
For example, $\scru(V_0)=\R^4$, so $V_0$ is porous.
By a {\bf deterrence system} we mean an element $(V,I)\in\scrc\times\scri$
such that $V=V_0$ on $(\R^4)\backslash(I^4)$.
Let $\scrd$ be the set of all deterrence systems.
Note, for all $(V,I)\in\scrd$, that $V\in\scrv(a_I)$.
For all $(V,I)\in\scrd$, we define
\begin{eqnarray*}
\scru_\circ(V,I)&:=&(\scru(V))\cap(I^4),\\
\barscru(V,I)&:=&(\scru(V))\cap(\,\barI^4\,),\\
\scru_B(V,I)&:=&(\scru(V))\cap(B(I)),\\
\scru_B^\circ(V,I)&:=&(\scru(V))\cap(B_\circ(I)).
\end{eqnarray*}

For any $I\in\scri$, we define $\scrc_I:=\{P\in\scrc\,|\,P=V_0\hbox{ on }(\,\overline{3I}\,)^4\}$.
For any $(V,I)\in\scrd$, for any $P\in\scrc_I$,
we will denote by $\scrx_I(P,V)$ the function $X\in\scrc$ defined by:
\quad$X:=V$ on $(2I)^4$ \quad and \quad $X:=P$ on $\R\backslash(\,\barI^4\,)$.

\noindent
The ``$\scrx$'' stands for ``exchange'':
As $(V,I)\in\scrd$, we know that $V=V_0$ on $(\R^4)\backslash(I^4)$.
As $P\in\scrc_I$, we know that $P=V_0$ on $(\,\overline{3I}\,)^4$.
To construct $\scrx_I(P,V)$,
we start with~$P$, and then ``exchange'' $V_0$ for $V$ on~$I^4$.

Let $(V,I)\in\scrd$.
For any $(W,J)\in\scrd$, we say
$(W,J)$ is a {\bf modification of~$(V,I)$} if: \qquad
both \qquad [ $a_I<a_J$ ] \qquad and \qquad [ $W=V$ on $\barI^4$ ].

\noindent
We define
\begin{eqnarray*}
\scrm(V,I)&:=&\{\,(W,J)\in\scrd\,\,\,|\,\,(W,J)\hbox{ is a modification of }(V,I)\,\},\\
\scrm_*(V,I)&:=&\{\,(W,J)\in\scrm(V,I)\,\,\,|\,\,W\in\scrv(a_I)\,\}.
\end{eqnarray*}

Let $I\in\scri$.
We define $\scrd_I^\times:=\{(P,K)\in\scrd\,|\,4I\subseteq K\hbox{ and }P\in\scrc_I\}$.
We denote by $\scrp_I$ the set of $(P,K)\in\scrd_I^\times$ such that,
for some integer $m>2a_I$, we have:
\begin{itemize}
\item$\Phi_m^P$ agrees with $\Id_4$ to all orders at $\xi_I$ \qquad and
\item for all $\tau\in B_\circ(I)$, \, for all $t\in(0,m)$,
      $$[\,\Phi_t^P(\tau)\in I^4\,]\qquad\Leftrightarrow\qquad[\,t<2a_I\,].$$
\end{itemize}

For any $C^\infty$ manifold $M$, for any integer $k\ge0$,
let $\pi_k^M:F_kM\to M$ denote the $k$th order frame bundle of $M$.

An action of a group $G$ on a topological space $X$ is
{\bf fixpoint rare} if, for any nonempty open subset $U$ of $X$,
for all~$g\in G\backslash\{1_G\}$, there exists $u\in U$ such that $gu\ne u$.
Any fixpoint rare action is effective.
For a $C^\omega$ action on a connected manifold, fixpoint rare and effective are equivalent.
Any continuous transitive action of~a~real Lie group preserves a $C^\omega$ structure,
from which it follows that:
By Lemma 6.1 of \cite{adams:locfree},
a continuous action of a connected real Lie group $G$ on a topological space $X$ is fixpoint rare iff,
for every nonempty $G$-invariant open subset $V$ of~$X$,
the $G$-action on $V$ is effective.

\section{Miscellaneous results\wrlab{sect-misc-results}}

\begin{lem}\wrlab{lem-fn-with-growth-is-proper}
Let $f:\R\to\R$ be differentiable.
Assume, for all $t\ge0$, that $f'(t)\ge0$.
Let $K$ be be a compact subset of~$\R$.
Assume that $f([0,\infty))\subseteq K$.
Let $a\ge0$ and $b>0$.
Then there exists $r\ge0$ such that
both $f'(r)<b$ and $f'(r+a)<b$.
\end{lem}

\begin{proof}
Suppose, for a contradiction, that, for all $t\ge0$, we have:
\begin{itemize}
\item[$(*)$] \qquad either \qquad $f'(t)\ge b$ \qquad or \qquad $f'(t+a)\ge b$.
\end{itemize}

Define $g:\R\to\R$ by $g(t)=f(t+a)$.
Since $f'\ge0$ on $[0,\infty)$, by the Mean Value Theorem,
$f$ is nondecreasing on $[0,\infty)$.
Then, because $a\ge0$, we have $f\le g$ on $[0,\infty)$.
Let $h:=f+g$.
Then $h\le 2g$ on $[0,\infty)$.
For all $t\ge0$, since $f'(t)\ge0$, since $f'(t+a)\ge0$
and since $(*)$ holds,
it follows that $(f'(t))+(f'(t+a))\ge b$,
{\it i.e.}, that $h'(t)\ge b$.

Let $c:=h(0)$.
Then, by the Mean Value Theorem,
for all $t>0$, we have $[(h(t))-c]/t\ge b$.
Then, for all $t\ge0$, $h(t)\ge c+bt$.
Then, since $b>0$, we get $\displaystyle{\lim_{t\to\infty}[h(t)]=\infty}$.
So, because $h\le2g$ on $[0,\infty)$,
it follows that $\displaystyle{\lim_{t\to\infty}[g(t)]=\infty}$.
Then $\displaystyle{\lim_{t\to\infty}[f(t)]=\lim_{t\to\infty}[g(t-a)]=\infty}$.
However, $K$ is compact and $f([0,\infty))\subseteq K$, contradiction.
\end{proof}

\begin{lem}\wrlab{lem-single-preimages-give-interval-preimage}
Let $f:\R\to\R$ be continuous.
Let $a,c,s,u\in\R$.
Assume that $a<c$ and that $s<u$.
Assume that $f^{-1}(a)=\{s\}$ and that $f^{-1}(c)=\{u\}$.
Assume that $\displaystyle{\lim_{t\to-\infty}[f(t)]=-\infty}$
and that $\displaystyle{\lim_{t\to\infty}[f(t)]=\infty}$.
Then $f^{-1}((a,c))=(s,u)$.
\end{lem}

\begin{proof}
Since $[f^{-1}(a)]\cap(s,u)=\emptyset=[f^{-1}(c)]\cap(s,u)$,
it follows that $a,c\notin f((s,u))$.
Because $f(s)=a<c$ and $c\notin f((s,u))$, it follows,
from the Intermediate Value Theorem,
that $f([s,u))\subseteq(-\infty,c)$.
Because $f(u)=c>a$ and $a\notin f((s,u))$, it follows,
from the Intermediate Value Theorem,
that $f((s,u])\subseteq(a,\infty)$.
Then
$$f((s,u))\qquad\subseteq\qquad(-\infty,c)\,\cap\,(a,\infty)\qquad=\qquad(a,c).$$
Then $(s,u)\subseteq f^{-1}((a,c))$.
It remains to show that $f^{-1}((a,c))\subseteq(s,u)$.
Let $t_0\in f^{-1}((a,c))$.
We wish to prove $s<t_0<u$.
We will show $s<t_0$; the proof of $t_0<u$ is similar.

Since $\displaystyle{\lim_{t\to-\infty}[f(t)]=-\infty}$,
choose $p<t_0$ such that $f(p)<a$.
Since $t_0\in f^{-1}((a,c))$,
it follows that $f(t_0)>a$.
By the Intermediate Value Theorem,
choose $q\in(p,t_0)$ such that $f(q)=a$.
Then $q\in f^{-1}(a)=\{s\}$, so $q=s$.
Then $s=q<t_0$.
\end{proof}

\section{Results about $V_0$\wrlab{sect-res-V0}}

Recall, from \secref{sect-notation}, the definition of $V_0$.
For all $t,w,x,y,z\in\R$, if $\rho:=(w,x,y,z)\in\R^4$, then
$\Phi_t^{V_0}(\rho)=(w,x,y,z+t)$.

\begin{lem}\wrlab{lem-vert-is-V0-invar}
Let $S\subseteq\R^3$.
Then all of the following are true:
\begin{itemize}
\item[(i)]$\forall t\in\R$, \quad $\Phi_t^{V_0}(S\times\R)=S\times\R$.
\item[(ii)]$\forall t\in\R$, $\forall a\in\R$, \quad
$\Phi_t^{V_0}(S\times[a,\infty))=S\times[a+t,\infty)$.
\item[(iii)]$\forall t\in\R$, $\forall a\in\R$, \quad
$\Phi_t^{V_0}(S\times(-\infty,a])=S\times(-\infty,a+t]$.
\end{itemize}
\end{lem}

\begin{proof}
Straightforward.
\end{proof}

\begin{lem}\wrlab{lem-omnibus-V0}
Let $I\in\scri$.
Then
\begin{itemize}
\item[(i)]$\Phi_{(0,2a_I)}^{V_0}(B_\circ(I))=I^4$,
\item[(ii)]$\Phi_{(-2a_I,0)}^{V_0}(T_\circ(I))=I^4$,
\item[(iii)]$\Phi_{[0,2a_I]}^{V_0}(T_\circ(I))\subseteq[(\,\overline{3I}\,)^4]\backslash[I^4]$,
\item[(iv)]$\Phi_{[-2a_I,0]}^{V_0}(B_\circ(I))\subseteq[(\,\overline{3I}\,)^4]\backslash[I^4]$,
\item[(v)]$\Phi_{[-2a_I,4a_I]}^{V_0}(B_\circ(I))\subseteq(\,\overline{3I}\,)^4$ \qquad and
\item[(vi)]$\Phi_{[-2a_I,2a_I]}^{V_0}\,(\,[(\,\barI\,)^3\backslash(I^3)]\,\times\,\barI\,)
\,\subseteq[(\,\overline{3I}\,)^4]\backslash[I^4]$.
\end{itemize}
\end{lem}

\begin{proof}
{\it Proof of (i):}
Since $B_\circ(I)=I^3\times\{-a_I\}$, it follows that
$$\Phi_{(0,2a_I)}^{V_0}(B_\circ(I))\qquad=\qquad I^3\,\,\times\,\,(\,-a_I+0\,,\,-a_I+2a_I\,),$$
so $\Phi_{(0,2a_I)}^{V_0}(B_\circ(I))=I^3\times(-a_I,a_I)=I^3\times I=I^4$.
{\it End of proof of (i).}

The proofs of (ii)-(vi) are similarly straightfoward.
\end{proof}

\begin{lem}\wrlab{lem-drop-from-top}
Let $I\in\scri$.
Let $\rho_1\in T_\circ(I)$.
Let $\rho':=\Phi_{-2a_I}^{V_0}(\rho_1)$.
Then $\rho'\in B_\circ(I)$ and $\rho_1=SU_I(\rho')$.
\end{lem}

\begin{proof}
Straightforward.
\end{proof}

\begin{lem}\wrlab{lem-midpoint-connection}
Let $I,J\in\scri$.
Assume $J\subseteq I$.
Let $s:=a_I-a_J$.
Give $B_\circ(I)$ and $B_\circ(J)$ their relative topologies, inherited from $\R^4$.
Then
\begin{itemize}
\item[(i)]$\Phi_s^{V_0}(\xi_I)=\xi_J$, \quad $\Phi_{-s}^{V_0}(\xi_J)=\xi_I$,
\item[(ii)]$\Phi_{[0,s]}^{V_0}(\xi_I)\,\,=\,\,\Phi_{[-s,0]}^{V_0}(\xi_J)\,\,\subseteq\,\,\{(0,0,0)\}\,\times\,[-a_I,-a_J]$ \qquad and
\item[(iii)]$\tau\mapsto\Phi_{-s}^{V_0}(\tau):B_\circ(J)\to B_\circ(I)$ is an open map.
\end{itemize}
\end{lem}

\begin{proof}
Straightforward.
\end{proof}

\begin{lem}\wrlab{lem-Pi4-effect-of-V0}
Let $\sigma\in\R^4$.
Let $t\in\R$.
Then $\Pi_4(\Phi_t^{V_0}(\sigma))=(\Pi_4(\sigma))+t$.
\end{lem}

\begin{proof}
Choose $w,x,y,z\in\R$ such that $\sigma=(w,x,y,z)$.
Then $\Pi_4(\sigma)=z$.
Then $\Pi_4(\Phi_t^{V_0}(\sigma))=\Pi_4((w,x,y,z+t))=z+t=(\Pi_4(\sigma))+t$.
\end{proof}

\begin{lem}\wrlab{lem-either-up-outside-or-down-outside}
Let $I\in\scri$ and let $\rho\in(\R^4)\backslash(I^4)$.
Then
\begin{itemize}
\item[(i)]$\Pi_4(\rho)\ge0\qquad\Rightarrow\qquad\Phi_{[0,\infty)}^{V_0}(\rho)\subseteq(\R^4)\backslash(I^4)$\qquad and
\item[(ii)]$\Pi_4(\rho)\le0\qquad\Rightarrow\qquad\Phi_{(-\infty,0]}^{V_0}(\rho)\subseteq(\R^4)\backslash(I^4)$.
\end{itemize}
\end{lem}

\begin{proof}
Fix $w,x,y,z\in\R$ such that $\rho=(w,x,y,z)$.
If $(w,x,y)\notin I^3$, then $\Phi_\R^{V_0}(\rho)=\{(w,x,y)\}\times\R\subseteq(\R^4)\backslash(I^4)$;
in this case,
$$\hbox{both}\quad\Phi_{(-\infty,0]}^{V_0}(\rho)\subseteq(\R^4)\backslash(I^4)\quad\hbox{and}\quad
\Phi_{[0,\infty)}^{V_0}(\rho)\subseteq(\R^4)\backslash(I^4)$$
are true, and we are done.
We therefore assume that $(w,x,y)\in I^3$.
Then, as $(w,x,y,z)=\rho\notin I^4$, we get $z\notin I$.
That is, $z\notin(-a_I,a_I)$.

If $\Pi_4(\rho)\ge0$, {\it i.e.}, if $z\ge0$, then, because $z\notin(-a_I,a_I)$, we get $z\ge a_I$,
so $\Phi_{[0,\infty)}^{V_0}(\rho)=\{(w,x,y)\}\,\times\,[z,\infty)\subseteq(\R^4)\backslash(I^4)$,
proving (i).

If $\Pi_4(\rho)\le0$, {\it i.e.}, if $z\le0$, then, as $z\notin(-a_I,a_I)$, we get $z\le-a_I$,
so $\Phi_{(-\infty,0]}^{V_0}(\rho)=\{(w,x,y)\}\,\times\,(-\infty,z]\subseteq(\R^4)\backslash(I^4)$,
proving (ii).
\end{proof}

\begin{lem}\wrlab{lem-nbd-in-B-saturates-to-nbd}
Let $I$ and $J$ be open subsets of $\R$.
Let $c\in\R$.
Let $t_0\in J$.
Give $I^3\times\{c\}$ the relative topology inherited from $\R^4$.
Let $N$ be an open subset of $I^3\times\{c\}$.
Then $\Phi_J^{V_0}(N)$ is an open subset of $\R^4$.
\end{lem}

\begin{proof}
Fix an open subset $U$ of $I^3$ such that $N=U\times\{c\}$.
Then $\Phi_J^{V_0}(N)=U\times(c+J)$.
Since $U$ is open in $I^3$, $U$ is open in $\R^3$.
Then, since $c+J$ is open in $\R$,
we see that $\Phi_J^{V_0}(N)$ is open in $\R^4$.
\end{proof}

\section{Coincidence of orbits\wrlab{sect-coinc-orbs}}

\begin{lem}\wrlab{lem-homothetic-flows-local}
Let $d\ge1$ be an integer.
Let $V,W:\R^d\to\R^d$ be complete.
Assume, for all $\sigma\in\R^d$, that there exists $c\in\R$
such that $W(\sigma)=c\cdot(V(\sigma))$.
Let $\sigma_0\in\R^d$.
Then there exists $\delta>0$ such that
$\Phi_{(-\delta,\delta)}^W(\sigma_0)\subseteq\Phi_\R^V(\sigma_0)$.
\end{lem}

\begin{proof}
Define $\beta:\R\to\R^d$ by $\beta(t)=\Phi_t^W(\sigma_0)$.
We wish to prove that there exists $\delta>0$ such that,
for all $t\in(-\delta,\delta)$, we have
$\beta(t)\in\Phi_\R^V(\sigma_0)$.

Let $\bfzero:=(0,\ldots,0)\in\R^d$.
If $W(\sigma_0)=\bfzero$, then,
for all $t\in\R$, we have
$\beta(t)=\Phi_t^W(\sigma_0)=\sigma_0\in\Phi_\R^V(\sigma_0)$,
so, for any choice of $\delta>0$, we are done.
We therefore assume $W(\sigma_0)\ne\bfzero$.
Fix $c_0\in\R$ such that $W(\sigma_0)=c_0\cdot(V(\sigma_0))$.
Then $c_0\cdot(V(\sigma_0))\ne\bfzero$, so $V(\sigma_0)\ne\bfzero$.

Let $M:=(\R^d)\backslash(V^{-1}(\bfzero))$.
Since $V(\sigma_0)\ne\bfzero$,
we see that $\sigma_0\in M$.
Since $V^{-1}(\bfzero)$ is $V$-invariant,
it follows that $M$ is $V$-invariant as well.
Then $\Phi_\R^V(\sigma_0)\subseteq M$.
Define $f:M\to\R$ by
$$\hbox{for all }\sigma\in M,\qquad\qquad W(\sigma)\,\,=\,\,[f(\sigma)]\,[V(\sigma)].$$
Then $f:M\to\R$ is $C^\infty$.
Define $G:\R\to\R$ by $G(q)=f(\Phi_q^V(\sigma_0))$.
Then $G:\R\to\R$ is also $C^\infty$.
By local existence of solutions of ODEs,
fix $\delta>0$ and $s:(-\delta,\delta)\to\R$ such that
$$s(0)\,=\,0\quad\qquad\hbox{and}\quad\qquad\forall t\in(-\delta,\delta),\quad s'(t)\,=\,G(s(t)).$$
Let $I:=(-\delta,\delta)\subseteq\R$.
We wish to show, for all $t\in I$, that $\beta(t)\in\Phi_\R^V(\sigma_0)$.

Define $\alpha:I\to\R^4$ by $\alpha(t)=\Phi_{s(t)}^V(\sigma_0)$.
For all $t\in I$, we have $\alpha(t)\in\Phi_\R^V(\sigma_0)$.
So it suffices to show that, for all $t\in I$, $\beta(t)=\alpha(t)$.
We have $\beta(0)=\Phi_0^W(\sigma_0)=\sigma_0=\Phi_0^V(\sigma_0)=\alpha(0)$.
Also, for all $t\in I$, we have $\beta'(t)=W(\beta(t))$.
Then, by uniqueness of solutions of ODEs,
it suffices to show, for all $t\in I$, that $\alpha'(t)=W(\alpha(t))$.
Fix $t_0\in I$.
We wish to prove that $\alpha'(t_0)=W(\alpha(t_0))$.

Let $q_0:=s(t_0)$.
Then $\alpha(t_0)=\Phi_{q_0}^V(\sigma_0)$.
Then $\alpha(t_0)\in\Phi_\R^V(\sigma_0)\subseteq M$.
Then $W(\alpha(t_0))=[f(\alpha(t_0))][V(\alpha(t_0))]$.
By the Chain Rule, we have
$$\alpha'(t_0)\quad=\quad[\,\,s'(t_0)\,\,]\,\,[\,\,(d/dq)_{q=q_0}\,(\,\Phi_q^V(\sigma_0)\,)\,\,].$$
We have $s'(t_0)=G(s(t_0))=G(q_0)=f(\Phi_{q_0}^V(\sigma_0))=f(\alpha(t_0))$.
Also,
$$(d/dq)_{q=q_0}\,(\,\Phi_q^V(\sigma_0)\,)\quad=\quad V(\Phi_{q_0}^V(\sigma_0))\quad=\quad V(\alpha(t_0)).$$
Then $\alpha'(t_0)=[f(\alpha(t_0))][V(\alpha(t_0))]=W(\alpha(t_0))$, as desired.
\end{proof}

\begin{lem}\wrlab{lem-homothetic-flows}
Let $d\ge1$ be an integer.
Let $V,W:\R^d\to\R^d$ be complete.
Assume, for all $\sigma\in\R^d$, that there exists $c\in\R$
such that $W(\sigma)=c\cdot(V(\sigma))$.
Then, for all $\sigma\in\R^d$, we have $\Phi_\R^W(\sigma)\subseteq\Phi_\R^V(\sigma)$.
\end{lem}

\begin{proof}
Fix $\sigma\in\R$. We will show that $\Phi_{[0,\infty)}^W(\sigma)\subseteq\Phi_\R^V(\sigma)$;
the proof that $\Phi_{(-\infty,0]}^W(\sigma)\subseteq\Phi_\R^V(\sigma)$ is similar.
Let $S:=\{t\in[0,\infty)\,|\,\Phi_t^W(\sigma)\in\Phi_\R^V(\sigma)\}$.
Assume, for a contradiction, that $S\subsetneq[0,\infty)$.

Because $\Phi_0^W(\sigma)=\sigma=\Phi_0^V(\sigma)\in\Phi_\R^V(\sigma)$, we see that $0\in S$.

Let $t_0:=\inf[0,\infty)\backslash S$.
Let $\sigma_0:=\Phi_{t_0}^W(\sigma)$.
By \lref{lem-homothetic-flows-local},
fix $\delta>0$ such that 
$\Phi_{(-\delta,\delta)}^W(\sigma_0)\subseteq\Phi_\R^V(\sigma_0)$.
Because $t_0=\inf[0,\infty)\backslash S$, we have
$$t_0\ge0,\qquad[0,t_0)\,\,\,\subseteq\,\,\,S\qquad\hbox{and}\qquad[t_0,t_0+\delta)\,\,\,\not\subseteq\,\,\,S.$$

Let $b:=\min\{\delta/2,t_0\}$.
Then $b\in[0,\delta)$, so $-b\in(-\delta,0]\subseteq(-\delta,\delta)$.

Either $t_0=0$ or $t_0>0$.
If $t_0=0$, then $b=0$, in which case $t_0-b=0\in S$.
If $t_0>0$, then $0<b\le t_0$, so $t_0-b\in[0,t_0)\subseteq S$.
In either case, $t_0-b\in S$.
That is, $\Phi_{t_0-b}^W(\sigma)\in\Phi_\R^V(\sigma)$.
Also,
$$\Phi_{t_0-b}^W(\sigma)\quad=\quad\Phi_{-b}^W(\sigma_0)\quad\in\quad
\Phi_{(-\delta,\delta)}^W(\sigma_0)\quad\subseteq\quad\Phi_\R^V(\sigma_0).$$
Then $\Phi_{t_0-b}^W(\sigma)\in(\Phi_\R^V(\sigma))\cap(\Phi_\R^V(\sigma_0))$,
so $\emptyset\ne(\Phi_\R^V(\sigma))\cap(\Phi_\R^V(\sigma_0))$,
so $\Phi_\R^V(\sigma)=\Phi_\R^V(\sigma_0)$.
Then
$$\Phi_{[t_0,t_0+\delta)}^W(\sigma)\,\,=\,\,\Phi_{[0,\delta)}^W(\sigma_0)\,\,\subseteq\,\,
\Phi_{(-\delta,\delta)}^W(\sigma_0)\,\,\subseteq\,\,\Phi_\R^V(\sigma_0)\,\,=\,\,\Phi_\R^V(\sigma),$$
so $[t_0,t_0+\delta)\subseteq S$, contradiction.
\end{proof}

\begin{cor}\wrlab{cor-homothetic-flows}
Let $d\ge1$ be an integer.
Let $V,W:\R^d\to\R^d$ be complete.
Assume, for all $\sigma\in\R^d$, that there exists $c\in\R\backslash\{0\}$
such that $W(\sigma)=c\cdot(V(\sigma))$.
Then, for all $\sigma\in\R^d$, we have $\Phi_\R^W(\sigma)=\Phi_\R^V(\sigma)$.
\end{cor}

\begin{proof}
By \lref{lem-homothetic-flows}, for all $\sigma\in\R^d$,
we have $\Phi_\R^W(\sigma)\subseteq\Phi_\R^V(\sigma)$;
we need to show that $\Phi_\R^V(\sigma)\subseteq\Phi_\R^W(\sigma)$.

For all $\sigma\in\R^d$, there exists $c\in\R\backslash\{0\}$ such that $V(\sigma)=(1/c)\cdot(W(\sigma))$.
So, interchanging $V$ and $W$ and replacing $c$ with $1/c$, \lref{lem-homothetic-flows} shows,
for all $\sigma\in\R^d$, that $\Phi_\R^V(\sigma)\subseteq\Phi_\R^W(\sigma)$,
as desired.
\end{proof}

\section{Coincidence of flows\wrlab{sect-gen-flow}}

\begin{lem}\wrlab{lem-orbits-agree}
Let $N$ be an interval in $\R$.
Assume that $0\in N$.
Let $d\ge1$ be an integer.
Let $V,W:\R^d\to\R^d$ be complete.
Let $\sigma\in\R^d$.
Assume, for all $t\in N$, that $W(\Phi_t^V(\sigma))=V(\Phi_t^V(\sigma))$.
Then, for all $t\in N$, we have $\Phi_t^W(\sigma)=\Phi_t^V(\sigma)$.
\end{lem}

\begin{proof}
Define $a,b:\R\to\R^d$ by
$a(t)=\Phi_t^V(\sigma)$ and $b(t)=\Phi_t^W(\sigma)$.
We wish to show, for all $t\in N$, that $b(t)=a(t)$.

We have $a(0)=\sigma=b(0)$.
Also, for all $t\in\R$, $b'(t)=W(b(t))$.
So, by uniqueness of solutions of ODEs,
it suffices to show, for all $t\in N$, that $a'(t)=W(a(t))$.

For all $t\in N$, we have
$W(a(t))=W(\Phi_t^V(\sigma))=V(\Phi_t^V(\sigma))=V(a(t))$;
then $a'(t)=V(a(t))=W(a(t))$, as desired.
\end{proof}

\begin{cor}\wrlab{cor-vect-agree-implies-flow-agree}
Let $N$ be an interval in $\R$.
Let $\barN$ denote the closure in $\R$ of $N$.
Assume that $0\in\barN$.
Let $d\ge1$ be an integer and let $A,B\subseteq\R^d$.
Let $V,W:\R^d\to\R^d$ be complete.
Assume $W=V$ on $B$.
Assume $\Phi_N^V(A)\subseteq B$.
Let $\barA$ denote the closure in $\R^d$ of $A$.
Then, for all $t\in\barN$, we have $\Phi_t^W=\Phi_t^V$ on $\barA$.
\end{cor}

\begin{proof}
Let $\sigma\in\barA$.
We wish to show, for all $t\in\barN$, that
$\Phi_t^W(\sigma)=\Phi_t^V(\sigma)$.

Let $\barB$ be the closure in $\R^d$ of $B$.
By assumption, $\Phi_N^V(A)\subseteq B$,
so, by continuity, $\Phi_\barN^V(\,\barA\,)\subseteq\barB$.
Also, since $W=V$ on $B$, continuity yields
$W=V$ on $\barB$.
For all $t\in\barN$, we have $\Phi_t^V(\sigma)\in\Phi_{\barN}^V(\,\barA\,)\subseteq\barB$.
Then, for all~$t\in\barN$, we have $W(\Phi_t^V(\sigma))=V(\Phi_t^V(\sigma))$.

Then, by \lref{lem-orbits-agree},
for all $t\in\barN$, we have $\Phi_t^W(\sigma)=\Phi_t^V(\sigma)$.
\end{proof}

Let $d\ge1$ be an integer and let $M:=\R^d$.
Let $k\ge0$ be an integer.
Let $\pi:=\pi_k^M:F_kM\to M$ be the $k$th order frame bundle of $M$.
Any diffeomorphism $g:M\to M$ induces a bundle diffeomorphism $F_kg:F_kM\to F_kM$.
A complete $X:M\to\R^d$ therefore induces, for all $s\in\R$, 
a bundle diffeomorphism $F_k\Phi_s^X:F_kM\to F_kM$.
For any two diffeomorphisms $g,h:M\to M$, for any $\lambda\in F_kM$, we have:
\begin{itemize}
\item[] [$g$ and $h$ agree to order $k$ at $\pi(\lambda)$] \, iff \, [$(F_kg)(\lambda)=(F_kh)(\lambda)$].
\end{itemize}
The tangent bundle of $F_kM$ is denoted $TF_kM$.
The bundle of $k$-jets of germs of vector fields on $M$ is denoted $J_kTM$.
Any $C^\infty$ function $X:M\to\R^d$ represents a vector field $M\to TM$,
which ``prolongs'' to a section $J_kX:M\to J_kTM$.
If $X:M\to\R^d$ is complete and if $\mu\in F_kM$,
then someone who knows both $\mu$ and
\begin{itemize}
\item[]the $k$-jet of the germ of $X$ at $\pi(\mu)$
\end{itemize}has enough information to compute
\begin{itemize}
\item[]the tangent vector at $0$ of the curve $s\mapsto(F_k\Phi_s^X)(\mu):\R\to F_kM$.
\end{itemize}
More succinctly, there is a bundle map
$$\Psi_d^k\qquad:\qquad F_kM\,\,\,\times_M\,\,\,J_kTM\qquad\longrightarrow\qquad TF_kM$$
such that, for any complete $X:M\to\R^d$, for any $\mu\in F_kM$,
$$\Psi_d^k\,\,(\,\,\mu\,,\,(J_kX)(\pi(\mu))\,\,)\quad=\quad(d/ds)_{s=0}\,\,[\,\,(F_k\Phi_s^X)(\mu)\,\,].$$

\begin{lem}\wrlab{lem-orbits-agree-to-all-orders}
Let $N$ be an interval in $\R$.
Assume that $0\in N$.
Let $d\ge1$ be an integer.
Let $V,W:\R^d\to\R^d$ be complete.
Let $\sigma\in\R^d$.
Assume, for all $t\in N$, that $W$ agrees with $V$ to all orders at $\Phi_t^V(\sigma)$.
Then, for all $t\in N$, $\Phi_t^W$ agrees with $\Phi_t^V$ to all orders at $\sigma$.
\end{lem}

\begin{proof}
Fix an integer $k\ge0$.
We wish to show, for all $t\in N$,
that $\Phi_t^W$~and~$\Phi_t^V$ agree at $\sigma$ to order $k$.

Let $M:=\R^d$.
Let $\pi:=\pi_k^M:F_kM\to M$ be the $k$th order frame bundle of $M$.
Fix $\sigma^*\in\pi^{-1}(\sigma)$.
We wish to show, for all $t\in N$, that
$(F_k\Phi_t^W)(\sigma^*)=(F_k\Phi_t^V)(\sigma^*)$.
For all $t\in\R$, let
$$\alpha_t:=F_k\Phi_t^V:F_kM\to F_kM\quad\hbox{and}\quad\beta_t:=F_k\Phi_t^W:F_kM\to F_kM.$$
We wish to show, for all $t\in N$, that $\beta_t(\sigma^*)=\alpha_t(\sigma^*)$.

Let $a,b:\R\to F_kM$ be defined by
$$a(t)=\alpha_t(\sigma^*)\qquad\hbox{and}\qquad b(t)=\beta_t(\sigma^*).$$
We wish to show, for all $t\in N$, that $b(t)=a(t)$. 

Let $V^*,W^*:F_kM\to TF_kM$ be the vector fields on $F_kM$ defined by
$$V^*(\rho)\,\,=\,\,(d/ds)_{s=0}\,[\,\alpha_s(\rho)\,],
\qquad
W^*(\rho)\,\,=\,\,(d/ds)_{s=0}\,[\,\beta_s(\rho)\,].$$
Then, for all $t\in\R$, we have
\begin{eqnarray*}
W^*(b(t))&=&(d/ds)_{s=0}[\beta_s(b(t))]\,=\,(d/ds)_{s=0}[\beta_s(\beta_t(\sigma^*))]\\
&=&(d/ds)_{s=0}[\beta_{t+s}(\sigma^*)]\,=\,(d/ds)_{s=0}[b(t+s)]=b'(t).
\end{eqnarray*}
Also, $a(0)=\alpha_0(\sigma^*)=\sigma^*=\beta_0(\sigma^*)=b(0)$.
So, by uniqueness of solutions of ODEs,
we wish to prove: for all $t\in N$, $a'(t)=W^*(a(t))$.
Fix $t_0\in N$ and let $\mu:=a(t_0)$.
We wish to prove $a'(t_0)=W^*(\mu)$.

We have $\mu=a(t_0)=\alpha_{t_0}(\sigma^*)$ and $\alpha_{t_0}=F_k\Phi_{t_0}^V$ and $\pi(\sigma^*)=\sigma$, so
$$\pi(\mu)\,=\,\pi(\alpha_{t_0}(\sigma^*))\,=\,\pi((F_k\Phi_{t_0}^V)(\sigma^*))\,=\,
\Phi_{t_0}^V(\pi(\sigma^*))\,=\,\Phi_{t_0}^V(\sigma).$$
Then, by assumption, $W$ and $V$ agree at $\pi(\mu)$ to all orders and, in particular, to order $k$.
Then $(J_kV)(\pi(\mu))=(J_kW)(\pi(\mu))$.
Also, by definition of $\Psi_d^k$, we have
\begin{eqnarray*}
\Psi_d^k\,\,(\,\,\mu\,,\,(J_kV)(\pi(\mu))\,\,)\,\,&=&\,\,(d/ds)_{s=0}\,\,[\,\,(F_k\Phi_s^V)(\mu)\,\,]\qquad\hbox{and}\\
\Psi_d^k\,\,(\,\,\mu\,,\,(J_kW)(\pi(\mu))\,\,)\,\,&=&\,\,(d/ds)_{s=0}\,\,[\,\,(F_k\Phi_s^W)(\mu)\,\,].
\end{eqnarray*}
Then
$(d/ds)_{s=0}[(F_k\Phi_s^V)(\mu)]=(d/ds)_{s=0}[(F_k\Phi_s^W)(\mu)]$.
For all $s\in\R$, $F_k\Phi_s^V=\alpha_s$ and $F_k\Phi_s^W=\beta_s$.
Then $(d/ds)_{s=0}[\alpha_s(\mu)]=(d/ds)_{s=0}[\beta_s(\mu)]$.

For all $s\in\R$, we have
$\alpha_s(\mu)=\alpha_s(\alpha_{t_0}(\sigma^*))=\alpha_{t_0+s}(\sigma^*)=a(t_0+s)$.
Therefore
$(d/ds)_{s=0}[\alpha_s(\mu)]=a'(t_0)$.
By definition of $W^*$, we have
$W^*(\mu)=(d/ds)_{s=0}[\beta_s(\mu)]$.

Then $a'(t_0)=(d/ds)_{s=0}[\alpha_s(\mu)]=(d/ds)_{s=0}[\beta_s(\mu)]=W^*(\mu)$.
\end{proof}

\begin{cor}\wrlab{cor-vanishing-to-periodic}
Let $d\ge1$ be an integer.
Let $\bfzero:\R^d\to\R^d$ be defined by:
for all $\tau\in\R^d$, $\bfzero(\tau)=(0,\ldots,0)\in\R^d$.
Let $V:\R^d\to\R^d$ be complete.
Let $\sigma\in\R^d$.
Assume that $V$ agrees with $\bfzero$ to all orders at $\sigma$.
Then, for all $t\in\R$, $\Phi_t^V$ agrees with $\Id_d$ to all orders at $\sigma$.
\end{cor}

\section{Results about invariance and local invariance\wrlab{sect-local-invariance}}

Recall, from \secref{sect-notation}, the definitions of
\begin{itemize}
\item[]{\bf$V$-invariant} \qquad\qquad and \qquad\qquad {\bf locally $V$-invariant}.
\end{itemize}

\begin{lem}\wrlab{lem-locally-invar-criterion}
Let $V,W:\R^d\to\R^d$ both be complete.
Let $U$ be an open subset of $\R^d$.
Assume that $V=W$ on $U$.
Let $S\subseteq\R^d$ be $V$-invariant.
Then $S\cap U$ is locally $W$-invariant.
\end{lem}

\begin{proof}
Fix $\sigma\in S\cap U$.
We wish to show that there is an open neighborhood $N$ in $\R$ of $0$
such that $\Phi_N^W(\sigma)\subseteq S\cap U$.

Define $f:\R\to\R^d$ by $f(r)=\Phi_r^W(\sigma)$.
We have $f(0)=\sigma\in U$, so $0\in f^{-1}(U)$.
By continuity of $f$, $f^{-1}(U)$ is open in $\R$,
so each connected component of $f^{-1}(U)$ is open in $\R$ as well.
Let $N$ be the connected component of $f^{-1}(U)$ satisfying $0\in N$.
Then
$$\Phi_N^W(\sigma)\quad=\quad f(N)\quad\subseteq\quad f(f^{-1}(U))\quad\subseteq\quad U.$$
It remains to show that $\Phi_N^W(\sigma)\subseteq S$.

Since $S$ is $V$-invariant, and since $\sigma\in S$,
it follows that $\Phi_\R^V(\sigma)\subseteq S$.
Since $W=V$ on $U$ and since $\Phi_N^W(\sigma)\subseteq U$,
it follows, from \lref{lem-orbits-agree},
that $\Phi_N^V(\sigma)=\Phi_N^W(\sigma)$.
Then $\Phi_N^W(\sigma)=\Phi_N^V(\sigma)\subseteq\Phi_\R^V(\sigma)\subseteq S$.
\end{proof}

\begin{lem}\wrlab{lem-nw-dense-saturation}
Let $W:\R^d\to\R^d$ be complete.
Let $Z\subseteq\R^d$ be a locally $W$-invariant, meager subset of $\R^d$.
Then $\Phi_\R^W(Z)$ is meager in~$\R^d$.
\end{lem}

\begin{proof}
Since $Z$ is meager in $\R^d$, it follows, for all $t\in\R$, that $\Phi^W_t(Z)$ is meager as well.
So, since $\displaystyle{\Phi_\Q^W(Z)=\bigcup_{q\in\Q}\,\left(\Phi^W_q(Z)\right)}$,
the set $\Phi_\Q^W(Z)$ is also meager in $\R^d$.
It therefore suffices to show that $\Phi_\R^W(Z)=\Phi_\Q^W(Z)$.

Since $\Q\subseteq\R$, it follows that $\Phi_\Q^W(Z)\subseteq\Phi_\R^W(Z)$,
and it remains to prove that $\Phi_\R^W(Z)\subseteq\Phi_\Q^W(Z)$.
Let $\zeta\in Z$ and let $t\in\R$. We wish to show that $\Phi_t^W(\zeta)\in\Phi_\Q^W(Z)$.

As $\zeta\in Z$ and $Z$ is locally $W$-invariant, fix an open neighborhood $N$ in $\R$ of~$0$
such that $\Phi_N^W(\zeta)\subseteq Z$.
By density of $\Q$ in $\R$, choose $q\in\Q\cap(t-N)$.
Choose $r\in N$ such that $q=t-r$. Then $t=q+r$.

Then $\Phi_t^W(\zeta)=\Phi_q^W(\Phi_r^W(\zeta))\in\Phi_\Q^W(\Phi_N^W(\zeta))\subseteq\Phi_\Q^W(Z)$.
\end{proof}

\section{Results about deterrence systems\wrlab{sect-gen-deter}}

\begin{lem}\wrlab{lem-V-V0-agreement-for-D}
Let $(V,I)\in\scrd$.
Then all of the following are true:
\begin{itemize}
\item[(i)]For all $t\in\R$, \quad $\Phi_t^V=\Phi_t^{V_0}$ on $[(\R^3)\backslash(I^3)]\times\R$.
\item[(ii)]For all $t\ge0$, \quad $\Phi_t^V=\Phi_t^{V_0}$ on $\R^3\times[a_I,\infty)$.
\item[(iii)]For all $t\le0$, \quad $\Phi_t^V=\Phi_t^{V_0}$ on $\R^3\times(-\infty,-a_I]$.
\end{itemize}
\end{lem}

\begin{proof}
By \lref{lem-vert-is-V0-invar}(i), for all $t\in\R$, we have
$$\Phi_t^{V_0}\,(\,[(\R^3)\backslash(I^3)]\times\R\,)\quad\subseteq\quad[(\R^3)\backslash(I^3)]\,\times\,\R.$$
Also, $V=V_0$ on $[(\R^3)\backslash(I^3)]\times\R$.
Then \lref{lem-orbits-agree} yields (i).

By \lref{lem-vert-is-V0-invar}(ii), for all $t\ge0$, we have
$$\Phi_t^{V_0}\,(\,\R^3\times[a_I,\infty)\,)\quad\subseteq\quad\R^3\,\times\,[a_I,\infty).$$
Also, $V=V_0$ on $\R^3\times[a_I,\infty)$.
Then \lref{lem-orbits-agree} yields (ii).

By \lref{lem-vert-is-V0-invar}(iii), for all $t\le0$, we have
$$\Phi_t^{V_0}\,(\,\R^3\times(-\infty,-a_I]\,)\quad\subseteq\quad\R^3\,\times\,(-\infty,-a_I].$$
Also, $V=V_0$ on $\R^3\times(-\infty,-a_I]$.
Then \lref{lem-orbits-agree} yields (iii).
\end{proof}

\begin{lem}\wrlab{lem-I3-times-R-invar}
Let $(V,I)\in\scrd$.
Then $I^3\times\R$ is $V$-invariant.
\end{lem}

\begin{proof}
Let $P:=[(\R^3)\backslash(I^3)]\times\R$.
By \lref{lem-vert-is-V0-invar}(i), $P$ is $V_0$-invariant
So, by \lref{lem-V-V0-agreement-for-D}(i), $P$ is $V$-invariant.
Then $(\R^4)\backslash P$ is $V$-invariant.
So, since $(\R^4)\backslash P=I^3\times\R$, we are done.
\end{proof}

\begin{lem}\wrlab{lem-undeterred-criterion}
Let $(V,I)\in\scrd$, $\rho\in\R^4$.
The following are equivalent:
\begin{itemize}
\item[(a)]$\rho\in\scru(V)$, \qquad {\it i.e.}, \qquad $\Pi_4(\Phi_\R^V(\rho))=\R$.
\item[(b)]$(-\infty,-a_I)\,\cap\,[\Pi_4(\Phi_\R^V(\rho))]
\,\,\,\ne\,\,\,\emptyset\,\,\,\ne\,\,\,
(a_I,\infty)\,\cap\,[\Pi_4(\Phi_\R^V(\rho))]$.
\item[(c)]$\displaystyle{\lim_{t\to\infty}\,[\Pi_4(\Phi_t^V(\rho))]=\infty}$
\qquad and \qquad
$\displaystyle{\lim_{t\to-\infty}\,[\Pi_4(\Phi_t^V(\rho))]=-\infty}$.
\end{itemize}
\end{lem}

\begin{proof}
The implication (a\,$\Rightarrow$b) is immediate.

{\it Proof of (b\,$\Rightarrow$c):}
Let $s_0,u_0\in\R$ and
assume $\Pi_4(\Phi_{s_0}^V(\rho))<-a_I$ and $\Pi_4(\Phi_{u_0}^V(\rho))>a_I$.
We wish to show
$$\hbox{that}\quad
\lim_{t\to\infty}\Pi_4(\Phi_t^V(\rho))=\infty
\quad\hbox{and that}\quad
\lim_{t\to-\infty}\Pi_4(\Phi_t^V(\rho))=-\infty.$$
We will prove the former; the latter is similar.
As $\displaystyle{\lim_{t\to\infty}(a_I+t-u_0)=\infty}$,
it suffices to show, for all $t\ge u_0$,
that $\Pi_4(\Phi_t^V(\rho))>a_I+t-u_0$.
So fix $t\ge u_0$.
We wish to prove that $\Pi_4(\Phi_t^V(\rho))>a_I+t-u_0$.

Let $\tau:=\Phi_{u_0}^V(\rho)$.
Then $\Pi_4(\tau)=\Pi_4(\Phi_{u_0}^V(\rho))>a_I$, so $\tau\in\R^3\times(a_I,\infty)$.
Then, as $t-u_0\ge0$, by \lref{lem-V-V0-agreement-for-D}(ii),
we have $\Phi_{t-u_0}^V(\tau)=\Phi_{t-u_0}^{V_0}(\tau)$.

Then $\Phi_t^V(\rho)=\Phi_{t-u_0}^V(\Phi_{u_0}^V(\rho))=\Phi_{t-u_0}^V(\tau)=\Phi_{t-u_0}^{V_0}(\tau)$.
By \lref{lem-Pi4-effect-of-V0},
$\Pi_4(\Phi_{t-u_0}^{V_0}(\tau))=(\Pi_4(\tau))+t-u_0$,
so, because $\Phi_t^V(\rho)=\Phi_{t-u_0}^{V_0}(\tau)$ and $\Pi_4(\tau)>a_I$,
we get $\Pi_4(\Phi_t^V(\tau))>a_I+t-u_0$.
{\it End of proof of (b\,$\Rightarrow$c).}

The Intermediate Value Theorem yields (c\,$\Rightarrow$a).
\end{proof}

\begin{cor}\wrlab{cor-U-open}
Let $(V,I)\in\scrd$.
Give $B(I)$ and $B_\circ(I)$ their relative topologies inherited from $\R^4$.
Then
\begin{itemize}
\item[(i)]$\scru(V)$ is open in $\R^4$,
\item[(ii)]$\scru_B(V,I)$ is open in $B(I)$ \qquad and
\item[(iii)]$\scru_B^\circ(V,I)$ is open in $B_\circ(I)$.
\end{itemize}
\end{cor}

\begin{proof}
Since (ii) and (iii) follow from (i), we need only prove (i).

Let $\sigma\in\scru(V)$.
We wish to show that there is an open subset $U$ of~$\R^4$
such that $\sigma\in U\subseteq\scru(V)$.

Since $\sigma\in\scru(V)$, we get $\Pi_4(\Phi_\R^V(\sigma))=\R$.
Fix $s_0,u_0\in\R$ such that
$$\Pi_4(\Phi_{s_0}^V(\sigma))<-a_I
\qquad\hbox{and}\qquad
\Pi_4(\Phi_{u_0}^V(\sigma))>a_I.$$
Let
$U:=\{\rho\in\R^4\,|\,\Pi_4(\Phi_{s_0}^V(\rho))<-a_I\hbox{ and }\Pi_4(\Phi_{u_0}^V(\rho))>a_I\}$.
By \lref{lem-undeterred-criterion}(b\,$\Rightarrow$a),
we have $U\subseteq\scru(V)$.
Then $U$~is an open subset of~$\R^4$ and $\sigma\in U\subseteq\scru(V)$, as desired
\end{proof}

\begin{cor}\wrlab{cor-porous-dense-comeager}
Let $(V,I)\in\scrd$.
Then
\begin{itemize}
\item[]$V$ is porous \qquad iff \qquad $\scru(V)$ is comeager in $\R^4$.
\end{itemize}
\end{cor}

\begin{proof}
By definition, $V$ is porous iff $\scru(V)$ is dense in $\R^4$.
By the Baire Category Theorem, an open subset of $\R^4$ is dense iff it is comeager.
Thus this result follows from \cref{cor-U-open}(i).
\end{proof}

\begin{lem}\wrlab{lem-at-most-one-time-through}
Let $(V,I)\in\scrd$.
Let $\sigma\in\scru(V)$.
Define $\psi:\R\to\R^4$ by $\psi(s)=\Phi_s^V(\sigma)$.
Let $\psi(r),\psi(t)\in I^4$.
Say $r\le t$.
Then $\psi([r,t])\subseteq I^4$.
\end{lem}

\begin{proof}
Let $s_0\in[r,t]$
and assume, for a contradiction, that $\psi(s_0)\notin I^4$.

Let $\rho:=\psi(s_0)=\Phi_{s_0}^V(\sigma)$.
Then $\rho\notin I^4$.
By \lref{lem-either-up-outside-or-down-outside},
either
\begin{itemize}
\item[(i)]$\Phi_{[0,\infty)}^{V_0}(\rho)\subseteq(\R^4)\backslash(I^4)$ \qquad or
\item[(ii)]$\Phi_{(-\infty,0]}^{V_0}(\rho)\subseteq(\R^4)\backslash(I^4)$.
\end{itemize}
Since $(V,I)\in\scrd$, we have $V=V_0$ on $(\R^4)\backslash(I^4)$.

Assume (i).
Then, by \lref{lem-orbits-agree}, we have
$\Phi_{t-s_0}^V(\rho)=\Phi_{t-s_0}^{V_0}(\rho)$.
Then $\psi(t)=\Phi_t^V(\sigma)=\Phi_{t-s_0}^V(\rho)=\Phi_{t-s_0}^{V_0}(\rho)
\in\Phi_{[0,\infty)}^{V_0}(\rho)\subseteq(\R^4)\backslash(I^4)$.
However, $\psi(t)\in I^4$, contradiction.

Assume (ii).
Then, by \lref{lem-orbits-agree}, we have
$\Phi_{r-s_0}^V(\rho)=\Phi_{r-s_0}^{V_0}(\rho)$.
Then $\psi(r)=\Phi_r^V(\sigma)=\Phi_{r-s_0}^V(\rho)=\Phi_{r-s_0}^{V_0}(\rho)
\in\Phi_{(-\infty,0]}^{V_0}(\rho)\subseteq(\R^4)\backslash(I^4)$.
However, $\psi(r)\in I^4$, contradiction.
\end{proof}

\begin{lem}\wrlab{lem-endpts-in-I4-implies-coincidence}
Let $(V,I)\in\scrd$.
Let $\sigma\in I^4$, $s_0\in\R$, $W\in\scrc$.
Assume $\Phi_{s_0}^V(\sigma)\in I^4$.
Assume $W=V$ on $I^4$.
Then $\Phi_{s_0}^W(\sigma)=\Phi_{s_0}^V(\sigma)$.
\end{lem}

\begin{proof}
Let $r:=\min\{0,s_0\}$, $t:=\max\{0,s_0\}$.
Then $\{0,s_0\}=\{r,t\}$.

Define $\psi:\R\to\R^4$ by $\psi(s)=\Phi_s^V(\sigma)$.
Then $\psi(0)=\sigma\in I^4$
and $\psi(s_0)=\Phi_{s_0}^V(\sigma)\in I^4$.
Then $\psi(0),\psi(s_0)\in I^4$, so $\psi(r),\psi(t)\in I^4$.
So, by \lref{lem-at-most-one-time-through},
we have $\psi([r,t])\subseteq I^4$, {\it i.e.},
we have $\Phi_{[r,t]}^V(\sigma)\subseteq I^4$.
So, because $s_0\in\{0,s_0\}=\{r,t\}\subseteq[r,t]$,
and because $W=V$ on $I^4$, by~\lref{lem-orbits-agree},
we conclude that $\Phi_{s_0}^W(\sigma)=\Phi_{s_0}^V(\sigma)$,
as desired.
\end{proof}

\begin{lem}\wrlab{lem-single-crossing}
Let $(V,I)\in\scrd$.
Let $\sigma\in\R^4$.
Let $b\in\R\backslash I$.
Then there is at most one $t_0\in\R$ such that $\Pi_4(\Phi_{t_0}^V(\sigma))=b$.
\end{lem}

\begin{proof}
Let $t_0,t'_0\in\R$.
Assume $t_0\le t'_0$ and
$\Pi_4(\Phi_{t_0}^V(\sigma))=b=\Pi_4(\Phi_{t'_0}^V(\sigma))$.
We wish to show that $t'_0=t_0$.

Since $b\notin I=(-a_I,a_I)$,
we know either that $b\le-a_I$ or that $b\ge a_I$.
We will assume that $b\ge a_I$;
the proof in the other case is similar.

Let $\sigma_0:=\Phi_{t_0}^V(\sigma)$.
Then $\Pi_4(\sigma_0)=b$.
Let $t_1:=t'_0-t_0$.
Then $\Phi_{t_1}^V(\sigma_0)=\Phi_{t'_0}^V(\sigma)$,
so $\Pi_4(\Phi_{t_1}^V(\sigma_0))=b$.
We have $\Pi_4(\sigma_0)=b\ge a_I$, {\it i.e.}, $\sigma_0\in\R^3\times[a_I,\infty)$.
So, since $t_1\ge0$, by \lref{lem-V-V0-agreement-for-D}(ii),
we conclude that $\Phi_{t_1}^V(\sigma_0)=\Phi_{t_1}^{V_0}(\sigma_0)$.
Then $b=\Pi_4(\Phi_{t_1}^V(\sigma_0))=\Pi_4(\Phi_{t_1}^{V_0}(\sigma_0))$,
so, by \lref{lem-Pi4-effect-of-V0},
$b=(\Pi_4(\sigma_0))+t_1$.
Then $b=b+t_1$, so $t_1=0$, so $t'_0=t_0$.
\end{proof}

\begin{lem}\wrlab{lem-defn-uf-df}
Let $(V,I)\in\scrd$.
Let $\sigma\in\barscru(V,I)$.
Then $(\Phi_\R^V(\sigma))\cap(T(I))$ and $(\Phi_\R^V(\sigma))\cap(B(I))$
both have exactly one element.
\end{lem}

\begin{proof}
We will show that $(\Phi_\R^V(\sigma))\cap(B(I))$
has exactly one element;
the proof for $(\Phi_\R^V(\sigma))\cap(T(I))$ is similar.
Let $b:=-a_I$.
Then $b\notin I$ and $B(I)=\barI^3\times\{b\}$.
Since $\sigma\in\barscru(V,I)\subseteq\scru(V)$, $\Pi_4(\Phi_\R^V(\sigma))=\R$,
so fix $t_0\in\R$ such that $\Pi_4(\Phi_{t_0}^V(\sigma))=b$.
Let $\rho:=\Phi_{t_0}^V(\sigma)$.
Then $\rho\in\Phi_\R^V(\sigma)$.
Also, $\Pi_4(\rho)=b$, {\it i.e.}, $\rho\in\R^3\times\{b\}$.

By \lref{lem-I3-times-R-invar},
$I^3\times\R$ is $V$-invariant.
Then, by continuity, $\barI^3\times\R$ is also $V$-invariant.
So, since $\sigma\in\barscru(V,I)\subseteq\barI^4\subseteq\barI^3\times\R$,
we see that $\rho\in\barI^3\times\R$.
Then $\rho\in(\,\barI^3\times\R)\cap(\R^3\times\{b\})=\barI^3\times\{b\}=B(I)$.
Thus $\rho\in(\Phi_\R^V(\sigma))\cap(B(I))$,
so $(\Phi_\R^V(\sigma))\cap(B(I))$ has at least one element.
Assume $\tau\in(\Phi_\R^V(\sigma))\cap(B(I))$.
We wish to show that $\tau=\rho$.

Choose $t_1\in\R$ such that $\tau=\Phi_{t_1}^V(\sigma)$.
We have $\rho,\tau\in B(I)\subseteq\R^3\times\{b\}$.
So $\Pi_4(\rho)=b=\Pi_4(\tau)$.
That is, $\Pi_4(\Phi_{t_1}^V(\sigma))=b=\Pi_4(\Phi_{t_0}^V(\sigma))$.
So, by \lref{lem-single-crossing}, $t_1=t_0$.
Then $\tau=\Phi_{t_1}^V(\sigma)=\Phi_{t_0}^V(\sigma)=\rho$.
\end{proof}

\begin{lem}\wrlab{lem-single-crossing-refined}
Let $(V,I)\in\scrd$.
Let $\sigma\in\R^4$.
Let $b\in\R\backslash I$.
Let $t_0\in\R$.
Assume $\Pi_4(\Phi_{t_0}^V(\sigma))=b$.
Then both of the following are true:
\begin{itemize}
\item[(i)]For all $t>t_0$, we have $\Pi_4(\Phi_t^V(\sigma))>b$.
\item[(ii)]For all $t<t_0$, we have $\Pi_4(\Phi_t^V(\sigma))<b$.
\end{itemize}
\end{lem}

\begin{proof}
We prove only (i); the proof of (ii) is similar.
Define $f:\R\to\R$ by $f(t)=\Pi_4(\Phi_t^V(\sigma))$.
We wish to show that $f((t_0,\infty))\subseteq(b,\infty)$.

We have $f(t_0)=b$.
By \lref{lem-single-crossing},
we get $f(\R\backslash\{t_0\})\subseteq\R\backslash\{b\}$.
Then $f((t_0,\infty))\subseteq f(\R\backslash\{t_0\})\subseteq\R\backslash\{b\}=(-\infty,b)\cup(b,\infty)$.
Then, because $f((t_0,\infty))$ is connected, we have
$$\hbox{either}\qquad f((t_0,\infty))\subseteq(-\infty,b)\qquad
\hbox{or}\qquad f((t_0,\infty))\subseteq(b,\infty).$$
It therefore suffices to prove $f((t_0,\infty))\not\subseteq(-\infty,b)$.

Let $\tau:=\Phi_{t_0}^V(\sigma)$.
Then $\Pi_4(\tau)=f(t_0)=b\notin I$.
It follows that $\tau\in(\R^4)\backslash(I^4)$,
and so $V(\tau)=V_0(\tau)$.
Thus $V(\tau)=(0,0,0,1)$.
Then
$$(d/dt)_{t=t_0}[\Phi_t^V(\sigma)]\quad=\quad V(\Phi_{t_0}^V(\sigma))\quad=\quad V(\tau)\quad=\quad(0,0,0,1).$$
We have $f'(t_0)=(d/dt)_{t=t_0}[\Pi_4(\Phi_t^V(\sigma))]=\Pi_4((d/dt)_{t=t_0}[\Phi_t^V(\sigma)])$.
It follows that $f'(t_0)=\Pi_4(0,0,0,1)=1$.

Since $f(t_0)=b$ and $f'(t_0)>0$, it follows, for some $t>t_0$, that $f(t)>b$.
That is, $f((t_0,\infty))\not\subseteq(-\infty,b)$.
\end{proof}

\begin{cor}\wrlab{cor-single-crossing-refined}
Let $(V,I)\in\scrd$.
Let $\sigma\in\R^4$.
Let $b\in\R\backslash I$.
Let $t,t_0\in\R$.
Assume $\Pi_4(\Phi_{t_0}^V(\sigma))=b$.
Then both of the following are true:
\begin{itemize}
\item[(i)]$t>t_0\qquad\Leftrightarrow\qquad\Pi_4(\Phi_t^V(\sigma))>b$
\item[(ii)]$t<t_0\qquad\Leftrightarrow\qquad\Pi_4(\Phi_t^V(\sigma))<b$.
\end{itemize}
\end{cor}

\begin{proof}
We prove (i);
the proof of (ii) is similar.
By (i) of \lref{lem-single-crossing-refined},
we have $\Rightarrow$ of (i).
We wish to prove $\Leftarrow$ of (i).
Assume $\Pi_4(\Phi_t^V(\sigma))>b$
and assume, for a contradiction, that $t\le t_0$.

Since $\Pi_4(\Phi_t^V(\sigma))\ne b=\Pi_4(\Phi_{t_0}^V(\sigma))$,
we get $t\ne t_0$.
Then $t<t_0$.
Then, by (ii) of \lref{lem-single-crossing-refined},
$\Pi_4(\Phi_t^V(\sigma))<b$, contradiction.
\end{proof}

\begin{lem}\wrlab{lem-invariance-of-halfspaces}
Let $(V,I)\in\scrd$.
Let $b\in\R\backslash I$.
Let $H:=\R^3\times(-\infty,b)$.
Then $\Phi_{(-\infty,0]}^V(H)\subseteq H$.
\end{lem}

\begin{proof}
Let $\tau\in H$, $t_1\le0$, $\rho:=\Phi_{t_1}^V(\tau)$.
We wish to prove that $\rho\in H$.

If $t_1=0$, then $\rho=\tau\in H$, and we are done, so we assume $t_1<0$.

We have $\tau\in H$, so $\Pi_4(\tau)<b$.
If $b\notin\Pi_4(\Phi_\R^V(\tau))$,
then, by the Intermediate Value Theorem,
we have $\Pi_4(\phi_\R^V(\tau))\subseteq(-\infty,b)$,
from which we get
$\rho\in\Phi_\R^V(\tau)\subseteq\Pi_4^{-1}((-\infty,b))=H$,
and we are done.
We therefore assume that $b\in\Pi_4(\Phi_\R^V(\tau))$.

Fix $s\in\R$ such that $b=\Pi_4(\Phi_s^V(\tau))$.
Let $\sigma:=\Phi_s^V(\tau)$.
Let $t_0:=0$.
Then $\Pi_4(\Phi_{t_0}^V(\sigma))=\Pi_4(\sigma)=\Pi_4(\Phi_s^V(\tau))=b$.
Also, $\Phi_{-s}^V(\sigma)=\tau$.

We have $\Pi_4(\Phi_{-s}^V(\sigma))=\Pi_4(\tau)<b$.
Then, by \cref{cor-single-crossing-refined}(ii), we get $-s\le t_0$.
So $-s+t_1\le t_0+t_1$.
As $t_1<0$, $t_0+t_1<t_0$.
Then $-s+t_1\le t_0+t_1<t_0$.
Then, by \cref{cor-single-crossing-refined}(i),
$\Pi_4(\Phi_{-s+t_1}^V(\sigma))<b$.
Then $\rho=\Phi_{t_1}^V(\tau)=\Phi_{t_1}^V(\Phi_{-s}^V(\sigma))=\Phi_{-s+t_1}^V(\sigma)\in\Pi_4^{-1}((-\infty,b))=H$.
\end{proof}

\begin{lem}\wrlab{lem-trap-per-orb}
Let $(V,I)\in\scrd$ and let $\sigma\in\R^4$.
Assume that $(V,\sigma)$ is periodic.
Then $\Phi_\R^V(\sigma)\subseteq I^4$.
\end{lem}

\begin{proof}
Fix $t\in\R$.
Let $\rho:=\Phi_t^V(\sigma)$.
We wish to show that $\rho\in I^4$.

As $(V,\sigma)$ is periodic,
fix an integer $n\ne0$ such that $\Phi_n^V(\sigma)=\sigma$.
Then
$$\Phi_n^V(\rho)\,\,=\,\,\Phi_n^V(\Phi_t^V(\sigma))\,\,=\,\,
\Phi_t^V(\Phi_n^V(\sigma))\,\,=\,\,\Phi_t^V(\sigma)\,\,=\,\,\rho.$$
Then $\rho=\Phi_{-n}^V(\Phi_n^V(\rho))=\Phi_{-n}^V(\rho)$.
Since $n\ne0\ne-n$, by \lref{lem-Pi4-effect-of-V0},
we have $\Pi_4(\Phi_n^{V_0}(\rho))\ne\Pi_4(\rho)\ne\Pi_4(\Phi_{-n}^{V_0}(\rho))$.
Then $\Phi_n^{V_0}(\rho)\ne\rho\ne\Phi_{-n}^{V_0}(\rho)$.

Then
$\Phi_n^V(\rho)=\rho\ne\Phi_n^{V_0}(\rho)$
and
$\Phi_{-n}^V(\rho)=\rho\ne\Phi_{-n}^{V_0}(\rho)$.
Then
$$\Phi_n^V(\rho)\ne\Phi_n^{V_0}(\rho),\qquad
\Phi_{|n|}^V(\rho)\ne\Phi_{|n|}^{V_0}(\rho),\qquad
\Phi_{-|n|}^V(\rho)\ne\Phi_{-|n|}^{V_0}(\rho).$$

Let $S_1:=I^3\times\R$.
If $\rho\notin S_1$, then, by \lref{lem-V-V0-agreement-for-D}(i), for all $t\in\R$,
$\Phi_t^V(\rho)=\Phi_t^{V_0}(\rho)$,
and, in particular, $\Phi_n^V(\rho)=\Phi_n^{V_0}(\rho)$, contradiction.
Thus $\rho\in S_1$.
Let $S_2:=\R^3\times(-\infty,a_I)$.
If $\rho\notin S_2$, then, by \lref{lem-V-V0-agreement-for-D}(ii), for all $t\ge0$,
$\Phi_t^V(\rho)=\Phi_t^{V_0}(\rho)$,
and, in particular, $\Phi_{|n|}^V(\rho)=\Phi_{|n|}^{V_0}(\rho)$, contradiction.
Thus $\rho\in S_2$.
Let $S_3:=\R^3\times(-a_I,\infty)$.
If $\rho\notin S_3$, then, by \lref{lem-V-V0-agreement-for-D}(iii), for all $t\le0$,
$\Phi_t^V(\rho)=\Phi_t^{V_0}(\rho)$,
and, in particular, $\Phi_{-|n|}^V(\rho)=\Phi_{-|n|}^{V_0}(\rho)$, contradiction.
Thus $\rho\in S_3$.

Then $\rho\in S_1\cap S_2\cap S_3=I^4$.
\end{proof}

\section{Downflow, upflow and timeflow\wrlab{sect-dtuflow}}

Let $(V,I)\in\scrd$.
By \cref{cor-U-open}(i), $\scru(V)$ is open in $\R^4$.
Also, as $I\in\scrc$, $I^4$ is open in $\R^4$.
Then, since $\scru_\circ(V,I)=(\scru(V))\cap(I^4)$, we see that
$\scru_\circ(V,I)$ is open in $\R^4$.
By \lref{lem-defn-uf-df}, for all $\sigma\in\barscru(V,I)$,
each of
$$(\Phi_\R^V(\sigma))\,\cap\,(B(I))\qquad\hbox{and}\qquad(\Phi_\R^V(\sigma))\,\cap\,(T(I))$$
has exactly one element.
The {\bf downflow map} of $(V,I)$ is the function
$DF_I^V:\barscru(V,I)\to B(I)$
which maps each $\sigma\in\barscru(V,I)$ to the unique element in the set
$(\Phi_\R^V(\sigma))\cap(B(I))$.
This map is a kind of projection, in the sense that,
for all $\sigma\in\scru_B(V,I)$, we have $DF_I^V(\sigma)=\sigma$.
The {\bf upflow map} of $(V,I)$ is the function
$UF_I^V:\barscru(V,I)\to T(I)$
which maps each $\sigma\in\barscru(V,I)$ to the unique element in the set
$(\Phi_\R^V(\sigma))\cap(T(I))$.

By the definitions of $DF_I^V$ and $UF_I^V$,
for each $\sigma\in\barscru(V,I)$,
there exists $t_\sigma\in\R$ such that
$\Phi_{t_\sigma}^V(DF_I^V(\sigma))=UF_I^V(\sigma)$;
moreover, by \lref{lem-single-crossing}, this $t_\sigma$ is unique.
By \cref{cor-single-crossing-refined}(i),
for all $\sigma\in\barscru(V,I)$, $t_\sigma>0$.
The {\bf timeflow map} of $(V,I)$ is the function
$TF_I^V:\barscru(V,I)\to(0,\infty)$
defined by $TF_I^V(\sigma)=t_\sigma$.
Then $TF_I^V$ is $C^0$ on $\barscru(V,I)$
and is $C^\infty$ on $\scru_\circ(V,I)$.

\begin{lem}\wrlab{lem-SU-and-Phi-V0}
Let $I\in\scri$.
Then
\begin{itemize}
\item[(i)]$\scru_B(V_0,I)=B(I)$ \qquad and
\item[(ii)]on $B(I)$, we have $SU_I=\Phi_{2a_I}^{V_0}=UF_I^{V_0}$.
\end{itemize}
\end{lem}

\begin{proof}
Since $\scru(V_0)=\R^4$, we have
$\scru_B(V_0,I)=(\R^4)\cap(B(I))=B(I)$.
Let $\rho\in B(I)$.
We wish to prove that
$SU_I(\rho)=\Phi_{2a_I}^{V_0}(\rho)=UF_I^{V_0}(\rho)$.

Choose $w,x,y\in\barI$ such that $\rho=(w,x,y,-a_I)$.
Then we have $SU_I(\rho)=(w,x,y,a_I)$.
Also, we have
$$\Phi_{2a_I}^{V_0}(\rho)\quad=\quad(w,x,y,-a_I+2a_I)\quad=\quad(w,x,y,a_I).$$
It therefore remains to prove that
$UF_I^{V_0}(\rho)=(w,x,y,a_I)$.

We have
$(w,x,y,a_I)=\Phi_{2a_I}^{V_0}(\rho)\in\Phi_\R^{V_0}(\rho)$
and $(w,x,y,a_I)\in T(I)$.
Then $(w,x,y,a_I)\in(\Phi_\R^{V_0}(\rho))\cap(T(I))$,
so $UF_I^{V_0}(\rho)=(w,x,y,a_I)$.
\end{proof}

\begin{lem}\wrlab{lem-undeterred-orbit}
Let $(V,I)\in\scrd$, let $\sigma\in\scru_B^\circ(V,I)$
and let $t_0:=TF_I^V(\sigma)$.
Then, for all $t\in\R$, we have:
\quad $[\,\Phi_t^V(\sigma)\in I^4\,]\,\,\Leftrightarrow\,\,[\,0<t<t_0\,]$.
\end{lem}

\begin{proof}
Define $\psi:\R\to\R^4$ by $\psi(t)=\Phi_t^V(\sigma)$.
We wish to prove that $\psi^{-1}(I^4)=(0,t_0)$.
Define $f:\R\to\R$ by $f(t)=\Pi_4(\psi(t))$.
Because $\sigma\in\scru_B^\circ(V,I)\subseteq B_\circ(I)\subseteq I^3\times\R$,
by \lref{lem-I3-times-R-invar}, we get
$\psi(\R)\subseteq I^3\times\R$.
Then $\psi^{-1}(I^4)=f^{-1}(I)$.
We wish to prove that $f^{-1}(I)=(0,t_0)$.

We have $\sigma\in B_\circ(I)\subseteq\R^3\times\{-a_I\}$,
so $\Pi_4(\sigma)=-a_I$.
Also, $\psi(0)=\sigma$.
Then $f(0)=\Pi_4(\psi(0))=\Pi_4(\sigma)=-a_I$.

We have
$\sigma\in\scru_B^\circ(V,I)\subseteq\scru_B(V,I)$,
so $\sigma=DF_I^V(\sigma)$.
Therefore, because $t_0=TF_I^V(\sigma)$,
we get $\Phi_{t_0}^V(\sigma)=UF_I^V(\sigma)$.
It follows that $\Phi_{t_0}^V(\sigma)\in T(I)\subseteq\R^3\times\{a_I\}$.
Then $f(t_0)=\Pi_4(\Phi_{t_0}^V(\sigma))=a_I$.

By \lref{lem-single-crossing},
$f^{-1}(-a_I)$ and $f^{-1}(a_I)$ both have at most one element.
So, since $0\in f^{-1}(-a_I)$ and since $t_0\in f^{-1}(a_I)$,
it follows that $f^{-1}(-a_I)=\{0\}$ and $f^{-1}(a_I)=\{t_0\}$.
We have $\sigma\in\scru_B^\circ(V,I)\subseteq\scru(V)$, so,
by \lref{lem-undeterred-criterion}(a\,$\Rightarrow$c), we have
$\displaystyle{\lim_{t\to\infty}\,[f(t)]=\infty}$ and
$\displaystyle{\lim_{t\to-\infty}\,[f(t)]=-\infty}$.
Then, by \lref{lem-single-preimages-give-interval-preimage},
we have $f^{-1}((-a_I,a_I))=(0,t_0)$.
That is, we have $f^{-1}(I)=(0,t_0)$, as desired.
\end{proof}

\begin{lem}\wrlab{lem-agree-inside-I4}
Let $(V,I)\in\scrd$ and let $W\in\scrc$.
Assume that $W=V$ on~$I^4$.
Let $\sigma\in\scru_B^\circ(V,I)$.
Let $t_0:=TF_I^V(\sigma)$.
Then, for all $t\in[0,t_0]$, we have $\Phi_t^W(\sigma)=\Phi_t^V(\sigma)$.
\end{lem}

\begin{proof}
By \lref{lem-undeterred-orbit},
for all $t\in(0,t_0)$, we have 
$\Phi_t^V(\sigma)\subseteq I^4$.
The result then follows from \cref{cor-vect-agree-implies-flow-agree}.
\end{proof}

\begin{lem}\wrlab{lem-UF-SU-agreement-orbit-agreement}
Let $(V,I)\in\scrd$.
Then the following are equivalent:
\begin{itemize}
\item[(a)]$UF_I^V=SU_I$ on $\scru_B(V,I)$.
\item[(b)]For all $\sigma\in\scru_B^\circ(V,I)$, \quad $(\Phi_\R^V(\sigma))\cup(I^4)\,=\,(\Phi_\R^{V_0}(\sigma))\cup(I^4)$.
\item[(c)]For all $\sigma\in(\scru(V))\backslash(I^4)$, \quad $(\Phi_\R^V(\sigma))\cup(I^4)\,=\,(\Phi_\R^{V_0}(\sigma))\cup(I^4)$.
\end{itemize}
\end{lem}

\begin{proof}
{\it Proof of (a\,$\Rightarrow$b):}
Assume that (a) is true.
Fix $\sigma\in\scru_B^\circ(V,I)$.
Let $S:=\Phi_\R^V(\sigma)$ and $S_0:=\Phi_\R^{V_0}(\sigma)$.
We wish to prove $S\cup(I^4)=S_0\cup(I^4)$.

We have $\sigma\in\scru_B^\circ(V,I)\subseteq\scru_B(V,I)$
and $\sigma\in\scru_B^\circ(V,I)\subseteq B(I)$.
Let
$$\tau:=SU_I(\sigma),\qquad t:=TF_I^V(\sigma),\qquad t_0:=TF_I^{V_0}(\sigma)$$
By (a), $\tau=UF_I^V(\sigma)$.
By \lref{lem-SU-and-Phi-V0}, we have $SU_I(\sigma)=UF_I^{V_0}(\sigma)$,
{\it i.e.}, $\tau=UF_I^{V_0}(\sigma)$.
Since $\sigma\in\scru_B(V,I)$, we get $\sigma=DF_I^V(\sigma)=DF_I^{V_0}(\sigma)$.

We have $\sigma=DF_I^V(\sigma)$ and $\tau=UF_I^V(\sigma)$
and $t=TF_I^V(\sigma)$.
It follows that $\tau=\Phi_t^V(\sigma)$.
We have $\sigma=DF_I^{V_0}(\sigma)$ and $\tau=UF_I^{V_0}(\sigma)$.
Moreover, we have $t_0=TF_I^{V_0}(\sigma)$.
It follows that $\tau=\Phi_{t_0}^{V_0}(\sigma)$.
Let
\begin{itemize}
\item[]$A:=\Phi_{(-\infty,0]}^V(\sigma)$,\quad$B:=\Phi_{(0,t)}^V(\sigma)$,\quad$C:=\Phi_{[t,\infty)}^V(\sigma)$,
\item[]$A_0:=\Phi_{(-\infty,0]}^{V_0}(\sigma)$,\quad$B_0:=\Phi_{(0,t_0)}^{V_0}(\sigma)$,\quad$C_0:=\Phi_{[t_0,\infty)}^{V_0}(\sigma)$.
\end{itemize}
By \lref{lem-undeterred-orbit}, we have $B\subseteq I^4$, so,
since $S=A\cup B\cup C$,
$$S\,\cup\,(I^4)\quad=\quad(A\,\cup\,B\,\cup\,C)\,\cup\,(I^4)\quad=\quad A\,\cup\,C\,\cup\,(I^4).$$
By \lref{lem-undeterred-orbit}, we have $B_0\subseteq I^4$, so,
since $S_0=A_0\cup B_0\cup C_0$,
$$S_0\,\cup\,(I^4)\quad=\quad(A_0\,\cup\,B_0\,\cup\,C_0)\,\cup\,(I^4)\quad=\quad A_0\,\cup\,C_0\,\cup\,(I^4).$$
It therefore suffices to show that $A=A_0$ and that $C=C_0$.

We have $\sigma\in B(I)\subseteq\R^3\times(-\infty,a_I]$.
Thus, by \lref{lem-V-V0-agreement-for-D}(iii),
we have $\Phi_{(-\infty,0]}^V(\sigma)=\Phi_{(-\infty,0]}^{V_0}(\sigma)$, {\it i.e.}, $A=A_0$.

We have $\tau=UF_I^V(\sigma)\in T(I)\subseteq\R^3\times[a_I,\infty)$.
Thus, by \lref{lem-V-V0-agreement-for-D}(ii),
we have $\Phi_{[0,\infty)}^V(\tau)=\Phi_{[0,\infty)}^{V_0}(\tau)$.
So, since $\tau=\Phi_t^V(\sigma)=\Phi_{t_0}^{V_0}(\sigma)$, this yields
$\Phi_{[t,\infty)}^V(\sigma)=\Phi_{[t_0,\infty)}^{V_0}(\sigma)$, {\it i.e.}, $C=C_0$.
{\it End of proof of (a\,$\Rightarrow$b).}

{\it Proof of (b\,$\Rightarrow$c):}
Assume that (b) is true.
Let $\sigma\in\scru(V)$ and assume that $\sigma\notin I^4$.
We wish to show that
$(\Phi_\R^V(\sigma))\cup(I^4)=(\Phi_\R^{V_0}(\sigma))\cup(I^4)$.

If $\sigma\notin I^3\times\R$,
{\it i.e.}, if $\sigma\in[(\R^3)\backslash(I^3)]\times\R$,
then, by \lref{lem-V-V0-agreement-for-D}(i),
we have $\Phi_\R^V(\sigma)=\Phi_\R^{V_0}(\sigma)$,
and so $(\Phi_\R^V(\sigma))\cup(I^4)=(\Phi_\R^{V_0}(\sigma))\cup(I^4)$,
and we are done.
We therefore assume that $\sigma\in I^3\times\R$.

Since $\sigma\in\scru(V)$, we have $\Pi_4(\Phi_\R^V(\sigma))=\R$.
Choose $t_0\in\R$ such that $\Pi_4(\Phi_{t_0}^V(\sigma))=-a_I$.
Let $\sigma_0:=\Phi_{t_0}^V(\sigma)$.
Then $\Pi_4(\sigma_0)=-a_I$.
That is, $\sigma_0\in\R^3\times\{-a_I\}$.
Since $\sigma\in I^3\times\R$,
by \lref{lem-I3-times-R-invar},
$\sigma_0\in I^3\times\R$.
Then
$$\sigma_0\,\,\,\in\,\,\,(I^3\times\R)\,\cap\,(\R^3\times\{-a_I\})\,\,\,=\,\,\, I^3\,\times\,\{-a_I\}\,\,\,=\,\,\, B_\circ(I).$$
Since $\sigma\in\scru(V)$ and $\sigma_0\in\Phi_\R^V(\sigma)$,
and since $\scru(V)$ is $V$-invariant,
it follows that $\sigma_0\in\scru(V)$.
Then $\sigma_0\in(\scru(V))\cap(B_\circ(I))=\scru_B^\circ(V,I)$.
By~(b), $(\Phi_\R^V(\sigma_0))\cup(I^4)=(\Phi_\R^{V_0}(\sigma_0))\cup(I^4)$.
It therefore suffices to show that
$\Phi_\R^V(\sigma_0)=\Phi_\R^V(\sigma)$ and that $\Phi_\R^{V_0}(\sigma_0)=\Phi_\R^{V_0}(\sigma)$.
Because we have $\sigma_0\in\Phi_\R^V(\sigma)$,
it follows that $\Phi_\R^V(\sigma_0)=\Phi_\R^V(\sigma)$.
It remains to show that $\Phi_\R^{V_0}(\sigma_0)=\Phi_\R^{V_0}(\sigma)$.

We have
$\sigma=\Phi_{-t_0}^V(\sigma_0)\in\Phi_\R^V(\sigma_0)\subseteq(\Phi_\R^V(\sigma_0))\cup(I^4)=(\Phi_\R^{V_0}(\sigma_0))\cup(I^4)$.
So, since $\sigma\notin I^4$,
it follows that $\sigma\in\Phi_\R^{V_0}(\sigma_0)$.
Then $\Phi_\R^{V_0}(\sigma_0)=\Phi_\R^{V_0}(\sigma)$.
{\it End of proof of (b\,$\Rightarrow$c).}

{\it Proof of (c\,$\Rightarrow$a):}
Assume that (c) is true.
Let $\sigma\in\scru_B(V,I)$.
We wish to show that $UF_I^V(\sigma)=SU_I(\sigma)$

We have $\sigma\in\scru_B(V,I)\subseteq B(I)$
and $(B(I))\cap(I^4)=\emptyset$.
Then $\sigma\notin I^4$.
So, since $\sigma\in\scru_B(V,I)\subseteq\scru(V)$,
we get $\sigma\in(\scru(V))\backslash(I^4)$.
By (c),
$$(\Phi_\R^V(\sigma))\,\cup\,(I^4)
\quad=\quad
(\Phi_\R^{V_0}(\sigma))\,\cup\,(I^4).$$
So, intersecting with $T(I)$, we get
$$[\,(\Phi_\R^V(\sigma))\cup(I^4)\,]\,\,\cap\,\,[\,T(I)\,]
\,\,=\,\,
[\,(\Phi_\R^{V_0}(\sigma))\cup(I^4)\,]\,\,\cap\,\,[\,T(I)\,].$$
So, because $I^4\cap(T(I))=\emptyset$,
we get
$$[\Phi_\R^V(\sigma)]\,\cap\,[T(I)]\quad=\quad[\Phi_\R^{V_0}(\sigma)]\,\cap\,[T(I)].$$
Thus $UF_I^V(\sigma)=UF_I^{V_0}(\sigma)$.
By \lref{lem-SU-and-Phi-V0},
$UF_I^{V_0}(\sigma)=SU_I(\sigma)$.
Then $UF_I^V(\sigma)=UF_I^{V_0}(\sigma)=SU_I(\sigma)$.
{\it End of proof of (c\,$\Rightarrow$a).}
\end{proof}

\begin{lem}\wrlab{lem-interior-dnflow}
Let $(V,I)\in\scrd$.
Then $DF_I^V(\scru_\circ(V,I))\subseteq B_\circ(I)$.
\end{lem}

\begin{proof}
Because $\scru_\circ(V,I)\subseteq I^4\subseteq I^3\times\R$,
by \lref{lem-I3-times-R-invar}, it follows that
$\Phi_\R^V(\scru_\circ(V,I))\subseteq I^3\times\R$.
Then
$DF_I^V(\scru_\circ(V,I))\subseteq\Phi_\R^V(\scru_\circ(V,I))\subseteq I^3\times\R$.

We have $\scru_\circ(V,I)\subseteq\barscru(V,I)$ and $DF_I^V(\,\barscru(V,I))\subseteq B(I)$.
Then
$$DF_I^V(\scru_\circ(V,I))\quad\subseteq\quad DF_I^V(\,\barscru(V,I))\quad\subseteq\quad B(I).$$
Then
$DF_I^V(\scru_\circ(V,I))\subseteq(I^3\times\R)\cap(B(I))=B_\circ(I)$.
\end{proof}

\section{Vector fields $V_\#^s$ on $\R$ with variable travel time\wrlab{sect-vertVF}}

For this section, let $I:=(-95,95)\subseteq\R$ and
define $W:\R\to\R$ by $W(y)=1-[1/2][\zeta_I(y)]$.
For all $y\in\barI$, we have $W(y)=1/2$.
Then $\Phi_{380}^W(-95)=95$.
For all $y\in\R$, we have $1/2\le W(y)\le1$.
For this section, fix $C>380$ such that $\Phi_C^W(-100)=100$.

For all $x\in\R$, define $V_*^x:\R\to\R$ by
$$V_*^x(y)\quad=\quad1\,\,-\,\,[1/2][\zeta_I(y)][1-(\zeta(x))];$$
then $(x,y)\mapsto V_*^x(y):\R\times\R\to\R$ is $C^\infty$.
Moreover,
\begin{itemize}
\item$\forall x\in\R$, \, $\forall y\in\R\backslash(-96,96)$, \qquad $V_*^x(y)=1$.
\item$\forall x,y\in\R$, \qquad\qquad\qquad\quad $1/2\le V_*^x(y)\le1$.
\end{itemize}
Moreover,
\begin{itemize}
\item[(i)]$\forall x\le0$, \, $\forall y\in\R$, \qquad $V_*^x(y)=1$ \qquad\qquad\qquad and
\item[(ii)]$\forall x\ge1$, \, $\forall y\in\R$, \qquad $V_*^x(y)=W(y)$.
\end{itemize}
By (i), for all $x\le0$, we have $\Phi_{200}^{V_*^x}(-100)=100$.
By (ii), for all $x\ge1$, we have $\Phi_C^{V_*^x}(-100)=\Phi_C^W(-100)=100$.

For all $y\in\R$, define $\beta_y:\R\to\R$ by $\beta_y(x)=V_*^x(y)$.
Then we have $\beta_y=1-[1/2][\zeta_I(y)][1-\zeta]$,
so $\beta_y':=[\zeta_I(y)][(\zeta')/2]$.
Then
\begin{itemize}
\item[(iii)]for all $y\in\R$, \qquad $\beta_y'\le0$ on~$(0,1)$ \qquad\qquad and
\item[(iv)]for all $y\in\barI$, \qquad $\beta_y'<0$ on~$(0,1)$.
\end{itemize}

For all $x\in\R$, let $\gamma(x)\in[200,C]$ be the unique real number
such that $\Phi_{\gamma(x)}^{V_*^x}(-100)=100$.
Define $F:\R^2\to\R$ by $F(x,t)=\Phi_t^{V_*^x}(-100)$.
Then, for all $x\in\R$, $\gamma(x)$ is defined implicitly by $F(x,\gamma(x))=100$.
Let $\partial_1,\partial_2$ be the standard framing of $\R^2$.
Then, for all $(x,t)\in\R^2$, we have $(\partial_2F)(x,t)=V_*^x(F(x,t))\ge1/2$.
That is, $\partial_2F\ge1/2$.
Then, by the Implicit Function Theorem,
$\gamma:\R\to\R$ is~$C^\infty$, and, moreover,
\begin{itemize}
\item[(v)]
$\forall x\in\R$, \quad\,\,
$\gamma'(x)\,\,\,=\,\,\,-\,\,[\,(\partial_1F)(x,\gamma(x))\,]\,\,/\,\,[\,(\partial_2F)(x,\gamma(x))\,]$.
\end{itemize}

From (iii) and (iv) above, we can prove:
\begin{itemize}
\item[(vi)]$\forall(x,t)\in(0,1)\times(5,\infty)$,\qquad $(\partial_1F)(x,t)<0$.
\end{itemize}
[Thanks to R.~Moeckel for showing me the following elementary proof of (vi):
Define $U:\R^2\to\R^2$ by $U(x,y)=(0,V_*^x(y))$.
Then $U$ is $C^\infty$ and bounded, and, therefore, complete.
The defining equation for $\Phi_t^U(x,y)$ is $(\partial/\partial t)(\Phi_t^U(x,y))=U(\Phi_t^U(x,y))$.
Differentiating this equation with respect to $x$ and $y$,
we obtain the defining equation for the induced flow in the tangent bundle $T\R^2$ of $\R^2$.
We identify $T\R^2$ with $\R^2\times\R^2$.
This induced flow then identifies with a flow on $\R^2\times\R^2$
that combines $\Phi_t^U$, operating on the first two coordinates,
with a time-dependent vector field on $\R^2$, whose flow operates on the last two.
For each $x,y,t\in\R$, that time-dependent vector field is linear
and is represented by the Jacobian
$J_t(x,y):=(DU)(\Phi_t^U(x,y)):\R^2\to\R^2$.
For all $x,y,t\in\R$,
$$(\,J_t(x,y)\,)\,(\,\{0\}\,\times\,\R\,)\qquad\subseteq\qquad\{0\}\,\times\,\R.$$
Geometrically, this says: For all $x,y,t\in\R$,
the vector field represented by $J_t(x,y)$,
when restricted to the line $\{0\}\times\R$,
consists of vectors that all vanish or point straight upward or point straight downward.
Using~(iii) above, we see that, on $0<x<1$, $y=-100$, $0\le t\le5$,
$$(\,J_t(x,y)\,)\,(\,(0,\infty)\,\times\,\{0\}\,)\qquad\subseteq\qquad\{0\}\,\times\,(-\infty,0].$$
Geometrically, this says: On $0<x<1$, $y=-100$, $0\le t\le5$,
the vector field represented by $J_t(x,y)$,
when restricted to the ray $(0,\infty)\times\{0\}$,
consists of vectors that all either vanish or point straight downward.
Using (iv) above, we see that, on $0<x<1$, $y=-100$, $t>5$,
$$(\,J_t(x,y)\,)\,(\,(0,\infty)\,\times\,\{0\}\,)\qquad\subseteq\qquad\{0\}\,\times\,(-\infty,0).$$
Geometrically, this says: On $0<x<1$, $y=-100$, $t>5$,
the vector field represented by $J_t(x,y)$,
when restricted to the ray $(0,\infty)\times\{0\}$,
consists of vectors that all point straight downward.
Let $E$ denote the restriction of the vector field $\partial_1$ to the segment $0<x<1$, $y=-100$.
The three geometric observations above show that, for time $t>5$,
any tangent vector in $E$ is constrained to flow to a vector pointing into
the open fourth quadrant $(0,\infty)\times(-\infty,0)$.
On $(x,t)\in(0,1)\times(5,\infty)$,
each value of $(\partial_1F)(x,t)$ is exactly the slope of such a vector.
A vector that points into the open fourth quadrant has negative slope.]

For all $x\in\R$, we have $\gamma(x)\ge200$.
Then, for all~$x\in(0,1)$, we have 
$(x,\gamma(x))\in(0,1)\times(5,\infty)$,
so, by (vi) above, $(\partial_1F)(x,\gamma(x))<0$.
Then, because $\partial_2F\ge1/2$, by (v) above,
we see, for all $x\in(0,1)$, that $\gamma'(x)>0$.
Also, $\gamma$ is $C^0$ on~$[0,1]$.
Then, by the Mean Value Theorem, $\gamma$ is increasing on $[0,1]$.
Also,
\begin{itemize}
\item[]$\gamma=200$ on $(-\infty,0]$ \qquad and \qquad $\gamma=C$ on $[1,\infty)$.
\end{itemize}
It follows, for all $s\in[200,C]$, that there is a unique $x_s\in[0,1]$ such that $\gamma(x_s)=s$.
By the Inverse Function Theorem, we have
\begin{itemize}
\item $s\mapsto x_s:[200,C]\to\R$ \,\, is $C^0$ \qquad and
\item$s\mapsto x_s:(200,C)\to\R$ \,\, is $C^\infty$.
\end{itemize}

Since $\gamma(0)=200$, it follows that $x_{200}=0$.
Recall that $C>380$.
For all $s\in[200,380]$, define $V_\#^s:\R\to\R$ by $V_\#^s(y):=V_*^{x_s}(y)$.
Then
\begin{itemize}
\item$(s,y)\,\,\mapsto\,\,V_\#^s(y)\,:\,[200,380]\,\,\times\,\,\R\to\R$ \,\, is $C^0$,
\item$(s,y)\,\,\mapsto\,\,V_\#^s(y)\,:\,(200,380)\,\,\times\,\,\R\to\R$ \,\, is $C^\infty$,
\item$\forall s\in[200,380]$, $\forall y\in\R$, \quad $1/2\le V_\#^s(y)\le1$,
\item$\forall s\in[200,380]$, $\forall y\in\R\backslash(-96,96)$, \quad $V_\#^s(y)=1$,
\item$\forall y\in\R$, \quad $V_\#^{200}(y)=1$ \qquad and
\item$\forall s\in[200,380]$, \quad $\Phi_s^{V_\#^s}(-100)=100$.
\end{itemize}

\section{Definitions of and results about $\scrd_+$, $\scrd^\#$, $\scrd_+^\#$ and $\scrd_*$\wrlab{sect-gettoD}}

Recall, from \secref{sect-dtuflow},
the definitions of $DF_I^V$, $UF_I^V$ and $TF_I^V$.
Let
\begin{eqnarray*}
\scrd_+&:=&\{(V,I)\in\scrd\,|\,UF_I^V=SU_I\hbox{ on }\scru_B(V,I)\},\\
\scrd^\#&:=&\{(V,I)\in\scrd\,|\,\xi_I\in\scru_B(V,I)\}\qquad\hbox{and}\\
\scrd_+^\#&:=&(\scrd_+)\,\cap\,(\scrd^\#).
\end{eqnarray*}
For all $(V,I)\in\scrd$, give $\scru_B^\circ(V,I)$ and $B_\circ(I)$
their relative topologies, inherited from $\R^4$;
then, by \cref{cor-U-open}(iii), $\scru_B^\circ(V,I)$ is open in $B_\circ(I)$,
so the open sets of $\scru_B^\circ(V,I)$ are all open in $B_\circ(I)$, as well.
Finally, let $\scrd_*$ denote the set of all $(V,I)\in\scrd_+^\#$ such that
\begin{itemize}
\item[] for some integer $j\ge1$ and some open neighborhood $N$ in~$\scru_B^\circ(V,I)$
of $\xi_I$, we have: \qquad $TF_I^V=j$ on $N$.
\end{itemize}

\begin{lem}\wrlab{lem-D0-stable}
Let $(V_1,I_1)\in\scrd$ and let $(V,I)\in\scrm(V_1,I_1)$.
Assume $V((\R^4)\backslash(I_1^4))\subseteq\{(0,0,0)\}\times(0,\infty)$.
Then all of the following are true:
\begin{itemize}
\item[(i)] For all $\sigma\in\R^4$, we have $\Phi_\R^{V_1}(\sigma)=\Phi_\R^V(\sigma)$.
\item[(ii)]$\scru(V_1)=\scru(V)$.
\item[(iii)]If $V_1$ is porous, then $V$ is porous.
\item[(iv)]If $(V_1,I_1)\in\scrd_+$, then $(V,I)\in\scrd_+$.
\item[(v)]If $(V_1,I_1)\in\scrd^\#$, then $(V,I)\in\scrd^\#$.
\end{itemize}
\end{lem}

\begin{proof}
Let $S:=\{(0,0,0)\}\times(0,\infty)\subseteq\R^4$.
Then $V((\R^4)\backslash(I_1^4))\subseteq S$.
Also, since $(V_1,I_1)\in\scrd$, we get
$V_1=V_0$ on $(\R^4)\backslash(I_1^4)$, so
$$V_1((\R^4)\backslash(I_1^4))\quad=\quad\{(0,0,0,1)\}\quad\subseteq\quad S.$$
For all $s,s'\in S$, there exists $c>0$ such that $s'=cs$.
Then, for all $\sigma\in(\R^4)\backslash(I_1^4)$,
there exists $c>0$ such that $V_1(\sigma)=c\cdot(V(\sigma))$.
Since $(V,I)\in\scrm(V_1,I_1)$,
it follows that $V=V_1$ on $I_1^4$.
Then, for all $\sigma\in\R^4$, there exists $c>0$ such that $V_1(\sigma)=c\cdot(V(\sigma))$.
Then \cref{cor-homothetic-flows} yields (i).
From (i), we see, for all $\sigma\in\R^4$, that
$$[\,\,\Pi_4(\Phi_\R^{V_1}(\sigma))\,=\,\R\,\,]\qquad\Leftrightarrow\qquad[\,\,\Pi_4(\Phi_\R^V(\sigma))\,=\,\R\,\,],$$
{\it i.e.}, that
$[\sigma\in\scru(V_1)]\Leftrightarrow[\sigma\in\scru(V)]$.
This proves (ii).
If $V_1$ is porous, then $\scru(V_1)$ is dense in $\R^4$,
so, by (ii), $\scru(V)$ is dense in~$\R^4$, which means that $V$ is porous.
This proves (iii).

{\it Proof of (iv):}
Assuming $(V_1,I_1)\in\scrd_+$, it follows that $UF_{I_1}^{V_1}=SU_{I_1}$ on~$\scru_B(V_1,I_1)$.
Then, by \lref{lem-UF-SU-agreement-orbit-agreement}(a\,$\Rightarrow$c), we have
\begin{itemize}
\item[]for all $\sigma\in(\scru(V_1))\backslash(I_1^4)$, \quad
$(\Phi_\R^{V_1}(\sigma))\cup(I_1^4)\,=\,(\Phi_\R^{V_0}(\sigma))\cup(I_1^4)$.
\end{itemize}
We have $(V,I)\in\scrm(V_1,I_1)$, so $I_1\subseteq I$.
Then
\begin{itemize}
\item[]for all $\sigma\in(\scru(V_1))\backslash(I^4)$, \quad
$(\Phi_\R^{V_1}(\sigma))\cup(I^4)\,=\,(\Phi_\R^{V_0}(\sigma))\cup(I^4)$.
\end{itemize}
So, by (i) and (ii), we conclude that
\begin{itemize}
\item[]for all $\sigma\in(\scru(V))\backslash(I^4)$, \quad
$(\Phi_\R^V(\sigma))\cup(I^4)\,=\,(\Phi_\R^{V_0}(\sigma))\cup(I^4)$.
\end{itemize}
Also, $(V,I)\in\scrm(V_1,I_1)\subseteq\scrd$.
Then, by \lref{lem-UF-SU-agreement-orbit-agreement}(c\,$\Rightarrow$a),
we see that $UF_I^V=SU_I$ on $\scru_B(V,I)$.
Then $(V,I)\in\scrd_+$.
{\it End of proof of (iv).}

{\it Proof of (v):}
Assuming $(V_1,I_1)\in\scrd^\#$, we have $\xi_{I_1}\in\scru_B(V_1,I_1)$.
Let $t_0:=a_I-a_{I_1}$.
By \lref{lem-V-V0-agreement-for-D}(iii),
$\Phi_{-t_0}^{V_1}(\xi_{I_1})=\Phi_{-t_0}^{V_0}(\xi_{I_1})$.
By \lref{lem-midpoint-connection}(i),
$\Phi_{-t_0}^{V_0}(\xi_{I_1})=\xi_I$.
Then $\Phi_{-t_0}^{V_1}(\xi_{I_1})=\Phi_{-t_0}^{V_0}(\xi_{I_1})=\xi_I$.
So, because $\xi_{I_1}\in\scru_B(V_1,I_1)\subseteq\scru(V_1)$,
and because $\scru(V_1)$ is $V_1$-invariant,
we see that $\xi_I\in\scru(V_1)$.
So, by (ii), we conclude that $\xi_I\in\scru(V)$.
Moreover, $\xi_I\in B(I)$
and $(V,I)\in\scrm(V_1,I_1)\subseteq\scrd$.
Then
$$\xi_I\quad\in\quad(\scru(V))\,\cap\,(B(I))\quad=\quad\scru_B(V,I).$$
Then $(V,I)\in\scrd^\#$, as desired.
{\it End of proof of~(v).}
\end{proof}

\begin{lem}\wrlab{lem-upflow-const-start}
Let $(W,J)\in\scrd^\#$.
Assume that $W$ is porous.
Then there exists $(V_1,I_1)\in(\scrm_*(W,J))\cap(\scrd_+^\#)$
such that $V_1$ is porous.
\end{lem}

\begin{proof}
Define $\lambda:\R^4\to\R^4$ by $\lambda(w,x,y,z)=(-w,-x,-y,z)$.
Let $c:=a_J+1$.
Define a reflection $R:\R^4\to\R^4$ by
$$R\,(\,w\,,\,x\,,\,y\,,\,z\,)\quad=\quad(\,\,w\,\,,\,\,x\,\,,\,\,y\,\,,\,\,2c-z\,\,).$$
For any $X:\R^4\to\R^4$, we define $\barX:\R^4\to\R^4$ by $\barX(\tau)=\lambda(X(R(\tau)))$.
For any $\gamma:\R\to\R^4$, we define $\bargamma:\R\to\R^4$ by $\bargamma(t)=R(\gamma(-t))$.
If $\gamma:\R\to\R^4$ and $X:\R^4\to\R^4$ are both $C^\infty$,
and if $\gamma$ is a flowline of~$X$,
then $\bargamma$ is a flowline of $\barX$.

Let $S:=\R^3\times[c-1,c+1]\subseteq\R^4$.
Then $S\subseteq(\R^4)\backslash(J^4)$ and $R(S)=S$.
We have $\lambda(0,0,0,1)=(0,0,0,1)$, so $\barV_0=V_0$.
Since $(W,J)\in\scrd^\#\subseteq\scrd$, we have $W=V_0$ on $(\R^4)\backslash(J^4)$.
Then $W=V_0$ on $S$, so $\barW=\barV_0$ on~$R(S)$, {\it i.e.}, on $S$.
Then, on $S$, we have $\barW=\barV_0=V_0=W$.

Let $H:=\R^3\times(-\infty,c+1)$.
Then $R(H)=\R^3\times(c-1,\infty)$.
Then $H\cup(R(H))=\R^4$ and $H\cap(R(H))\subseteq S$.
Define $V_1:\R^4\to\R^4$ by
\begin{itemize}
\item[]$V_1:=W$ on $H$ \qquad and \qquad $V_1:=\barW$ on $R(H)$.
\end{itemize}
Then $V_1:\R^4\to\R^4$ is $C^\infty$.
Moreover, $\barV_1=V_1$, so, for any flowline $\gamma:\R\to\R^4$ of $V_1$,
we see that $\bargamma:\R\to\R^4$ is also a flowline of $V_1$.

For all $\sigma\in\R^4$, if $\Pi_4(\sigma)=c$, then $R(\sigma)=\sigma$.
Therefore, for any map $\gamma:\R\to\R^4$, if $\Pi_4(\gamma(0))=c$, then $\bargamma(0)=R(\gamma(0))=\gamma(0)$.
Consequently, if $\gamma:\R\to\R^4$ is a flowline of $V_1$ and if $\Pi_4(\gamma(0))=c$,
then, by uniqueness of solutions of ODEs, we see that $\bargamma=\gamma$.

Since $(W,J)\in\scrd$, it follows that $W\in\scrc$, so $W(\R^4)\subseteq\barI_0^4$.
So, since $\lambda(\,\barI_0^4\,)=\barI_0^4$,
we see that $\barW(\R^4)\subseteq\barI_0^4$.
Then $V_1(\R^4)\subseteq\barI_0^4$.
Then $V_1\in\scrc$.

Let $I_1:=3J+3I_0$.
Then $a_{I_1}=3a_J+3a_{I_0}=3a_J+3=3c$.

As $W=V_0$ on $(\R^4)\backslash(J^4)$,
we see that $\barW=\barV_0=V_0$ on~$(\R^4)\backslash(R(J^4))$.
Then $V_1=V_0$ on $[\R^4]\backslash[(J^4)\cup(R(J^4))]$.
So, since $(J^4)\cup(R(J^4))\subseteq I_1^4$,
it follows that $V_1=V_0$ on $(\R^4)\backslash(I_1^4)$.
Then $(V_1,I_1)\in\scrd$.
As $J^4\subseteq H$, we get $V_1=W$ on $J^4$.
Also, $a_J<3a_J+3=a_{I_1}$.
Then $(V_1,I_1)\in\scrm(W,J)$.

Since $\R^3\times(-\infty,-a_J)\subseteq H$,
we see that $V_1=W$ on $\R^3\times(-\infty,-a_J)$.
As $(W,J)\in\scrd$, we get $W\in\scrv(a_J)$.
That is, $W=V_0$ on~$\R^3\times(-\infty,-a_J)$.
Therefore $V_1=W=V_0$ on~$\R^3\times(-\infty,-a_J)$,
and so we have $V_1\in\scrv(a_J)$.
So, because $(V_1,I_1)\in\scrm(W,J)$, we conclude that $(V_1,I_1)\in\scrm_*(W,J)$.
It remains to prove that $(V_1,I_1)\in\scrd_+^\#$
and that $V_1$ is porous.

{\it Claim 1:} Let $\tau\in H$ and let $t\in\R$.
Assume that $\Phi_t^W(\tau)\in H$.
Then $\Phi_t^{V_1}(\tau)=\Phi_t^W(\tau)$.
{\it Proof of Claim 1:}
Let $s:=\max\{0,t\}$.
Let $\sigma:=\Phi_s^W(\tau)$.
As $s\in\{0,t\}$, we get $\sigma\in\{\tau,\Phi_t^W(\tau)\}\subseteq H$.
We have $c+1=a_J+2>a_J$, so $c+1\in\R\backslash J$.
Also, recall that $(W,J)\in\scrd$.
Then, by \lref{lem-invariance-of-halfspaces},
$\Phi_{(-\infty,0]}^W(\sigma)\subseteq H$.
So, as $V_1=W$ on $H$,
we see, from \lref{lem-orbits-agree},
that $\Phi_{-|t|}^{V_1}(\sigma)=\Phi_{-|t|}^W(\sigma)$.
If $t\le0$, then $\sigma=\tau$ and $-|t|=t$,
and so $\Phi_t^{V_1}(\tau)=\Phi_{-|t|}^{V_1}(\sigma)=\Phi_{-|t|}^W(\sigma)=\Phi_t^W(\tau)$, as desired.
We therefore assume that $t>0$.
Then $\sigma=\Phi_t^W(\tau)$ and $|t|=t$.
Then
$$\Phi_{-t}^{V_1}(\sigma)\,=\,\Phi_{-|t|}^{V_1}(\sigma)\,=\,
\Phi_{-|t|}^W(\sigma)\,=\,\Phi_{-t}^W(\sigma)\,=\,
\Phi_{-t}^W(\Phi_t^W(\tau))\,=\,\tau.$$
Applying $\Phi_t^{V_1}$ to this yields $\sigma=\Phi_t^{V_1}(\tau)$.
That is, $\Phi_t^W(\tau)=\Phi_t^{V_1}(\tau)$, as desired.
{\it End of proof of Claim 1.}

{\it Claim 2:} $(\scru(W))\cap H\subseteq\scru(V_1)$.
{\it Proof of Claim 2:}
Fix $\tau\in\scru(W)$, and assume that $\tau\in H$.
We wish to show that $\tau\in\scru(V_1)$.

Since $\tau\in\scru(W)$, we have $\Pi_4(\Phi_\R^W(\tau))=\R$.
Fix $t_0,t_1\in\R$
such that $\Pi_4(\Phi_{t_0}^W(\tau))=c$
and such that $\Pi_4(\Phi_{t_1}^W(\tau))=-4c$.
Then
$$\Phi_{t_0}^W(\tau),\,\,\Phi_{t_1}^W(\tau)\quad\in\quad\Pi_4^{-1}((-\infty,c+1))\quad=\quad H.$$
Then, by Claim~1, we get
$\Phi_{t_0}^{V_1}(\tau)=\Phi_{t_0}^W(\tau)$
and $\Phi_{t_1}^{V_1}(\tau)=\Phi_{t_1}^W(\tau)$.
Then $\Pi_4(\Phi_{t_0}^{V_1}(\tau))=\Pi_4(\Phi_{t_0}^W(\tau))=c$
and $\Pi_4(\Phi_{t_1}^{V_1}(\tau))=\Pi_4(\Phi_{t_1}^W(\tau))=-4c$.

Define $\gamma_0:\R\to\R^4$ by $\gamma_0(t)=\Phi_{t_0+t}^{V_1}(\tau)$.
Then $\gamma_0(0)=\Phi_{t_0}^{V_1}(\tau)$ and $\gamma_0(t_1-t_0)=\Phi_{t_1}^{V_1}(\tau)$
and $\gamma_0(t_0-t_1)=\Phi_{2t_0-t_1}^{V_1}(\tau)$.
As $\gamma_0$ is a flowline of $V_1$ and
$\Pi_4(\gamma_0(0))=\Pi_4(\Phi_{t_0}^{V_1}(\tau))=c$,
it follows that $\gamma_0=\bargamma_0$.
Then $\gamma_0(t_0-t_1)=\bargamma_0(t_0-t_1)=R(\gamma_0(t_1-t_0))$.
So, by definition of $R$, we get
$\Pi_4(\gamma_0(t_0-t_1))=2c-[\Pi_4(\gamma_0(t_1-t_0))]$.
Also, we have
$$\Pi_4(\gamma_0(t_1-t_0))\,\,=\,\,\Pi_4(\Phi_{t_1}^{V_1}(\tau))\,\,=\,\,-4c.$$
Then
$\Pi_4(\Phi_{2t_0-t_1}^{V_1}(\tau))=\Pi_4(\gamma_0(t_0-t_1))=2c-[-4c]=6c$.
Thus
\begin{itemize}
\item$\Pi_4(\Phi_{2t_0-t_1}^{V_1}(\tau))\,\,=\,\,6c\,\,>\,\,3c\,\,=\,\,a_{I_1}$ \qquad\quad and
\item$\Pi_4(\Phi_{t_1}^{V_1}(\tau))\,\,=\,\,-4c\,\,<\,\,-3c\,\,=\,\,-a_{I_1}$.
\end{itemize}
So, since $(V_1,I_1)\in\scrd$,
by \lref{lem-undeterred-criterion}(b\,$\Rightarrow$a),
we get $\tau\in\scru(V_1)$, as desired.
{\it End of proof of Claim 2.}

Let $\scrf$ be the set of all flowlines $\gamma:\R\to\R^4$ of $V_1$ such that $\Pi_4(\gamma(0))=c$
and $\Pi_4(\gamma(\R))=\R$.
Then, for all $\gamma\in\scrf$, we have $\bargamma=\gamma$.
Let $\displaystyle{U:=\bigcup_{\gamma\in\scrf}\,[\gamma(\R)]}$ and $U_0:=\scru(V_1)$.

{\it Claim 3:} $U\subseteq U_0$.
{\it Proof of Claim 3:}
Fix $\gamma\in\scrf$ and $t_1\in\R$.
Let $\sigma_1:=\gamma(t_1)$.
We wish to show that $\sigma_1\in\scru(V_1)$.

Since $\gamma\in\scrf$,
we see that $\gamma$ is a flowline of $V_1$ and
that $\Pi_4(\gamma(\R))=\R$.
Let $\sigma_0:=\gamma(0)$.
Because $\gamma$ is a flowline of $V_1$,
it follows, for all $t\in\R$, that $\gamma(t)=\Phi_t^{V_1}(\sigma_0)$.
Then $\Phi_{t_1}^{V_1}(\sigma_0)=\gamma(t_1)=\sigma_1$.

Then, for all $t\in\R$, we have
$\Phi_t^{V_1}(\sigma_1)=\Phi_{t+t_1}^{V_1}(\sigma_0)=\gamma(t+t_1)$.
Then $\Phi_\R^{V_1}(\sigma_1)=\gamma(\R)$,
so $\Pi_4(\Phi_\R^{V_1}(\sigma_1))=\Pi_4(\gamma(\R))=\R$.
Then $\sigma_1\in\scru(V_1)$, as desired.
{\it End of proof of Claim 3.}

{\it Claim 4:} $U_0\subseteq U$.
{\it Proof of Claim 4:}
Fix $\sigma\in\scru(V_1)$.
We wish to show that there exists $\gamma\in\scrf$ such that $\sigma\in\gamma(\R)$.

Because $\sigma\in\scru(V_1)$,
it follows that $\Pi_4(\Phi_\R^{V_1}(\sigma))=\R$.
Choose $t_1\in\R$ such that $\Pi_4(\Phi_{t_1}^{V_1}(\sigma))=c$.
Define $\gamma:\R\to\R^4$ by $\gamma(t)=\Phi_{t+t_1}^{V_1}(\sigma)$.

Then $\gamma$ is a flowline of $V_1$ and $\Pi_4(\gamma(0))=\Pi_4(\Phi_{t_1}^{V_1}(\sigma))=c$.
Moreover, $\Pi_4(\gamma(\R))=\Pi_4(\Phi_\R^{V_1}(\sigma))=\R$.
Then $\gamma\in\scrf$.

Also, $\sigma=\Phi_0^{V_1}(\sigma)=\gamma(-t_1)\in\gamma(\R)$.
{\it End of proof of Claim~4.}

For any $\gamma\in\scrf$, we have $\bargamma=\gamma$, so $R(\gamma(\R))=\bargamma(\R)=\gamma(\R)$.
Then $\displaystyle{R(U)=\bigcup_{\gamma\in\scrf}\,[R(\gamma(\R))]=\bigcup_{\gamma\in\scrf}\,[\gamma(\R)]=U}$.
By Claim 3 and Claim 4, we have $U=U_0$.
Then $R(U_0)=R(U)=U=U_0$, so $U_0\cup(R(U_0))=U_0$.

Let $U_1:=(\scru(W))\cap H$.
By Claim 2, we have $U_1\subseteq U_0$.
Then $U_1\cup(R(U_1))\subseteq U_0\cup(R(U_0))=U_0$.
Since $W$ is porous, $\scru(W)$ is dense in $\R^4$.
Then, since $U_1=(\scru(W))\cap H$, it follows that $U_1$ is dense in $H$.
So, since $H\cup(R(H))=\R^4$,
we conclude that $U_1\cup(R(U_1))$ is dense in~$\R^4$.
So, since $U_1\cup(R(U_1))\subseteq U_0$,
we see that $U_0$ is dense in~$\R^4$.
So, as $U_0=\scru(V_1)$, $V_1$ is porous.
It remains to show that $(V_1,I_1)\in\scrd_+^\#$.

{\it Claim 5:} $(V_1,I_1)\in\scrd_+$.
{\it Proof of Claim 5:}
Recall that $(V_1,I_1)\in\scrd$.
Let $\rho\in\scru_B(V_1,I_1)$.
Let $\sigma:=SU_{I_1}(\rho)$.
We wish to prove
$UF_{I_1}^{V_1}(\rho)=\sigma$.
As $\rho\in\scru_B(V_1,I_1)\subseteq B(I_1)$,
we get $\sigma=SU_{I_1}(\rho)\in SU_{I_1}(B(I_1))=T(I_1)$.
So, as $\{UF_{I_1}^{V_1}(\rho)\}=(\Phi_\R^{V_1}(\rho))\cap(T(I_1))$,
we wish to prove $\sigma\in\Phi_\R^{V_1}(\rho)$.

We have $\rho\in\scru_B(V_1,I_1)\subseteq\scru(V_1)$,
so $\Pi_4(\Phi_\R^{V_1}(\rho))=\R$.
Fix $t_0\in\R$ such that $\Pi_4(\Phi_{t_0}^{V_1}(\rho))=c$.
Define $\gamma_1:\R\to\R^4$ by $\gamma_1(t)=\Phi_{t_0+t}^{V_1}(\rho)$.
Then $\gamma_1(0)=\Phi_{t_0}^{V_1}(\rho)$
and $\gamma_1(2c-t_0)=\Phi_{2c}^{V_1}(\rho)$.
As $\gamma_1$ is a flowline for $V_1$ and $\Pi_4(\gamma_1(0))=\Pi_4(\Phi_{t_0}^{V_1}(\rho))=c$,
we see that $\bargamma_1=\gamma_1$.
Then $\bargamma_1(\R)=\gamma_1(\R)$.

We have $\rho\in\scru_B(V_1,I_1)\subseteq B(I_1)\subseteq\R^3\times\{-a_{I_1}\}=\R^3\times\{-3c\}$.
Fix $w,x,y\in\R$ such that $\rho=(w,x,y,-3c)$.
Then $\Phi_{2c}^{V_0}(\rho)=(w,x,y,-c)$ and
$\sigma=SU_{I_1}(\rho)=(w,x,y,3c)=R(w,x,y,-c)$.
Moreover, because $\rho\in\R^3\times\{-3c\}$, we get
$\Phi_{[0,2c]}^{V_0}(\rho)\subseteq\R^3\times[-3c,-c]$.
Also, because $c=a_J+1$, it follows that
$\R^3\times[-3c,-c]\subseteq\R^3\times(-\infty,-a_J)$.
Therefore
$\Phi_{[0,2c]}^{V_0}(\rho)\subseteq\R^3\times[-3c,-c]\subseteq\R^3\times(-\infty,-a_J)$.
So, since $V_1=V_0$ on $\R^3\times(-\infty,-a_J)$, by \lref{lem-orbits-agree}, we get
$\Phi_{2c}^{V_1}(\rho)=\Phi_{2c}^{V_0}(\rho)$.
It follows that
$(w,x,y,-c)=\Phi_{2c}^{V_0}(\rho)=\Phi_{2c}^{V_1}(\rho)=\gamma_1(2c-t_0)$.
Then $\sigma=R(w,x,y,-c)=R(\gamma_1(2c-t_0))=\bargamma_1(t_0-2c)$,
and so we have $\sigma\in\bargamma_1(\R)=\gamma_1(\R)=\Phi_\R^{V_1}(\rho)$,
as desired.
{\it End of proof of Claim 5.}

Recall that $V_1=V_0$ on $\R^3\times(-\infty,-a_J)$.
Then, by continuity, $V_1=V_0$ on $\R^3\times(-\infty,-a_J]$.
As $(W,J)\in\scrd^\#$, we get $\xi_J\in\scru_B(W,J)\subseteq\scru(W)$.
Also, we have $\Pi_4(\xi_J)=-a_J<a_J+2=c+1$, so $\xi_J\in H$.
It follows that $\xi_J\in(\scru(W))\cap H$.
Then, by Claim 2, we conclude that $\xi_J\in\scru(V_1)$.
Let $a:=a_{I_1}-a_J$.
By \lref{lem-midpoint-connection}(ii), we have
$$\Phi_{[-a,0]}^{V_0}(\xi_J)\,\,\subseteq\,\,\{(0,0,0)\}\,\times\,[-a_{I_1},-a_J].$$
Then $\Phi_{[-a,0]}^{V_0}(\xi_J)\subseteq\R^3\times(-\infty,-a_J]$.
So, as $V_1=V_0$ on $\R^3\times(-\infty,-a_J]$,
by \lref{lem-orbits-agree},
we get $\Phi_{-a}^{V_1}(\xi_J)=\Phi_{-a}^{V_0}(\xi_J)$.
By \lref{lem-midpoint-connection}(i), we have $\Phi_{-a}^{V_0}(\xi_J)=\xi_{I_1}$.
Then $\Phi_{-a}^{V_1}(\xi_J)=\Phi_{-a}^{V_0}(\xi_J)=\xi_{I_1}$.
So, since $\xi_J\in\scru(V_1)$ and since $\scru(V_1)$ is $V_1$-invariant,
we conclude that $\xi_{I_1}\in\scru(V_1)$.
So, because $\xi_{I_1}\in B(I_1)$ and because $(V_1,I_1)\in\scrd$,
it follows that
$$\xi_{I_1}\qquad\in\qquad(\scru(V_1))\,\,\cap\,\,(B(I_1))\qquad=\qquad\scru_B(V_1,I_1).$$
Then $(V_1,I_1)\in\scrd^\#$.
So, by Claim 5,
$(V_1,I_1)\in(\scrd_+)\cap(\scrd^\#)=\scrd_+^\#$.
\end{proof}

\begin{lem}\wrlab{lem-upflow-const}
Let $(W,J)\in\scrd^\#$.
Assume that $W$ is porous.
Then there exists $(V_*,I_*)\in(\scrm_*(W,J))\cap(\scrd_*)$
such that $V_*$ is porous.
\end{lem}

\begin{proof}
By \lref{lem-upflow-const-start},
choose $(V_1,I_1)\in(\scrm_*(W,J))\cap(\scrd^\#_+)$ such that $V_1$ is porous.
Then $(V_1,I_1)\in\scrd_+^\#\subseteq\scrd^\#\subseteq\scrd$.
As $(V_1,I_1)\in\scrd^\#$,
we get $\xi_{I_1}\in\scru_B(V_1,I_1)$.
Then $\xi_{I_1}\in\scru_B(V_1,I_1)\subseteq\scru(V_1)$
and $\xi_{I_1}\in B_\circ(I_1)$, so
$\xi_{I_1}\in(\scru(V_1))\cap(B_\circ(I_1))=\scru_B^\circ(V_1,I_1)$.

Because $(V_1,I_1)\in\scrd$,
by \cref{cor-U-open}(iii),
$\scru_B^\circ(V_1,I_1)$ is an open subset of $B_\circ(I_1)$.
Define $\Psi:\scru_B^\circ(V_1,I_1)\to(0,\infty)$ by $\Psi(\tau)=TF_{I_1}^{V_1}(\tau)$.
By continuity of $\Psi$ at $\xi_{I_1}$,
fix an open neighborhood $N_1$
in $\scru_B^\circ(V_1,I_1)$ of~$\xi_{I_1}$ such that,
for all $\tau\in N_1$, $|(\Psi(\tau))-(\Psi(\xi_{I_1}))|<1$.
Since $\Psi(\xi_{I_1})>0$, fix an integer $m\ge1$ such that $|(\Psi(\xi_{I_1}))-m|<1$.
Then, by the Triangle Inequality, we have:
for all $\tau\in N_1$, $|(\Psi(\tau))-m|<2$.
Define $f:N_1\to\R$ by $f(\tau)=m-(\Psi(\tau))+250$.
Then $248<f<252$.

Fix an open neighborhood $N$ in~$B_\circ(I_1)$ of~$\xi_{I_1}$
such that the closure $\barN$ in $B_\circ(I_1)$ of $N$
is compact and satisfies $\barN\subseteq N_1$.
Fix an open neighborhood $N_0$ in~$B_\circ(I_1)$ of~$\xi_{I_1}$
such that the closure $\barN_0$ in $B_\circ(I_1)$ of~$N$
satisfies $\barN_0\subseteq N$.
Then $N_0\subseteq N_1$.
Fix $g:B_\circ(I_1)\to\R$ such that
\begin{itemize}
\item$(w,x,y)\mapsto g(w,x,y,-a_{I_1})\,:\,I_1^3\to\R$ \quad is $C^\infty$,
\item$0\le g\le1$ on $B_\circ(I_1)$,
\item$g=1$ on $\barN_0$ \qquad and
\item$g=0$ on $(B_\circ(I_1))\backslash N$.
\end{itemize}
Define $h:B_\circ(I_1)\to\R$ by:
\begin{itemize}
\item$\forall\tau\in N_1$, \quad $h(\tau):=[f(\tau)][g(\tau)]+[200][1-(g(\tau))]$ \qquad and
\item$\forall\tau\in(B_\circ(I_1))\backslash\barN$, \quad $h(\tau):=200$.
\end{itemize}
Then $(w,x,y)\mapsto h(w,x,y,-a_{I_1})\,:\,I_1^3\to\R$ is $C^\infty$, and
\begin{itemize}
\item[]$h=f$ on $\barN_0$ \qquad\quad and \qquad\quad $200\le h<252$.
\end{itemize}
For $s\in[200,380]$, define $V_\#^s$ as in \secref{sect-vertVF}.
For all $\tau=(w,x,y,z)\in I_1^3\times\R$,
\begin{itemize}
\item let $\tau_\#:=(w,x,y,-a_{I_1})\in B_\circ(I_1)$ \qquad and
\item let $v_\tau:=V_\#^{h(\tau_\#)}(z-a_{I_1}-101)$.
\end{itemize}
Then $\tau\mapsto v_\tau:I_1^3\times\R\to\R$ is $C^\infty$.
By the properties of $V_\#^s$,
for all $\tau\in I_1^3\times\R$, we have $1/2\le v_\tau\le1$, {\it i.e.}, $v_\tau\in[1/2,1]$.
Let
$$Q_1:=I_1^3\,\times\,(a_{I_1}+1,a_{I_1}+5),\quad Q_2:=I_1^3\,\times\,(a_{I_1}+197,a_{I_1}+201).$$
For all $\tau=(w,x,y,z)\in Q_1\cup Q_2$,
we have
$$z-a_{I_1}-101\quad\in\quad(-100,-96)\cup(96,100)\quad\subseteq\quad\R\backslash(-96,96),$$
so $v_\tau=1$.
Let $K:=\{(w,x,y)\in I_1^3\,|\,(w,x,y,-a_{I_1})\in\barN\,\}$.
Then, because $\barN$ is compact, it follows that $K$ is a compact subset of $I_1^3$.
Define $\hatI:=(a_{I_1}+1,a_{I_1}+201)$ and
$Q_3:=[(I_1^3)\backslash K]\times\hatI$.
Then, for all $\tau\in Q_3$, we have $\tau_\#\in(B_\circ(I_1))\backslash\barN$,
so $h(\tau_\#)=200$, so $v_\tau=1$.

Let $A:=I_1^3\times\hatI$
and let $Y:A\to\R^4$ be defined by $Y(\tau)=(0,0,0,v_\tau)$.
Then $Y:A\to\R^4$ is $C^\infty$ and
\begin{itemize}
\item$Y(A)\quad\subseteq\quad\{(0,0,0)\}\times[1/2,1]$ \qquad\qquad and
\item$Y=V_0$ \,\, on \,\, $Q_1\cup Q_2\cup Q_3$.
\end{itemize}
From the definitions of $A$ and $Q_1$ and $Q_2$ and $Q_3$, we have both
$$A\,\,\backslash\,\,(Q_1\,\cup\,Q_2)\qquad=\qquad I_1^3\quad\times\quad[\,\,a_{I_1}+5\,\,,\,\,a_{I_1}+197\,\,],$$
and
$A\backslash(Q_3)=K\times\hatI$.
Let $L:=K\times[a_{I_1}+5,a_{I_1}+197]$.
Then, as $K$ is compact, $L$ is a compact subset of~$A_*:=\R^3\times(a_{I_1},\infty)$.
Let $Q_4:=(A_*)\backslash A$ and $Q:=Q_1\cup Q_2\cup Q_3\cup Q_4$.
Then $(A_*)\backslash(Q_4)=A$, and so
$(A_*)\backslash Q=A\backslash(Q_1\cup Q_2\cup Q_3)=
[A\backslash(Q_1\cup Q_2)]\cap[\,A\,\backslash(Q_3)]=L$,
which is compact and is, therefore, closed in $A_*$.
Then $Q$ is open in $A_*$,
and it follows that $Q$ is open in~$\R^4$.

Since $(V_1,I_1)\in\scrd$, we see that $V_1=V_0$ on $(\R^4)\backslash(I_1^4)$.
So, because $Q_4\subseteq A_*\subseteq(\R^4)\backslash(I_1^4)$,
we see that $V_1=V_0$ on~$Q_4$.
Let $\partial A$ be the boundary in $\R^4$ of $A$.
As $A$ is open in $\R^4$, we get $(\partial A)\cap A=\emptyset$.
As $A\subseteq\R^3\times[a_{I_1}+1,a_{I_1}+201]$,
we get $\partial A\subseteq\R^3\times[a_{I_1}+1,a_{I_1}+201]$.
Then $\partial A\subseteq\R^3\times(a_{I_1},\infty))= A_*$.
Then $\partial A\subseteq(A_*)\backslash A=Q_4\subseteq Q$.

Define $V_*:\R^4\to\R^4$ by
\begin{itemize}
\item[]$V_*:=Y$ on $A$ \qquad and \qquad $V_*:=V_1$ on $(\R^4)\backslash A$.
\end{itemize}
We have $Y=V_0$ on $Q_1\cup Q_2\cup Q_3$,
and we have $Q_1\cup Q_2\cup Q_3\subseteq A$.
Then $V_*=Y=V_0$ on~$Q_1\cup Q_2\cup Q_3$.
Recall that $V_1=V_0$ on $Q_4$.
So, as $Q_4\subseteq(\R^4)\backslash A$,
we get $V_*=V_1=V_0$ on $Q_4$.
Then $V_*=V_0$ on $Q_1\cup Q_2\cup Q_3\cup Q_4$, {\it i.e.}, on $Q$.
Then $V_*$ is $C^\infty$ on~$Q$.
So, since $V_*$ is also $C^\infty$ on~$\R^4\backslash(\partial A)$,
and since $\partial A\subseteq Q$,
we see that $V_*$ is $C^\infty$ on~$\R^4$.
Since $(V_1,I_1)\in\scrd$, we get $V_1\in\scrc$,
so $V_1(\R^4)\subseteq\barI_0^4$.
So, since $Y(A)\subseteq\{(0,0,0)\}\times[1/2,1]\subseteq\barI_0^4$,
we get $V_*(\R^4)\subseteq\barI_0^4$.
Then $V_*\in\scrc$.

Let $I_*:=I_1+202I_0$.
By definition of $V_*$, we have $V_*=V_1$ on $(\R^4)\backslash A$.
Also, recall that $V_1=V_0$ on $(\R^4)\backslash(I_1^4)$.
Then, because $A\cup(I_1^4)\subseteq I_*^4$,
we see that $V_*=V_1=V_0$ on~$(\R^4)\backslash(I_*^4)$.
Thus $(V_*,I_*)\in\scrd$.
We have $I_1^4\subseteq(\R^4)\backslash A$, so $V_*=V_1$ on~$I_1^4$.
Therefore, since
$$a_{I_*}\quad=\quad a_{I_1}\,+\,202a_{I_0}\quad=\quad a_{I_1}\,+\,202\quad>\quad a_{I_1},$$
we get $(V_*,I_*)\in\scrm(V_1,I_1)$.
So, since $(V_1,I_1)\in\scrm_*(W,J)\subseteq\scrm(W,J)$,
it follows that $(V_*,I_*)\in\scrm(W,J)$.

We have $\R^3\times(-\infty,-a_J)\subseteq(\R^4)\backslash A$,
so $V_*=V_1$ on~$\R^3\times(-\infty,-a_J)$.
Because $(V_1,I_1)\in\scrm_*(W,J)$, we have $V_1\in\scrv(a_J)$.
That is, we have $V_1=V_0$ on~$\R^3\times(-\infty,-a_J)$.
Then $V_*=V_1=V_0$ on~$\R^3\times(-\infty,-a_J)$.
Then $V_*\in\scrv(a_J)$.
Then $(V_*,I_*)\in\scrm_*(W,J)$.

Since $V_1=V_0$ on $(\R^4)\backslash(I_1^4)$ and since $V_0(\R^4)=\{(0,0,0,1)\}$,
we get $V_1((\R^4)\backslash(I_1^4))=\{(0,0,0,1)\}$.
Then $V_1((\R^4)\backslash(I_1^4))=\{(0,0,0)\}\times(0,\infty)$.
Moreover, $Y(A)\subseteq\{(0,0,0)\}\times[1/2,1]\subseteq\{(0,0,0)\}\times(0,\infty)$.
Then $V_*((\R^4)\backslash(I_1^4))\subseteq\{(0,0,0)\}\times(0,\infty)$.
Because of this, and because
\begin{itemize}
\item$(V_*,I_*)\,\,\in\,\,\scrm(V_1,I_1)$,
\item$V_1$ is porous \qquad\qquad and
\item$(V_1,I_1)\,\,\in\,\,\scrd_+^\#\,\,=\,\,(\scrd_+)\,\cap\,(\scrd^\#)$,
\end{itemize}
we see, from \lref{lem-D0-stable}(ii-v), that
\begin{itemize}
\item$\scru(V_1)=\scru(V_*)$,
\item$V_*$ is porous \qquad\qquad and
\item$(V_*,I_*)\,\,\in\,\,(\scrd_+)\,\cap\,(\scrd^\#)\,\,=\,\,\scrd_+^\#$.
\end{itemize}
It remains to show,
for some integer $j\ge1$ and some open neighborhood $N_*$ in~$\scru_B^\circ(V_*,I_*)$ of $\xi_{I_*}$,
that: \quad $TF_{I_*}^{V_*}=j$ on $N_*$.

Let $N_*:=\Phi_{-202}^{V_0}(N_0)$
and let $H:=\R^3\times(-\infty,-a_{I_1}]$.
Because we have $N_0\subseteq B_\circ(I_1)\subseteq\R^3\times\{-a_{I_1}\}$,
we conclude that $\Phi_{[-202,0]}^{V_0}(N_0)\subseteq H$.
Since $V_1=V_0$ on $(\R^4)\backslash(I_1^4)$
and since $H\subseteq(\R^4)\backslash(I_1^4)$,
we see that $V_1=V_0$ on $H$.
Then, by \cref{cor-vect-agree-implies-flow-agree}, we conclude that
$\Phi_{-202}^{V_1}=\Phi_{-202}^{V_0}$ on~$N_0$.
Then $N_*=\Phi_{-202}^{V_0}(N_0)=\Phi_{-202}^{V_1}(N_0)\subseteq\Phi_\R^{V_1}(N_0)$.
Also, we have $N_0\subseteq N_1\subseteq\scru_B^\circ(V_1,I_1)\subseteq\scru(V_1)$.
So, because $\scru(V_1)$ is $V_1$-invariant,
it follows that $\Phi_\R^{V_1}(N_0)\subseteq\scru(V_1)$.
Recall that $\scru(V_1)=\scru(V_*)$.
Then we have $N_*\subseteq\Phi_\R^{V_1}(N_0)\subseteq\scru(V_1)=\scru(V_*)$.

Recall that $a_{I_1}+202=a_{I_*}$.
So, by \lref{lem-midpoint-connection}(i),
$\xi_{I_*}=\Phi_{-202}^{V_0}(\xi_{I_1})$.
Also, by \lref{lem-midpoint-connection}(iii),
$\tau\mapsto\Phi_{-202}^{V_0}(\tau):B_\circ(I_1)\to B_\circ(I_*)$ is an open map.
Then, since $N_*=\Phi_{-202}^{V_0}(N_0)$ and
since $N_0$ is an open neighborhood in $B_\circ(I_1)$ of~$\xi_{I_1}$,
we see that $N_*$ is an open neighborhood in $B_\circ(I_*)$ of $\xi_{I_*}$.
Also, $N_*\subseteq(\scru(V_*))\cap(B_\circ(I_*))=\scru_B^\circ(V_*,I_*)$.
We conclude that $N_*$ is an open neighborhood in $\scru_B^\circ(V_*,I_*)$ of $\xi_{I_*}$.

Let $j:=m+454$.
We wish to prove that $TF_{I_*}^{V_*}=j$ on $N_*$.

Fix $\rho\in N_*$.
We wish to show that $TF_{I_*}^{V_*}(\rho)=j$.
That is, we wish to show that $\Phi_j^{V_*}(DF_{I_*}^{V_*}(\rho))=UF_{I_*}^{V_*}(\rho)$.
As $\rho\in N_*\subseteq\scru_B(V_*,I_*)$, we get $DF_{I_*}^{V_*}(\rho)=\rho$.
We therefore wish to prove that $\Phi_j^{V_*}(\rho)=UF_{I_*}^{V_*}(\rho)$.
So, since $\{UF_{I_*}^{V_*}(\rho)\}=(\Phi_\R^{V_*}(\rho))\cap(T(I_*))$
and since $\Phi_j^{V_*}(\rho)\in\Phi_\R^{V_*}(\rho)$,
it suffices to prove that $\Phi_j^{V_*}(\rho)\in T(I_*)$.

Let $\nu:=\Phi_{202}^{V_0}(\rho)$.
As $N_*=\Phi_{-202}^{V_0}(N_0)$, we see that $\Phi_{202}^{V_0}(N_*)=N_0$.
Then $\nu=\Phi_{202}^{V_0}(\rho)\in\Phi_{202}^{V_0}(N_*)=N_0\subseteq B_\circ(I_1)=I_1^3\times\{-a_{I_1}\}$.
Fix $w,x,y\in I_1$ such that $\nu=(w,x,y,-a_{I_1})$.
Recall that $a_{I_1}+202=a_{I_*}$.
Then $\rho=\Phi_{-202}^{V_0}(\nu)=(w,x,y,-a_{I_1}-202)=(w,x,y,-a_{I_*})$.

{\it Claim 1:} $\nu=\Phi_{202}^{V_*}(\rho)$.
{\it Proof of Claim 1:}
We have
$$\Phi_{[0,202]}^{V_0}(\rho)\quad=\quad\{(w,x,y)\}\,\times\,[-a_{I_*},-a_{I_*}+202].$$
So, as $-a_{I_*}+202=-a_{I_1}$, we get
$\Phi_{[0,202]}^{V_0}(\rho)\subseteq\R^3\times(-\infty,-a_{I_1}]=H$.

Because $H\subseteq(\R^4)\backslash A$,
we get $V_*=V_1$ on $H$.
Recall that $V_1=V_0$ on~$H$.
Then $V_*=V_1=V_0$ on $H$.
So, as $\Phi_{[0,202]}^{V_0}(\rho)\subseteq H$,
it follows, from \lref{lem-orbits-agree},
that $\Phi_{202}^{V_*}(\rho)=\Phi_{202}^{V_0}(\rho)$.

Then $\nu=\Phi_{202}^{V_0}(\rho)=\Phi_{202}^{V_*}(\rho)$, as desired.
{\it End of proof of Claim 1.}

We have $\nu\in N_0\subseteq\scru_B^\circ(V_1,I_1)$.
Let $t_0:=\Psi(\nu)=TF_{I_1}^{V_1}(\nu)$
and $\sigma:=\Phi_{t_0}^{V_*}(\nu)$
and $\lambda:=\Phi_1^{V_*}(\sigma)$.

{\it Claim 2:} $\sigma=(w,x,y,a_{I_1})$.
{\it Proof of Claim 2:}
As $\nu\in\scru_B^\circ(V_1,I_1)$,
we have $DF_{I_1}^{V_1}(\nu)=\nu$.
So, since $t_0=TF_{I_1}^{V_1}(\nu)$,
we get $\Phi_{t_0}^{V_1}(\nu)=UF_{I_1}^{V_1}(\nu)$.
Recall that $V_*=V_1$ on $I_1^4$.
So, since $(V_1,I_1)\in\scrd$, since $V_*\in\scrc$ and since $\nu\in\scru_B^\circ(V_1,I_1)$,
by \lref{lem-agree-inside-I4},
we get $\Phi_{t_0}^{V_*}(\nu)=\Phi_{t_0}^{V_1}(\nu)$.
Because $(V_1,I_1)\in\scrd_+^\#\subseteq\scrd_+$,
we see that $UF_{I_1}^{V_1}=SU_{I_1}$ on $\scru_B(V_1,I_1)$.
Then $UF_{I_1}^{V_1}(\nu)=SU_{I_1}(\nu)$.
Since $\nu=(w,x,y,-a_{I_1})$, $SU_{I_1}(\nu)=(w,x,y,a_{I_1})$.

Then $\sigma=\Phi_{t_0}^{V_*}(\nu)=\Phi_{t_0}^{V_1}(\nu)=UF_{I_1}^{V_1}(\nu)=SU_{I_1}(\nu)=(w,x,y,a_{I_1})$, as desired.
{\it End of proof of Claim 2.}

{\it Claim 3:} $\lambda=(w,x,y,a_{I_1}+1)$.
{\it Proof of Claim 3:}
By Claim 2, $\Phi_{[0,1]}^{V_0}(\sigma)=\{(w,x,y)\}\times[a_{I_1},a_{I_1}+1]$
and $\Phi_1^{V_0}(\sigma)=(w,x,y,a_{I_1}+1)$.

Because we have $\R^3\times[a_{I_1},a_{I_1}+1]\subseteq(\R^4)\backslash A$,
it follows that $V_*=V_1$ on~$\R^3\times[a_{I_1},a_{I_1}+1]$.
Because we have $\R^3\times[a_{I_1},a_{I_1}+1]\subseteq(\R^4)\backslash(I_1^4)$,
it follows that $V_1=V_0$ on $\R^3\times[a_{I_1},a_{I_1}+1]$.
Then $V_*=V_1=V_0$ on~$\R^3\times[a_{I_1},a_{I_1}+1]$.
So, because
$$\Phi_{[0,1]}^{V_0}(\sigma)\,\,=\,\,\{(w,x,y)\}\,\times\,[a_{I_1},a_{I_1}+1]
\,\,\subseteq\,\,\R^3\,\times\,[a_{I_1},a_{I_1}+1],$$
by \lref{lem-orbits-agree},
we conclude that $\Phi_1^{V_*}(\sigma)=\Phi_1^{V_0}(\sigma)$.
Therefore we have $\lambda=\Phi_1^{V_*}(\sigma)=\Phi_1^{V_0}(\sigma)=(w,x,y,a_{I_1}+1)$.
{\it End of proof of Claim~3.}

Since $w,x,y\in I_1$, by Claim 3, $\lambda\in I_1^3\times\R$
and $\lambda_\#=(w,x,y,-a_{I_1})$.
Then $\lambda_\#=\nu\in B_\circ(I)$.
Let $s:=h(\nu)$ and $\chi:=\Phi_s^{V_*}(\lambda)$
and $\omega:=\Phi_1^{V_*}(\chi)$.

{\it Claim 4:} $\chi=(w,x,y,a_{I_1}+201)$.
{\it Proof of Claim 4:}
Because we have $s=h(\nu)\in h(B_\circ(I_1))$, we see that $200\le s<252$.
Define $b:\R\to\R$ by $b(t)=\Phi_t^{V_\#^s}(-100)$.
Then $b:\R\to\R$ is $C^\infty$ and, for all $t\in\R$,
we have  $b'(t)=V_\#^s(b(t))$, so $1/2\le b'(t)\le 1$.
Define $c:\R\to\R$ by $c(t)=[b(t)]+a_{I_1}+101$.
Then $c:\R\to\R$ is $C^\infty$ and, for all $t\in\R$,
we have $c'(t)=b'(t)$, and so $1/2\le c'(t)\le 1$.
As $b(s)=\Phi_s^{V_\#^s}(-100)=100$,
we get $c(s)=a_{I_1}+201$.
Define $\gamma:\R\to\R^4$ by $\gamma(t)=(w,x,y,c(t))$.
Then $\gamma(s)=(w,x,y,a_{I_1}+201)$.
We wish to prove that $\chi=\gamma(s)$, {\it i.e.}, that $\Phi_s^{V_*}(\lambda)=\gamma(s)$.
We will prove, for all $t\in[0,s]$, that $\gamma(t)=\Phi_t^{V_*}(\lambda)$.

Since $b(0)=-100$, we have $c(0)=a_{I_1}+1$.
By Claim 3, we know that $(w,x,y,a_{I_1}+1)=\lambda$.
Then $\gamma(0)=(w,x,y,a_{I_1}+1)=\lambda=\Phi_0^{V_*}(\lambda)$.
So, by uniqueness of solutions of ODEs, it suffices to show,
for all $t\in[0,s]$, that $\gamma'(t)=V_*(\gamma(t))$.
By continuity, it suffices to show,
for all $t\in(0,s)$, that $\gamma'(t)=V_*(\gamma(t))$.
Fix $t\in(0,s)$ and let $\mu:=\gamma(t)=(w,x,y,c(t))$.
We wish to prove that $\gamma'(t)=V_*(\mu)$.

As $\mu=(w,x,y,c(t))\in I_1^3\times\R$,
we get $\mu_\#=(w,x,y,-a_{I_1})=\nu$,
so $h(\mu_\#)=h(\nu)=s$.
As $c'\ge1/2$, by the Mean Value Theorem,
$c$ is increasing on~$\R$.
Then, since $0<t<s$, we have $c(0)<c(t)<c(s)$.
That is, $a_{I_1}+1<c(t)<a_{I_1}+201$.
So, since $w,x,y\in I_1$, we get
$\mu=(w,x,y,c(t))\in I_1^3\times(a_{I_1}+1,a_{I_1}+201)=I_1^3\times\hatI=A$.
So, since $V_*=Y$ on~$A$, it follows that $V_*(\mu)=Y(\mu)$.

We have $[c(t)]-a_{I_1}-101=b(t)$ and $b'(t)=V_\#^s(b(t))$ and $h(\mu_\#)=s$.
Then $v_\mu=V_\#^{h(\mu_\#)}([c(t)]-a_{I_1}-101)=V_\#^s(b(t))=b'(t)=c'(t)$.
Then we have $\gamma'(t)=(0,0,0,c'(t))=(0,0,0,v_\mu)=Y(\mu)=V_*(\mu)$, as desired.
{\it End of proof of Claim 4.}

{\it Claim 5:} $\omega=(w,x,y,a_{I_1}+202)$.
{\it Proof of Claim 5:}
We define $J:=[a_{I_1}+201,a_{I_1}+202]$.
By Claim~4,
$\Phi_{[0,1]}^{V_0}(\chi)=\{(w,x,y)\}\times J$ and
$\Phi_1^{V_0}(\chi)=(w,x,y,a_{I_1}+202)$.

Because $\R^3\times J\subseteq(\R^4)\backslash A$,
it follows that $V_*=V_1$ on $\R^3\times J$.
Because $\R^3\times J\subseteq(\R^4)\backslash(I_1^4)$,
it follows that $V_1=V_0$ on $\R^3\times J$.
Therefore we have $V_*=V_1=V_0$ on~$\R^3\times J$.
So, because
\begin{itemize}
\item[]$\Phi_{[0,1]}^{V_0}(\chi)\quad=\quad\{(w,x,y)\}\,\times\,J\quad\subseteq\quad\R^3\,\times\,J$,
\end{itemize}
by \lref{lem-orbits-agree},
we conclude that
$\Phi_1^{V_*}(\chi)=\Phi_1^{V_0}(\chi)$.
Thus we have $\omega=\Phi_1^{V_*}(\chi)=\Phi_1^{V_0}(\chi)=(w,x,y,a_{I_1}+202)$.
{\it End of proof of Claim 5.}

Recall that $\sigma=\Phi_{t_0}^{V_*}(\nu)$,
that $\lambda=\Phi_1^{V_*}(\sigma)$,
that $\chi=\Phi_s^{V_*}(\lambda)$
and that $\omega=\Phi_1^{V_*}(\chi)$.
We define $k:=t_0+s+2$.
Then $\omega=\Phi_k^{V_*}(\nu)$.
By Claim 1, we have $\nu=\Phi_{202}^{V_*}(\rho)$.
Then $\omega=\Phi_{k+202}^{V_*}(\rho)$.
We have $s=h(\nu)$ and $t_0=\Psi(\nu)$.
Thus, we have $k+202=t_0+s+204=(\Psi(\nu))+(h(\nu))+204$.
As $\nu\in N_0$, we get $h(\nu)=f(\nu)$.
So, because $f(\nu)=m-(\Psi(\nu))+250$,
we see that $(\Psi(\nu))+(h(\nu))=(\Psi(\nu))+(f(\nu))=m+250$.
Then $k+202=(\Psi(\nu))+(h(\nu))+204=m+250+204=m+454=j$.
Then $\Phi_j^{V_*}(\rho)=\Phi_{k+202}^{V_*}(\rho)=\omega$.
So, by Claim 5, we conclude that
$\Phi_j^{V_*}(\rho)=\omega=(w,x,y,a_{I_1}+202)$.
Moreover, $w,x,y\in I_1\subseteq I_*$ and $a_{I_1}+202=a_{I_*}$,
so $\Phi_j^{V_*}(\rho)\in I_*^3\times\{a_{I_*}\}=T_\circ(I_*)\subseteq T(I_*)$,
as desired.
\end{proof}

\section{The hyperbolic vector field $H$ on $\R^2$\wrlab{sect-start-hyp-vf}}

Let $\pi_1,\pi_2:\R^2\to\R$ be the coordinate projection
maps defined by $\pi_1(w,x)=w$ and $\pi_2(w,x)=x$.
Let $H_0:\R^2\to\R^2$ be defined by
$$H_0(w,x)\,\,=\,\,(d/dt)_{t=0}[(e^tw,e^{-t}x)]\,\,=\,\,(w,-x).$$
Then $H_0$ is complete and, for all $w,x,t\in\R$,
$\Phi_t^{H_0}(w,x)=(e^tw,e^{-t}x)$.

Let $Z:\R^2\to\R$ be the zero function defined by $Z(u)=0$.
Let $c_0:\R^2\to\R$ be a $C^\infty$ function satisfying
\begin{itemize}
\item$0<c_0<1$ on $\R^2\backslash\{(0,0)\}$,
\item$\forall u\in\R^2$, \quad $[c_0(u)]\,[H(u)]\,\,\in\,\,\barI_0^2$ \qquad and
\item$c_0$ agrees with $Z$ to all orders at $(0,0)$.
\end{itemize}
We define $H:=c_0H_0:\R^2\to\R^2$.
Then $H:\R^2\to\R^2$ is $C^\infty$ and $H(\R^2)\subseteq\barI_0^2$.
It follows that $H$ is complete.

I refer to $H_0$ and $H$ as ``hyperbolic''
because any orbit of either is contained in a level set of
the quadratic form $(w,x)\mapsto wx:\R^2\to\R$,
and because a generic level set of this quadratic form is a hyperbola.

\begin{lem}\wrlab{lem-reparam-H0-to-H}
Let $u\in\R^2$.
Then there is an increasing $C^\infty$ diffeomorphism
$g:\R\to\R$ such that, for all $t\in\R$,
$\Phi_t^H(u)=\Phi_{g(t)}^{H_0}(u)$.
\end{lem}

\begin{proof}
If $u=(0,0)$, then, for all $t\in\R$,
$\Phi_t^H(u)=(0,0)=\Phi_t^{H_0}(u)$,
and, in this case, we can define $g:\R\to\R$ by $g(t)=t$.
We may therefore assume that $u\ne(0,0)$.
Then, because $\R^2\backslash\{(0,0)\}$ is $H_0$-invariant,
it follows that $\Phi_\R^{H_0}(u)\subseteq\R^2\backslash\{(0,0)\}$.

Define $W:\R\to\R$ by $W(s)=c_0(\Phi_s^{H_0}(u))$.
Recall that $0<c_0<1$ on~$\R^2\backslash\{(0,0)\}$.
Then $0<W<1$.
Then $W:\R\to\R$ is $C^\infty$ and bounded, and, therefore, complete.
Define $g:\R\to\R$ by $g(t)=\Phi_t^W(0)$.
Then $g:\R\to\R$ is $C^\infty$ and satisfies both $g(0)=0$ and,
for all $t\in\R$, $g'(t)=W(g(t))$.
So, because $W>0$, we see that $g'>0$.
So, by the Mean Value Theorem, $g$ is increasing.
Then $g:\R\to\R$ is injective.

Define a constant function $V:\R\to\R$ by: for all $s\in\R$, $V(s)=1$.
For all $s\in\R$, $W(s)>0$, so $W(s)\ne0$.
Then, by \cref{cor-homothetic-flows}, $\Phi_\R^W(0)=\Phi_\R^V(0)$.
For all $t\in\R$, $\Phi_t^V(0)=t$, so $\Phi_\R^V(0)=\R$.
Therefore $g(\R)=\Phi_\R^W(0)=\Phi_\R^V(0)=\R$.
Then $g:\R\to\R$ is surjective.
Then $g:\R\to\R$ is injective and surjective, hence bijective.
So, as $g'>0$ and as $g:\R\to\R$ is $C^\infty$,
by the Inverse Function Theorem, $g^{-1}:\R\to\R$ is~$C^\infty$.
Then $g:\R\to\R$ is an increasing $C^\infty$ diffeomorphism.
It remains to prove, for all $t\in\R$, that $\Phi_{g(t)}^{H_0}(u)=\Phi_t^H(u)$.

Define $\gamma:\R\to\R^2$ by $\gamma(t)=\Phi_{g(t)}^{H_0}(u)$.
We wish to show, for all $t\in\R$, that $\gamma(t)=\Phi_t^H(u)$.
We have $\gamma(0)=\Phi_{g(0)}^{H_0}(u)=\Phi_0^{H_0}(u)=u=\Phi_0^H(u)$.
Therefore, by uniqueness of solutions of ODEs,
it suffices to prove, for all $t\in\R$, that
$\gamma'(t)=H(\gamma(t))$.

Differentiating the definition of~$\gamma$,
and using the Chain Rule, we see, for all~$t\in\R$,
that $\gamma'(t)=[g'(t)][H_0(\Phi_{g(t)}^{H_0}(u))]$.
Fix $t\in\R$ and let $u_1:=\gamma(t)=\Phi_{g(t)}^{H_0}(u)$.
We wish to prove that $[g'(t)][H_0(u_1)]=H(u_1)$.

By definition of $W$, we have $W(g(t))=c_0(\Phi_{g(t)}^{H_0}(u))$.
Then
$$g'(t)\quad=\quad W(\,g(t)\,)\quad=\quad c_0(\,\Phi_{g(t)}^{H_0}(u)\,)\quad=\quad c_0(u_1).$$
Then $[g'(t)][H_0(u_1)]=[c_0(u_1)][H_0(u_1)]=(c_0H_0)(u_1)=H(u_1)$.
\end{proof}

\begin{lem}\wrlab{lem-H-preserves-axes}
Both of the following are true:
\begin{itemize}
\item[(i)]$H(\R\times\{0\})\subseteq\R\times\{0\}$.
\item[(ii)]$H(\{0\}\times\R)\subseteq\{0\}\times\R$.
\end{itemize}
\end{lem}

\begin{proof}
We only prove (i). The proof of (ii) is similar.

Let $w\in\R$.
We wish to prove that $H(w,0)\in\R\times\{0\}$.

Let $a:=c_0(w,0)$.
As $H_0(w,0)=(w,0)$
and $H(w,0)=a\cdot[H_0(w,0)]$,
it follows that $H(w,0)=(aw,0)\in\R\times\{0\}$, as desired.
\end{proof}

\begin{lem}\wrlab{lem-asymptotic-coords-of-H}
All of the following are true:
\begin{itemize}
\item[(i)]For all $u\in\R^2\backslash[\{0\}\times\R]$, \quad
$|\pi_1(\Phi_t^H(u))|\to\infty$ as $t\to\infty$.
\item[(ii)]For all $u\in\R^2\backslash[\R\times\{0\}]$, \quad
$|\pi_2(\Phi_t^H(u))|\to\infty$ as $t\to-\infty$.
\item[(iii)]For all $u\in\R^2$, \quad
$t\mapsto|\pi_1(\Phi_t^H(u))|:\R\to\R$ is nondecreasing.
\item[(iv)]For all $u\in\R^2$, \quad
$t\mapsto|\pi_2(\Phi_t^H(u))|:\R\to\R$ is nonincreasing.
\end{itemize}
\end{lem}

\begin{proof}
We only prove (i) and (iii). The proofs of (ii) and (iv) are similar.

Fix $u=(w,x)\in\R\times\R=\R^2$.
We wish to prove both
\begin{itemize}
\item if \,\, $w\ne0$, \,\, then \,\, $|\pi_1(\Phi_t^H(u))|\to\infty$ as $t\to\infty$ \qquad\qquad and
\item$t\,\mapsto\,|\pi_1(\Phi_t^H(u))|\,\,:\,\,\R\,\to\,\R$ \quad is nondecreasing.
\end{itemize}

By \lref{lem-reparam-H0-to-H},
let $g:\R\to\R$ be an increasing $C^\infty$ diffeomorphism
such that, for all $t\in\R$, $\Phi_t^H(u)=\Phi_{g(t)}^{H_0}(u)$.
For all $t\in\R$, we have $\pi_1(\Phi_t^H(u))=\pi_1(\Phi_{g(t)}^{H_0}(u))=e^{g(t)}w$.
We therefore wish to prove both
\begin{itemize}
\item if \,\, $w\ne0$, \,\, then \,\, $e^{g(t)}|w|\to\infty$ as $t\to\infty$ \qquad\qquad and
\item$t\,\mapsto e^{g(t)}|w|\,\,\,:\,\,\R\,\to\,\R$ \quad is nondecreasing.
\end{itemize}

Because $g:\R\to\R$ is an increasing $C^\infty$ diffeomorphism,
it follows that $g(t)\to\infty$ as $t\to\infty$.
So, if $w\ne0$, then $e^{g(t)}|w|\to\infty$ as $t\to\infty$.

Since both $g:\R\to\R$ and $t\mapsto e^t|w|:\R\to\R$ are nondecreasing,
the composite $t\mapsto e^{g(t)}|w|:\R\to\R$ is nondecreasing as well.
\end{proof}

\begin{lem}\wrlab{lem-H0-fixes-00-to-all-orders}
Let $t\in\R$. Then the map $\Phi_t^H:\R^2\to\R^2$ agrees
with the identity $\Id_2:\R^2\to\R^2$ to all orders at $(0,0)$.
\end{lem}

\begin{proof}
Define $\bfzero:\R^2\to\R^2$ by: for all $u\in\R^2$, $\bfzero(u)=(0,0)$.
Because $c_0$ agrees with $Z$ agree to all orders at $(0,0)$,
because $H=c_0H_0$ and because $\bfzero=ZH_0$, it follows that
$H$ agrees with $\bfzero$ to all orders at $(0,0)$.
Then, by \cref{cor-vanishing-to-periodic},
$\Phi_t^H$ agrees with  $\Id_2$ to all orders at $(0,0)$.
\end{proof}

\section{The ``racetrack'' vector field $Q$ on $\R^2$\wrlab{sect-racetrack}}

For this section, let $S:=(4I_0)\times(12I_0)$,
let $B:=(4I_0)\times\{-12\}$ and let
$T:=(4I_0)\times\{12\}$.
Let $R$ be an open subset of $\R^2$ such that
\begin{itemize}
\item$R$ is diffeomorphic to an open annulus in $\R^2$ \qquad\qquad and
\item$S\,\cup\,B\cup\,T\quad\subseteq\quad R\quad\subseteq\quad(50I_0)^2$.
\end{itemize}
Let $Q_0:\R^2\to\R^2$ be the constant map defined by $Q_0(y,z)=(0,1)$.
Let $Q:\R^2\to\R^2$ be $C^\infty$ and satisfy
\begin{itemize}
\item$Q(\R^2)\subseteq\barI_0^2$,
\item$R$ is $Q$-invariant,
\item$Q=Q_0$ on $S$ \qquad\qquad and
\item$Q=Q_0$ on $[\R^2]\,\backslash\,[(50I_0)^2]$.
\end{itemize}
Then $Q$ is $C^\infty$ and bounded, so $Q$ is complete.

Fix an integer $m>24$.
Throughout this section, we assume:
\begin{itemize}
\item$\forall v\in R$, \qquad $\Phi_m^Q(v)=v$ \qquad\qquad\qquad and
\item$\forall v\in B$, \, $\forall t\in(0,m)$, \qquad $[\Phi_t^Q(v)\in S]\Leftrightarrow[t<24]$.
\end{itemize}

I picture $R$ in the shape of a racetrack, with $B$ as the starting line,
with $T$ as a checkpoint that occurs just before the first turn of the track,
and with $S$ as the portion of the track between $B$ and $T$.
While in $S$, the runners all travel straight upward in $\R^2$, with unit speed.
Using seconds as our units of time, there exists an integer $m>24$, such that
every runner takes exactly $m$ seconds to complete one lap around the track,
no matter what the runner's starting point is.
Also, each runner starting on the starting line, $B$, spends exactly $24$ seconds in $S$,
and then doesn't return to $S$ in the following $m-24$ seconds.

\begin{lem}\wrlab{lem-follow-Q-in-and-out-of-S}
Let $v=(y,z)\in\R\times\R=\R^2$.
Assume that $v\in S$.
Then, all of the following are true:
\begin{itemize}
\item[(i)]$\forall t\in[-12,12]$, \quad $\Phi_{t-z}^Q(v)=\Phi_{t-z}^{Q_0}(v)=(y,t)$.
\item[(ii)]$\forall t\in[12,m-12]$, \quad $\Phi_{t-z}^Q(v)\in(\R^2)\backslash S$.
\item[(iii)]$\forall t\in[m-12,m+12]$, \quad $\Phi_{t-z}^Q(v)=\Phi_{t-m-z}^{Q_0}(v)=(y,t-m)$.
\end{itemize}
\end{lem}

\begin{proof}
Since $(y,z)=v\in S=(4I_0)\times(12I_0)$, we get $y\in4I_0$ and $z\in12I_0$.
That is, $-4<y<4$ and $-12<z<12$.

{\it Proof of (i):}
For all $t\in\R$, $\Phi_{t-z}^{Q_0}(v)=(y,t)$.
From this and continuity, we need only show, for all $t\in(-12,12)$, that $\Phi_{t-z}^Q(v)=\Phi_{t-z}^{Q_0}(v)$.

We have $-12-z<0<12-z$, {\it i.e.}, $0\in(-12-z,12-z)$.
For all $t\in(-12-z,12-z)$, we have $t+z\in(-12,12)=12I_0$, and so
$$\Phi_t^{Q_0}(v)\quad=\quad(y,t+z)\quad\in\quad(4I_0)\,\times\,(12I_0)\quad\subseteq\quad S.$$
So, because $Q=Q_0$ on $S$,
we conclude, from \lref{lem-orbits-agree},
that, for all~$t\in(-12-z,12-z)$,
we have $\Phi_t^Q(v)=\Phi_t^{Q_0}(v)$.
Equivalently, for all~$t\in(-12,12)$,
we have $\Phi_{t-z}^Q(v)=\Phi_{t-z}^{Q_0}(v)$.
{\it End of proof of~(i).}

{\it Proof of (ii):}
Let $v_0:=\Phi_{-12-z}^Q(v)$.
Then, for all $t\in\R$, we have
$\Phi_{t+12}^Q(v_0)=\Phi_{t-z}^Q(v)$.
By (i), we have $v_0=(y,-12)$.
It follows that $v_0=(y,-12)\in(4I_0)\times\{-12\}=B$.
So, by definition of $m$, we see, for all~$t\in(0,m)$, that: \qquad
$[\Phi_t^Q(v_0)\in S]\Leftrightarrow[t<24]$.

So, for all~$t\in[24,m)$, we have
$\Phi_t^Q(v_0)\in(\R^2)\backslash S$.
So, since $(\R^2)\backslash S$ is closed,
by continuity, for all~$t\in[24,m]$, we get
$\Phi_t^Q(v_0)\in(\R^2)\backslash S$.
Equivalently, for all $t\in[12,m-12]$,
we have $\Phi_{t+12}^Q(v_0)\in(\R^4)\backslash S$,
so $\Phi_{t-z}^Q(v)=\Phi_{t+12}^Q(v_0)\in(\R^4)\backslash S$, as desired.
{\it End of proof of (ii).}

{\it Proof of (iii):}
Fix $t\in[m-12,m+12]$.
Since $\Phi_{t-m-z}^{Q_0}(v)=(y,t-m)$,
it suffices to show that $\Phi_{t-z}^Q(v)=\Phi_{t-m-z}^{Q_0}(v)$.

Since $t-m\in[-12,12]$, by (i), we see that $\Phi_{t-m-z}^Q(v)=\Phi_{t-m-z}^{Q_0}(v)$.
By definition of $m$, because $v\in S\subseteq R$, we see that $\Phi_m^Q(v)=v$.
Applying $\Phi_{t-m-z}^Q$ to this equation gives us $\Phi_{t-z}^Q(v)=\Phi_{t-m-z}^Q(v)$.
Then $\Phi_{t-z}^Q(v)=\Phi_{t-m-z}^Q(v)=\Phi_{t-m-z}^{Q_0}(v)$, as desired.
{\it End of proof of (iii).}
\end{proof}

\begin{lem}\wrlab{lem-reentry-time-to-I2}
Let $a\in(0,4]$.
Let $I:=(-a,a)$.
Let $y\in I$, $z\in12I_0$.
Let $v:=(y,z)\in\R^2$.
Then, for all $t\in[a,m-a]$,
$\Phi_{t-z}^Q(v)\in[\R^2]\backslash[I^2]$.
\end{lem}

\begin{proof}
Since $0<a\le4$, we have $I\subseteq4I_0$ and $[a,12]\subseteq[-12,12]$
and $[m-12,m-a]\subseteq[m-12,m+12]$.
Since $y\in I\subseteq4I_0$ and $z\in12I_0$,
we see that $v=(y,z)\in(4I_0)\times(12I_0)=S$.
Also, because $I\subseteq4I_0\subseteq12I_0$,
we get $I^2\subseteq(4I_0)\times(12I_0)=S$, so
$(\R^2)\backslash S\subseteq[\R^2]\backslash[I^2]$.

For all~$t\in[a,12]$, we have $t\ge a$, so $t\notin I$, so $(y,t)\in[\R^2]\backslash[I^2]$.
For all~$t\in[m-12,m-a]$, we have $t-m\le-a$, so $t-m\notin I$, so $(y,t-m)\in[\R^2]\backslash[I^2]$.

Combining all these observations with \lref{lem-follow-Q-in-and-out-of-S}, we conclude:
\begin{itemize}
\item[(a)]$\forall t\in[a,12]$, \qquad $\Phi_{t-z}^Q(v)=(y,t)\in[\R^2]\backslash[I^2]$.
\item[(b)]$\forall t\in[12,m-12]$,
           \qquad $\Phi_{t-z}^Q(v)\in(\R^2)\backslash S\subseteq[\R^2]\backslash[I^2]$.
\item[(c)]$\forall t\in[m-12,m-a]$,
            \qquad $\Phi_{t-z}^Q(v)=(y,t-m)\in[\R^2]\backslash[I^2]$.
\end{itemize}
Since $[a,m-a]=[a,12]\cup[12,m-12]\cup[m-12,m-a]$, by (a) and (b) and (c),
we see, for all $t\in[a,m-a]$, that $\Phi_{t-z}^Q(v)\in[\R^2]\backslash[I^2]$.
\end{proof}

We record the special cases $a=1$ and $a=4$ of \lref{lem-reentry-time-to-I2}:

\begin{cor}\wrlab{cor-reentry-time-to-I02}
Let $y\in I_0$, $z\in12I_0$.
Let $v:=(y,z)\in\R^2$.
Then, for all $t\in[1,m-1]$, we have
$\Phi_{t-z}^Q(v)\in[\R^2]\backslash[I_0^2]$.
\end{cor}

\begin{cor}\wrlab{cor-reentry-time-to-4I02}
Let $y\in4I_0$, $z\in12I_0$.
Let $v:=(y,z)\in\R^2$.
Then, for all $t\in[4,m-4]$, we have
$\Phi_{t-z}^Q(v)\in[\R^2]\backslash[(4I_0)^2]$.
\end{cor}

\section{Porousness and $\scrp_{I_0}$\wrlab{sect-results-scrp}}
Recall, from \secref{sect-notation}, that $I_0=(-1,1)\subseteq\R$.
The notation $\scrp_I$ (for any $I\in\scri$) is also defined in \secref{sect-notation}.
The main results of this section are \lref{lem-P0-porous} and \lref{lem-P0K0-in-scrp},
which, together, show that there exists $(P_0,K_0)\in\scrp_{I_0}$
such that $P_0$ is porous.

From the perspective of dynamical systems,
porousness is generic properness, while elements of $\scrp_{I_0}$ display strong periodicity.
We are therefore interested in vector fields that combine properness and periodicity.
These two dynamical properties are in opposition to one another,
so their conflation is a challenge.
A water pump is metaphor for this challenge, because the handle of the pump
(or whatever mechanism gives power to the pump) moves in a periodic manner,
while the water flow is proper in the sense that the water moves a great distance.

Let $H$ be the hyperbolic vector field of \secref{sect-start-hyp-vf}.
Let $R$ be the racetrack of \secref{sect-racetrack},
and let $Q$ be the racetrack vector field of \secref{sect-racetrack}.
A generic orbit of $H$ is proper, while $R$ is a nonempty open set of $Q$-periodic orbits.
We will blend $H$ and $Q$ together in a number of ways.
For example, we can simply form the product, $P$, of $H$ and $Q$.
That is, we can define a function $P:\R^4\to\R^4$ by: \qquad
for all \, $\rho\,=\,(u,v)\,\in\,\R^2\times\R^2\,=\,\R^4$,
$$P(\rho)\,\,:=\,\,(\,H(u)\,,\,Q(v)\,)\,\,\in\,\,\R^2\,\times\,\R^2\,\,=\,\,\R^4.$$
The flow of the vector field represented by $P$ can be described as a water pump, as follows.
The runners, driven by $Q$, go around racetrack~$R$.
Their motion powers the pump, and
the water flows out along the hyperbolas that are the orbits of $H$.
The maps $P_*$, $P_+$ and $P_0$, defined below, are variants of $P$.
They also fit into the water pump metaphor,
and one can invent stories to help visualize them.
For example, with~$P_*$, the pump has a defect:
While the runners are inside $(3I_0)^2$, the water flow stops,
but, when the runners are outside~$(4I_0)^2$,
the pump works at full strength.
Because the runners spend less than half of their time in~$(4I_0)^2$,
the problem is intermittent, and the water still flows, albeit with occasional delays.
With the pump defined by $P_+$, the defect just described still exists,
but, in addition, there's another twist:
Once the water has flowed far enough,
the runners abandon the racetrack and
start running straight upward in~$\R^2$ at unit speed.
With $P_0$ there is yet one more feature in the dynamics:
After the runners leave the racetrack,
the water eventually ceases to flow.
We now present details.

Recall that $R\subseteq\R^2$ is the racetrack of \secref{sect-racetrack}.
As in \secref{sect-racetrack}, let
$$S:=(4I_0)\times(12I_0)\subseteq\R^2\quad
\hbox{and}\quad B:=(4I_0)\times\{-12\}\subseteq\R^2.$$
Recall, from \secref{sect-notation}, the definition of $\zeta_I$ (for any $I\in\scri$).
Define $\alpha:\R^2\to\R$ by
$\alpha(y,z)=1-[\zeta_{3I_0}(y)][\zeta_{3I_0}(z)]$.
Then
\begin{itemize}
\item$0\le\alpha\le1$,
\item$\alpha=0$ on $(\,\overline{3I_0}\,)^2$ \qquad and
\item$\alpha=1$ on $[\R^2]\backslash[(4I_0)^2]$.
\end{itemize}
Define $P_*:\R^4\to\R^4$ by: \qquad
for all \, $\rho\,=\,(u,v)\,\in\,\R^2\times\R^2\,=\,\R^4$,
$$P_*(\rho)\,\,:=\,\,(\,[\alpha(v)][H(u)]\,,\,Q(v)\,)\,\,\in\,\,\R^2\,\times\R^2\,\,=\,\,\R^4.$$
Then $P_*:\R^4\to\R^4$ is $C^\infty$ and $P_*(\R^4)\subseteq\barI_0^4$.
Then $P_*$ is complete.

Fix $C^\infty$ maps $\beta,\gamma:\R^4\to\R$ such that
\begin{itemize}
\item$0\le\beta\le1$ \quad and \quad $0\le\gamma\le1$,
\item$\beta=1$ on $(\,\overline{100I_0}\,)^4$ \quad and \quad $\beta=0$ on $\R^4\,\,\backslash\,\,[(200I_0)^4]$ \qquad and
\item$\gamma=1$ on $(\,\overline{300I_0}\,)^4$ \quad and \quad $\gamma=0$ on $\R^4\,\,\backslash\,\,[(400I_0)^4]$.
\end{itemize}
Let $J_0:=200I_0$ and let $K_0:=400I_0$.
Then $\beta=0$ on $(\R^4)\backslash(J_0^4)$ and $\beta=\gamma=0$ on $(\R^4)\backslash(K_0^4)$.
Also, $\beta=\gamma=1$ on $(\,\overline{100I_0}\,)^4$.

Let $\bfzero:\R^2\to\R^2$ be the zero map defined by $\bfzero(w,x)=(0,0)$.
Let $Q_0:\R^2\to\R^2$ be the constant map defined by $Q_0(y,z)=(0,1)$.
Note that $V_0:\R^4\to\R^4$ satisfies: \qquad for all \, $\rho\,=\,(u,v)\,\in\,\R^2\times\R^2\,=\,\R^4$,
$$V_0(\rho)\quad=\quad(\,\bfzero(u)\,,\,Q_0(v)\,)\quad\in\quad\R^2\,\times\,\R^2\quad=\quad\R^4.$$
Define $X:\R^4\to\R^4$ by: \qquad for all \, $\rho\,=\,(u,v)\,\in\,\R^2\times\R^2\,=\,\R^4$,
$$X(\rho)\quad:=\quad(\,H(u)\,,\,Q_0(v)\,)\quad\in\quad\R^2\,\times\,\R^2\quad=\quad\R^4.$$
Then $X:\R^4\to\R^4$ is $C^\infty$ and $X(\R^4)\subseteq\barI_0^4$.
Let $P_+:=\beta P_*+(1-\beta)X$.
Then $P_+:\R^4\to\R^4$ is $C^\infty$ and $P_+(\R^4)\subseteq\barI_0^4$.
Let $P_0:=\gamma P_++(1-\gamma)V_0$.
Then $P_0:\R^4\to\R^4$ is $C^\infty$ and $P_0(\R^4)\subseteq\barI_0^4$.
Then $P_0$ is complete.

Define a $C^\infty$ map $\omega:\R^4\to\R$ by:
for all $\rho=(u,v)\in\R^2\times\R^2=\R^4$,
$$\omega(\rho)\quad:=\quad[\,\gamma(\rho)\,]\,\,[\,(\alpha(v))(\beta(\rho))\,+\,1\,-\,(\beta(\rho))\,].$$
For all $\rho\in\R^4$, we define $\Omega_\rho:\R\to\R$ by
$\displaystyle{\Omega_\rho(t)\:=\int_0^t\,\left[\omega\left(\Phi_r^{P_0}(\rho)\right)\right]\,dr}$.

Following the construction in~\secref{sect-racetrack},
fix an integer $m>24$ such that
\begin{itemize}
\item$\forall v\in R$, \qquad $\Phi_m^Q(v)=v$ \qquad\qquad\qquad and
\item$\forall v\in B$, \, $\forall t\in(0,m)$, \qquad $[\Phi_t^Q(v)\in S]\Leftrightarrow[t<24]$.
\end{itemize}
Let $\displaystyle{T_\times:=m-\left[\int_{-4}^4\,(\zeta_{3I_0}(t))\,dt\right]}$.

Let $Z_0:=(\{0\}\times\R)\cup(\R\times\{0\})\subseteq\R^2$.
Let $Z:=Z_0\times\R^2\subseteq\R^4$.
Then $Z_0$ is nowhere dense in $\R^2$, and so $Z$ is nowhere dense in $\R^4$.

Define $\Pi_{12}:\R^4\to\R^2$ and $\Pi_{34}:\R^4\to\R^2$ by
$$\Pi_{12}(w,x,y,z)\,\,=\,\,(w,x)\qquad\hbox{and}\qquad\Pi_{34}(w,x,y,z)\,\,=\,\,(y,z).$$

\begin{lem}\wrlab{lem-water-never-reverses}
Let $\sigma=(u,v)\in\R^2\times\R^2=\R^4$.
Then
\begin{itemize}
\item[(i)]$\omega(\sigma)\quad\ge\quad0$ \qquad\qquad\qquad\qquad and
\item[(ii)]$\Pi_{12}(P_0(\sigma))\quad=\quad[\omega(\sigma)]\,\,[H(u)]$.
\end{itemize}
\end{lem}

\begin{proof}
Let $p:=\alpha(v)$ and $q:=\beta(\sigma)$ and $r:=\gamma(\sigma)$.
Then $p,q,r\in[0,1]$.
Let $s:=pq+1-q$.
Since $pq\ge0$ and $1-q\ge0$, we get $s\ge0$.
By definition of $\omega$, we have $\omega(\sigma)=r\cdot[pq+1-q]$.
That is, $\omega(\sigma)=rs$.
So, as $r\ge0$ and $s\ge0$, we get $\omega(\sigma)\ge0$,
which proves (i).

As $\Pi_{12}(P_*(\sigma))=p\cdot[H(u)]$
and $\Pi_{12}(X(\sigma))=1\cdot[H(u)]$, we get
$\Pi_{12}(P_+(\sigma))=[qp+(1-q)(1)][H(u)]=[pq+1-q][H(u)]=s\cdot[H(u)]$.
Therefore, because $\Pi_{12}(V_0(\sigma))=\bfzero(u)=0\cdot[H(u)]$,
it follows that $\Pi_{12}(P_0(\sigma))=[rs+(1-r)(0)][H(u)]=[rs][H(u)]=[\omega(\sigma)][H(u)]$,
which proves (ii).
\end{proof}

\begin{lem}\wrlab{lem-following-the-waterflow-of-P0}
Let $\rho=(u,v)\in\R^2\times\R^2=\R^4$.
Define $f:\R\to\R$ by $f(t)=\Omega_\rho(t)$.
Then all of the following are true:
\begin{itemize}
\item[(i)]$f(0)=0$.
\item[(ii)]For all $t\in\R$, we have $f'(t)=\omega(\Phi_t^{P_0}(\rho))\ge0$.
\item[(iii)]$f:\R\to\R$ is nondecreasing.
\item[(iv)]For all $t\ge0$, we have $f(t)\ge0$.
\item[(v)]For all $t\in\R$, we have $\Pi_{12}(\Phi_t^{P_0}(\rho))=\Phi_{f(t)}^H(u)$.
\end{itemize}
\end{lem}

\begin{proof}
For all $t\in\R$, we have
$\displaystyle{f(t)=\Omega_\rho(t)=\int_0^t\,\left[\omega\left(\Phi_r^{P_0}(\rho)\right)\right]\,dr}$,
so $f(0)=0$, and, by the Fundamental Theorem of Calculus,
for all $t\in\R$, we have $f'(t)=\omega(\Phi_t^{P_0}(\rho))$;
also, by \lref{lem-water-never-reverses}(i), $\omega(\Phi_t^{P_0}(\rho))\ge0$.
This proves (i) and (ii).
By (ii), $f'\ge0$,
so, by the Mean Value Theorem, $f:\R\to\R$ is nondecreasing,
which proves (iii).
Next, (iv) follows from (i) and (iii).
It remains to prove (v).
Define $\theta,\kappa:\R\to\R^2$ by
\begin{itemize}
\item[]$\theta(t)\,\,=\,\,\Pi_{12}(\,\Phi_t^{P_0}(\rho)\,)
\qquad\hbox{and}\qquad
\kappa(t)\,\,=\,\,\Phi_{f(t)}^H(u)$.
\end{itemize}
We wish to show, for all $t\in\R$, that
$\theta(t)=\kappa(t)$.

For all $t\in\R$, define $V_t:\R^2\to\R^2$ by $V_t(u)=[f'(t)][H(u)]$.
Differentiating the definition of $\kappa$,
and using the Chain Rule, we conclude, for all $t\in\R$, that
$\kappa'(t)=[f'(t)][H(\Phi_{f(t)}^H(u))]$.
Then, for all $t\in\R$, we have $\kappa'(t)=[f'(t)][H(\kappa(t))]=V_t(\kappa(t))$.
By (i), we have $f(0)=0$.
Then $\kappa(0)=u=\Pi_{12}(\rho)=\theta(0)$.
So, by uniqueness of solutions of (time-dependent) ODEs,
it suffices to show, for all $t\in\R$, that
$\theta'(t)=V_t(\theta(t))$.

Differentiating the definition of $\theta$,
and using linearity of $\Pi_{12}$, we see, for all $t\in\R$, that
$\theta'(t)=\Pi_{12}(P_0(\Phi_t^{P_0}(\rho)))$.
Fix $t\in\R$ and let $\rho_1:=\Phi_t^{P_0}(\rho)$.
We wish to show that $\Pi_{12}(P_0(\rho_1))=V_t(\theta(t))$.

Choose $u_1,v_1\in\R^2$ such that $\rho_1=(u_1,v_1)\in\R^2\times\R^2=\R^4$.
By \lref{lem-water-never-reverses}(ii),
$\Pi_{12}(P_0(\rho_1))=[\omega(\rho_1)][H(u_1)]$.
By definition of $\theta$,
we have $\theta(t)=\Pi_{12}(\rho_1)$,
so $\theta(t)=u_1$.
By definition of $V_t$, we see that $V_t(u_1)=[f'(t)][H(u_1)]$.
By (ii), we have $f'(t)=\omega(\rho_1)$.
Then $\Pi_{12}(P_0(\rho_1))\!=\![\omega(\rho_1)][H(u_1)]\!=\![f'(t)][H(u_1)]\!=\!V_t(u_1)\!=\!V_t(\theta(t))$.
\end{proof}

\begin{lem}\wrlab{lem-Pi1-Pi2-of-P0-orbits-are-monotone}
Let $\rho\in\R^4$.
Then
\begin{itemize}
\item[(i)]$t\mapsto|\Pi_1(\Phi_t^{P_0}(\rho))|:\R\to\R$ is nondecreasing \qquad and
\item[(ii)]$t\mapsto|\Pi_2(\Phi_t^{P_0}(\rho))|:\R\to\R$ is nonincreasing.
\end{itemize}
\end{lem}

\begin{proof}
We will only prove (i). The proof of (ii) is similar.

Choose $u,v\in\R^2$ such that $\rho=(u,v)\in\R^2\times\R^2=\R^4$.
Define $f:\R\to\R$ by $f(t)=\Omega_\rho(t)$.
By \lref{lem-following-the-waterflow-of-P0}(iii), $f:\R\to\R$ is nondecreasing.
Let $\pi_1,\pi_2:\R^2\to\R$ be the coordinate projection maps
defined by $\pi_1(w,x)=w$ and $\pi_2(w,x)=x$.
By \lref{lem-following-the-waterflow-of-P0}(v), 
for all $t\in\R$, $\Pi_{12}(\Phi_t^{P_0}(\rho))=\Phi_{f(t)}^H(u)$,
so $\Pi_1(\Phi_t^{P_0}(\rho))\,\,=\,\,\pi_1(\Phi_{f(t)}^H(u))$.
We therefore wish to prove: $t\mapsto|\pi_1(\Phi_{f(t)}^H(u))|:\R\to\R$ is nondecreasing.

By \lref{lem-asymptotic-coords-of-H}(iii),
we know that $t\mapsto|\pi_1(\Phi_t^H(u))|:\R\to\R$ is nondecreasing.
So, since $f:\R\to\R$ is also nondecreasing,
the composite $t\mapsto|\pi_1(\Phi_{f(t)}^H(u))|:\R\to\R$
is nondecreasing, as desired.
\end{proof}

\begin{lem}\wrlab{lem-following-the-runners-of-Pstar}
Let $\rho=(u,v)\in\R^2\times\R^2=\R^4$.
Then, for all $t\in\R$, we have $\Pi_{34}(\Phi_t^{P_*}(\rho))=\Phi_t^Q(v)$.
\end{lem}

\begin{proof}
Define $\lambda:\R\to\R^2$
by $\lambda(t)=\Pi_{34}(\Phi_t^{P_*}(\rho))$.
We wish to show, for all $t\in\R$, that $\lambda(t)=\Phi_t^Q(v)$.
As $\lambda(0)=\Pi_{34}(\rho)=v=\Phi_0^Q(v)$,
by uniqueness of solutions of ODEs, it suffices to prove,
for all $t\in\R$, that $\lambda'(t)=Q(\lambda(t))$.
Differentiating the definition of $\lambda$, and using linearity of $\Pi_{34}$,
we get: for all $t\in\R$, $\lambda'(t)=\Pi_{34}(P_*(\Phi_t^{P_*}(\rho)))$.
Fix $t\in\R$ and let $\rho_1:=\Phi_t^{P_*}(\rho)$.
We wish to prove that $\Pi_{34}(P_*(\rho_1))=Q(\lambda(t))$.

Fix $u_1,v_1\in\R^2$ such that $\rho_1=(u_1,v_1)\in\R^2\times\R^2=\R^4$.
By definition of $\lambda$, we have $\lambda(t)=\Pi_{34}(\rho_1)$, so $\lambda(t)=v_1$.
By definition of $P_*$, we have $\Pi_{34}(P_*(\rho_1))=Q(v_1)$.
Then $\Pi_{34}(P_*(\rho_1))=Q(v_1)=Q(\lambda(t))$.
\end{proof}

\begin{lem}\wrlab{lem-P0-calculations}
All of the following are true.
\begin{itemize}
\item[(i)]$P_*=P_0$ on $(\,\overline{100I_0}\,)^4$.
\item[(ii)]$P_0=V_0$ on $(\,\overline{3I_0}\,)^4$.
\item[(iii)]$\omega=\alpha\circ\Pi_{34}$ on $(\,\overline{100I_0}\,)^4$.
\item[(iv)]$\omega=1$ on $[J_0^4]\,\backslash\,[J_0^2\times(4I_0)^2]$.
\item[(v)]$\forall\rho\in\R^2\times S$, \qquad $\Pi_{34}(P_0(\rho))=(0,1)\in\R^2$.
\item[(vi)]$\forall\rho\in[\R^4]\,\backslash\,[(200I_0)^2\times(50I_0)^2]$,
            \qquad $\Pi_{34}(P_0(\rho))=(0,1)\in\R^2$.
\item[(vii)]$Z$ is $P_0$-invariant.
\item[(viii)]$\R^2\times R$ is $P_*$-invariant.
\end{itemize}
\end{lem}

\begin{proof}
Let $w:=(0,1)\in\R^2$.
For all $v\in\R^2$, we have $Q_0(v)=w$.
So, for all $\rho\in\R^4$, we have $\Pi_{34}(V_0(\rho))=\Pi_{34}(X(\rho))=w$
and, moreover, by the construction of~$Q$ in \secref{sect-racetrack}, we have both
$$[\,\,\forall v\in S,\,\,Q(v)=w\,\,]\qquad\hbox{and}\qquad[\,\,\forall v\in[\R^2]\backslash[(50I_0)^2],\,\,Q(v)=w\,\,].$$

By definition of $P_*$, for all $\rho=(u,v)\in\R^2\times\R^2=\R^4$,
we have
$$P_*(\rho)\,\,=\,\,(\,[\alpha(v)][H(u)]\,,\,Q(v)\,)\,\,\in\,\,\R^2\,\times\R^2\,\,=\,\,\R^4.$$
so $\Pi_{34}(P_*(\rho))=Q(v)$.

{\it Proof of (i):}
On $(\,\overline{100I_0}\,)^4\,$,
because $\beta=1$ and $\gamma=1$, we get $P_*=P_+$ and $P_+=P_0$.
So, on $(\,\overline{100I_0}\,)^4\,$,
$P_*=P_+=P_0$.
{\it End of proof of (i).}

{\it Proof of (ii):}
Fix $\rho=(u,v)\in(\,\overline{3I_0}\,)^2\times(\,\overline{3I_0}\,)^2=(\,\overline{3I_0}\,)^4$.
We wish to show that $P_0(\rho)=(0,0,0,1)\in\R^4$.

As $v\in(\,\overline{3I_0}\,)^2$,
we have $\alpha(v)=0$.
As $v\in(\,\overline{3I_0}\,)^2\subseteq(4I_0)\times(12I_0)=S$,
we see that $Q(v)=w$.
As $\rho\in(\,\overline{3I_0}\,)^4\subseteq(\,\overline{100I_0}\,)^4$,
by (i), $P_*(\rho)=P_0(\rho)$.
Therefore $P_0(\rho)=P_*(\rho)=([\alpha(v)][H(u)],Q(v))=(0\cdot[H(u)],w)$,
and so $P_0(\rho)=(0,0,0,1)\in\R^4$, as desired.
{\it End of proof of (ii).}

{\it Proof of (iii):}
Let $\rho=(u,v)\in(\,\overline{100I_0}\,)^2\times(\,\overline{100I_0}\,)^2=(\,\overline{100I_0}\,)^4$.
We wish to show that $\omega(\rho)=\alpha(\Pi_{34}(\rho))$,
{\it i.e.}, that $\omega(\rho)=\alpha(v)$.

As $\rho\in(\,\overline{100I_0}\,)^4$,
we get $\beta(\rho)=\gamma(\rho)=1$.
Then, by definition of $\omega$,
$\omega(\rho)=[1][(\alpha(v))(1)+1-(1)]=\alpha(v)$, as desired.
{\it End of proof of (iii).}

{\it Proof of (iv):}
Let $\rho=(u,v)\in J_0^2\times J_0^2=J_0^4$
and assume that $v\notin(4I_0)^2$.
We wish to prove that $\omega(\rho)=1$.

Since $v\in[\R^2]\backslash[(4I_0)^2]$, we have $\alpha(v)=1$.
Since $\rho\in J_0^4\subseteq(\,\overline{300I_0}\,)^4$, we have $\gamma(\rho)=1$.
Then, from the definition of $\omega$, we conclude that
$\omega(\rho)=[1][(1)(\beta(\rho))+1-(\beta(\rho))]=1$, as desired.
{\it End of proof of (iv).}

{\it Proof of (v):}
Fix $u\in\R^2$ and $v\in S$.
Let $\rho=(u,v)\in\R^2\times\R^2=\R^4$.
We wish to show that $\Pi_{34}(P_0(\rho))=w$.

Because $v\in S$, we get $Q(v)=w$.
Then $\Pi_{34}(P_*(\rho))=Q(v)=w$.
So, since we have both $P_+=\beta P_*+(1-\beta)X$ and $\Pi_{34}(X(\rho))=w$,
we see that $\Pi_{34}(P_+(\rho))=w$.
So, since we have both $P_0=\gamma P_++(1-\gamma)V_0$ and $\Pi_{34}(V_0(\rho))=w$,
we see that $\Pi_{34}(P_0(\rho))=w$.
{\it End of proof of (v).}

{\it Proof of (vi):}
Fix $\rho=(u,v)\in\R^2\times\R^2=\R^4$
and assume that $\rho\notin(200I_0)^2\times(50I_0)^2$.
We wish to prove that $\Pi_{34}(P_0(\rho))=w$.
We have $\Pi_{34}(V_0(\rho))=w$,
so, because $P_0=\gamma P_++(1-\gamma)V_0$,
it suffices to prove that $\Pi_{34}(P_+(\rho))=w$.

If $u\notin J_0^2$,
then $\rho=(u,v)\notin J_0^4$, and so $\beta(\rho)=0$, and so,
since $P_+=\beta P_*+(1-\beta)X$, we conclude that $P_+(\rho)=X(\rho)$,
which yields $\Pi_{34}(P_+(\rho))=\Pi_{34}(X(\rho))=w$,
as desired.
We may therefore assume $u\in J_0^2$, {\it i.e.}, $u\in(200I_0)^2$.
So, since $(u,v)=\rho\notin(200I_0)^2\times(50I_0)^2$, we get $v\notin(50I_0)^2$.
Then $Q(v)=w$.
Then $\Pi_{34}(P_*(\rho))=Q(v)=w$.
So, since we have both $P_+=\beta P_*+(1-\beta)X$ and $\Pi_{34}(X(\rho))=w$,
we see that $\Pi_{34}(P_+(\rho))=w$, as desired.
{\it End of proof of (vi).}

Let $Z':=\R\times\{0\}\times\R^2\subseteq\R^4$
and let $Z'':=\{0\}\times\R\times\R^2\subseteq\R^4$.

{\it Claim 1:}
$P_0(Z')\subseteq Z'$.
{\it Proof of Claim 1:}
Fix $\rho\in Z'$.
We wish to prove that $P_0(\rho)\in Z'$.

Fix $u,v\in\R^2$ such that $\rho=(u,v)\in\R^2\times\R^2=\R^4$.
Let $c:=\omega(\rho)$.
By \lref{lem-water-never-reverses}(ii),
$\Pi_{12}(P_0(\rho))=c\cdot[H(u)]$.
Since $\rho\in Z'$, it follows that $u\in\R\times\{0\}$.
Then, by \lref{lem-H-preserves-axes}(i),
we have $H(u)\in\R\times\{0\}$.
Then $\Pi_{12}(P_0(\rho))=c\cdot[H(u)]\in\R\times\{0\}$,
so $P_0(\rho)\in\Pi_{12}^{-1}(\R\times\{0\})=Z'$, as desired.
{\it End of proof of Claim~1.}

{\it Claim 2:}
$P_0(Z'')\subseteq Z''$.
{\it Proof of Claim 2:}
Similiar to the proof of~Claim 1.
{\it End of proof of Claim~2.}

{\it Proof of (vii):}
Both $Z'$ and $Z''$ are vector subspaces of $\R^4$.
So, by Claim 1, $Z'$~is $P_0$-invariant,
and, by Claim 2, $Z''$ is $P_0$-invariant.
So, since $Z=(Z')\cup(Z'')$,
$Z$ is $P_0$-invariant as well.
{\it End of proof of (vii).}

{\it Proof of (viii):}
Fix $t\in\R$.
We wish to prove $\Phi_t^{P_*}(\R^2\times R)\subseteq\R^2\times R$.
Fix $\rho=(u,v)\in\R^2\times\R^2=\R^4$, assume that $v\in R$
and let $\rho_1:=\Phi_t^{P_*}(\rho)$.
We wish to show that $\rho_1\in\R^2\times R$,
{\it i.e.}, that $\Pi_{34}(\rho_1)\in R$.

By the construction of $Q$ in \secref{sect-racetrack}, $R$ is $Q$-invariant.
So, as $v\in R$, we get $\Phi_t^Q(v)\in R$.
By \lref{lem-following-the-runners-of-Pstar},
we have $\Pi_{34}(\Phi_t^{P_*}(\rho))=\Phi_t^Q(v)$.
Then $\Pi_{34}(\rho_1)=\Pi_{34}(\Phi_t^{P_*}(\rho))=\Phi_t^Q(v)\in R$.
{\it End of proof of (viii).}
\end{proof}

\begin{lem}\wrlab{lem-following-the-runners-of-P0-in-J04}
Let $\rho=(u,v)\in\R^2\times\R^2=\R^4$.
Let $J\subseteq\R$ be an interval.
Assume that $0\in J$ and that $\Phi_J^{P_0}(\rho)\subseteq(\,\overline{100I_0}\,)^4$.
Then, for all~$t\in J$, we have
$\Pi_{34}(\Phi_t^{P_0}(\rho))=\Phi_t^Q(v)$.
\end{lem}

\begin{proof}
Fix $t\in J$.
By \lref{lem-following-the-runners-of-Pstar}, we have
$\Pi_{34}(\Phi_t^{P_*}(\rho))=\Phi_t^Q(v)$,
so it suffices to show that $\Phi_t^{P_0}(\rho)=\Phi_t^{P_*}(\rho)$.

By \lref{lem-P0-calculations}(i),
we have $P_*=P_0$ on $(\,\overline{100I_0}\,)^4\,$.
So, since $0\in J$ and $\Phi_J^{P_0}(\rho)\subseteq(\,\overline{100I_0}\,)^4$,
by \lref{lem-orbits-agree},
we get $\Phi_t^{P_0}(\rho)=\Phi_t^{P_*}(\rho)$.
\end{proof}

\begin{lem}\wrlab{lem-above-or-below-means-nostop}
Let $\rho\in\R^4$. Then both of the following are true:
\begin{itemize}
\item[(i)]Say $\Pi_4(\rho)\ge a_{J_0}$.
Then, for all $t\ge0$, $\Pi_4(\Phi_t^{P_0}(\rho))\ge a_{J_0}$.
\item[(ii)]Say $\Pi_4(\rho)\le-a_{J_0}$.
Then, for all $t\le0$, $\Pi_4(\Phi_t^{P_0}(\rho))\le-a_{J_0}$.
\end{itemize}
\end{lem}

\begin{proof}
We will only prove (i). The proof of (ii) is similar.

Define $\lambda:\R\to\R$ by $\lambda(t)=\Pi_4(\Phi_t^{P_0}(\rho))$.
Fix $r\in[0,\infty)$ and assume, for a contradiction, that $\lambda(r)<a_{J_0}$.

Recall that $J_0=200I_0$, so $a_{J_0}=200a_{I_0}=200$.
Choose $b\ge50$ such that $\lambda(r)<b<a_{J_0}$.
Let $A:=\{t\in[0,\infty)|\lambda(t)\ge b\}$.
Since $\lambda(r)<b$, we have $r\notin A$.
Let $t_0:=\inf[0,\infty)\backslash A$.
Then $[0,t_0)\subseteq A$.
Moreover, for all~$\delta>0$,
we have $[t_0,t_0+\delta)\not\subseteq A$.
Let $\delta_1,\delta_2,\ldots$ be a sequence in $(0,\infty)$
such that $\delta_j\to0$ as $j\to\infty$.
For all integers $j\ge1$, choose $u_j\in[t_0,t_0+\delta_j)$
such that $u_j\notin A$.
Then $u_j\to t_0$ as $j\to\infty$, and, for all integers $j\ge1$,
we have $\lambda(u_j)<b$.
Taking the limit as $j\to\infty$,
and using continuity of $\lambda$,
we conclude that $\lambda(t_0)\le b$.
By definition of $\lambda$, $\lambda(0)=\Pi_4(\rho)$.
By assumption, $a_{J_0}\le\Pi_4(\rho)$.
Then $\lambda(t_0)\le b<a_{J_0}\le\Pi_4(\rho)=\lambda(0)$,
so $\lambda(t_0)<\lambda(0)$.
Then $t_0\ne0$.
So, since $t_0\in[0,\infty)$, we conclude that $t_0>0$.
Then, by the Mean Value Theorem, fix $t_1\in(0,t_0)$ such that
$\lambda'(t_1)=[(\lambda(t_0))-(\lambda(0))]/t_0$.
From this, because $\lambda(t_0)<\lambda(0)$ and because $t_0>0$,
it follows that $\lambda'(t_1)<0$.

Because $t_1\in(0,t_0)\subseteq[0,t_0)\subseteq A$,
we conclude that $\lambda(t_1)\ge b$.
We define $\rho_1:=\Phi_{t_1}^{P_0}(\rho)$.
Then, by definition of $\lambda$, we have $\lambda(t_1)=\Pi_4(\rho_1)$.
Therefore $\Pi_4(\rho_1)=\lambda(t_1)\ge b\ge50$,
and so $\Pi_4(\rho_1)\notin 50I_0$,
and so $\rho_1\in[\R^4]\backslash[(200I_0)^2\times(50I_0)^2]$.
Then, by \lref{lem-P0-calculations}(vi),
we see that $\Pi_{34}(P_0(\rho_1))=(0,1)\in\R^2$, and so $\Pi_4(P_0(\rho_1))=1$.

Differentiating the definition of $\lambda$, and using linearity of $\Pi_4$,
we see, for all $t\in\R$, that $\lambda'(t)=\Pi_4(P_0(\Phi_t^{P_0}(\rho)))$.
Then $\lambda'(t_1)=\Pi_4(P_0(\rho_1))$.
Then $1=\Pi_4(P_0(\rho_1))=\lambda'(t_1)<0$, contradiction.
\end{proof}

\begin{cor}\wrlab{cor-above-or-below-means-nostop}
Let $\sigma\in\R^4$. Then both of the following are true:
\begin{itemize}
\item[(i)]Say $\Pi_4(\sigma)>-a_{J_0}$.
Then, for all $t\ge0$, $\Pi_4(\Phi_t^{P_0}(\sigma))>-a_{J_0}$.
\item[(ii)]Say $\Pi_4(\sigma)<a_{J_0}$.
Then, for all $t\le0$, $\Pi_4(\Phi_t^{P_0}(\sigma))<a_{J_0}$.
\end{itemize}
\end{cor}

\begin{proof}
We will only prove (i). The proof of (ii) is similar.
Fix $t\ge0$ and let $\rho:=\Phi_t^{P_0}(\sigma)$.
Assume, for a contradiction, that $\Pi_4(\rho)\le-a_{J_0}$.

Then, by \lref{lem-above-or-below-means-nostop}(ii),
$\Pi_4(\Phi_{-t}^{P_0}(\rho))\le-a_{J_0}$.
So, as $\Phi_{-t}^{P_0}(\rho)=\sigma$,
$\Pi_4(\sigma)\le-a_{J_0}$,
contradicting the assumption that $\Pi_4(\sigma)>-a_{J_0}$.
\end{proof}

\begin{lem}\wrlab{lem-adapted-flow-goes-down}
Let $\sigma\in\Pi_3^{-1}(\R\backslash J_0)$.
Then, for all $t\in\R$, we have $\Pi_3(\Phi_t^{P_0}(\sigma))=\Pi_3(\sigma)$.
\end{lem}

\begin{proof}
Let $a:=\Pi_3(\sigma)$.
Then $a\in\R\backslash J_0$.
Define $F,G:\R^3\to\R^4$ by
$$F(w,x,z)\,\,=\,\,(w,x,a,z)\qquad\hbox{and}\qquad G(w,x,z)\,\,=\,\,(w,x,0,z).$$
Let $X:=\Pi_3^{-1}(a)=F(\R^3)\subseteq\R^4$
and $Y:=\Pi_3^{-1}(0)=G(\R^3)\subseteq\R^4$.
Define $\pi:\R^4\to\R^3$ by $\pi(w,x,y,z)=(w,x,z)$.
For all $\tau\in X$, we have $F(\pi(\tau))=\tau$.
For all $\tau\in Y$, we have $G(\pi(\tau))=\tau$.
We wish to show, for all $t\in\R$,
that $\Pi_3(\Phi_t^{P_0}(\sigma))=a$, {\it i.e.},
that $\Phi_t^{P_0}(\sigma)\in X$.

Since $\Pi_3(\sigma)=a$, we get $\sigma\in\Pi_3^{-1}(a)=X$, so $F(\pi(\sigma))=\sigma$.
Let $\rho:=\pi(\sigma)$.
Then $F(\rho)=F(\pi(\sigma))=\sigma$.
Define $V:\R^3\to\R^3$ by $V(\phi)=\pi(P_0(F(\phi)))$.
Then $V(\R^3)\subseteq\pi(P_0(\R^4))\subseteq\pi(\,\barI_0^4\,)=\barI_0^3$.
So, since $V$ is $C^\infty$, we see that $V$ is complete.
Define $\gamma:\R\to\R^4$ by $\gamma(t)=F(\Phi_t^V(\rho))$.
Then, for all~$t\in\R$, we have $\gamma(t)\in F(\R^3)=X$.
It therefore suffices to show, for all $t\in\R$, that $\gamma(t)=\Phi_t^{P_0}(\sigma)$.
We have $\gamma(0)=F(\rho)=\sigma=\Phi_0^{P_0}(\sigma)$.
So, by uniqueness of solutions of ODEs,
it suffices to show, for all $t\in\R$, that $\gamma'(t)=P_0(\gamma(t))$.
For all $t\in\R$,
$$\gamma'(t)\,\,=\,\,[d/dt][F(\Phi_t^V(\rho))]\,\,=\,\,G([d/dt][\Phi_t^V(\rho)])\,\,=\,\,G(V(\Phi_t^V(\rho))).$$
Fix $t\in\R$ and define $\lambda:=\Phi_t^V(\rho)$.
Then we have both $\gamma'(t)=G(V(\lambda))$ and $P_0(\gamma(t))=P_0(F(\Phi_t^V(\rho)))=P_0(F(\lambda))$,
and so we wish to prove that $G(V(\lambda))=P_0(F(\lambda))$.

Let $\mu:=F(\lambda)$.
Then $\mu\in F(\R^3)=\Pi_3^{-1}(a)$, so $\Pi_3(\mu)=a$.
Then $\Pi_3(\mu)=a\in\R\backslash J_0\subseteq\R\backslash(50I_0)$,
so $\mu\in[\R^4]\backslash[(200I_0)^2\times(50I_0)^2]$.
So, by \lref{lem-P0-calculations}(vi), we have $\Pi_{34}(P_0(\mu))=(0,1)\in\R^2$,
so $\Pi_3(P_0(\mu))=0$.
Let $\nu:=P_0(\mu)=P_0(F(\lambda))$.
Then $\Pi_3(\nu)=\Pi_3(P_0(\mu))=0$,
and so $\nu\in\Pi_3^{-1}(0)=Y$.
Then $G(\pi(\nu))=\nu$.
By definition of $V$, we have $V(\lambda)=\pi(P_0(F(\lambda)))$.
That is, $V(\lambda)=\pi(\nu)$.

Then $G(V(\lambda))=G(\pi(\nu))=\nu=P_0(F(\lambda))$, as desired.
\end{proof}

\begin{cor}\wrlab{cor-Pi3inverse-of-J0-is-P0-invar}
The set $\Pi_3^{-1}(J_0)$ is $P_0$-invariant.
\end{cor}

\begin{proof}
For all $\sigma\in\Pi_3^{-1}(\R\backslash J_0)$,
for all $t\in\R$,
by \lref{lem-adapted-flow-goes-down},
we get $\Pi_3(\Phi_t^{P_0}(\sigma))=\Pi_3(\sigma)$,
so $\Pi_3(\Phi_t^{P_0}(\sigma))\in\R\backslash J_0$,
{\it i.e.}, $\Phi_t^{P_0}(\sigma)\in\Pi_3^{-1}(\R\backslash J_0)$.
This shows $\Pi_3^{-1}(\R\backslash J_0)$ is $P_0$-invariant.
Then $\R^4\backslash[\Pi_3^{-1}(\R\backslash J_0)]$ is also $P_0$-invariant.
So, since $\Pi_3^{-1}(J_0)=\R^4\backslash[\Pi_3^{-1}(\R\backslash J_0)]$,
we are done.
\end{proof}

\begin{lem}\wrlab{lem-never-home-fly-up}
Let $\rho\in\R^4$, $t_0\in\R$.
Then both of the following hold:
\begin{itemize}
\item[(i)]Say $\Phi_{[t_0,\infty)}^{P_0}(\rho)\subseteq[\R^4]\backslash[J_0^4]$.
Then $\displaystyle{\lim_{t\to\infty}\,\Pi_4(\Phi_t^{P_0}(\rho))=\infty}$.
\item[(ii)]Say $\Phi_{(-\infty,t_0]}^{P_0}(\rho)\subseteq[\R^4]\backslash[J_0^4]$.
Then $\displaystyle{\lim_{t\to-\infty}\,\Pi_4(\Phi_t^{P_0}(\rho))=-\infty}$.
\end{itemize}
\end{lem}

\begin{proof}
We will only prove (i). The proof of (ii) is similar.

Define $\lambda:\R\to\R$ by $\lambda(t)=\Pi_4(\Phi_t^{P_0}(\rho))$.
We wish to show that $\displaystyle{\lim_{t\to\infty}\,\lambda(t)=\infty}$.
Let $c:=(\lambda(t_0))-t_0$.
Because $\displaystyle{\lim_{t\to\infty}\,(c+t)=\infty}$,
it suffices to show, for all $t\ge t_0$, that $\lambda(t)=c+t$.
Since $\lambda(t_0)=c+t_0$, it suffices to show, for all $t\ge t_0$, that $\lambda'(t)=1$.
Differentiating the definition of $\lambda$, and using linearity of $\Pi_4$,
we conclude, for all $t\in\R$, that $\lambda'(t)=\Pi_4(P_0(\Phi_t^{P_0}(\rho)))$.
Fix $t\ge t_0$ and let $\rho_1:=\Phi_t^{P_0}(\rho)$.
We wish to show that $\Pi_4(P_0(\rho_1))=1$.

Since $t\ge t_0$, $\Phi_t^{P_0}(\rho)\in\Phi_{[t_0,\infty)}^{P_0}(\rho)$.
That is, $\rho_1\in\Phi_{[t_0,\infty)}^{P_0}(\rho)$.
Then
$$\rho_1\,\,\in\,\,\Phi_{[t_0,\infty)}^{P_0}(\rho)\,\,\subseteq\,\,[\R^4]\backslash[J_0^4]
\,\,\subseteq\,\,[\R^4]\,\backslash\,[(200I_0)^2\times(50I_0)^2].$$
So, by \lref{lem-P0-calculations}(vi), we conclude that $\Pi_{34}(P_0(\rho_1))=(0,1)\in\R^2$.
Then $\Pi_4(P_0(\rho_1))=1$, as desired.
\end{proof}

\begin{lem}\wrlab{lem-water-cant-leave-box-and-return}
Let $a>0$.
Let $I:=(-a,a)\subseteq\R$.
Let $s,u\in\R$ and assume that $s<u$.
Let $\rho\in\R^4$.
Assume that $\Phi_s^{P_0}(\rho),\Phi_u^{P_0}(\rho)\in I^2\times\R^2$.
Then $\Phi_{[s,u]}^{P_0}(\rho)\subseteq I^2\times\R^2$.
\end{lem}

\begin{proof}
Fix $t\in[s,u]$.
We wish to prove $\Phi_t^{P_0}(\rho)\in I^2\times\R^2$.
That is, we wish to prove
both $\Pi_1(\Phi_t^{P_0}(\rho))\in I$
and $\Pi_2(\Phi_t^{P_0}(\rho))\in I$.
We will only prove the former.
The proof of the latter is similar.

Since $t\le u$, by \lref{lem-Pi1-Pi2-of-P0-orbits-are-monotone}(i),
$|\Pi_1(\Phi_t^{P_0}(\rho))|\le|\Pi_1(\Phi_u^{P_0}(\rho))|$.
Since $\Phi_u^{P_0}(\rho)\in I^2\times\R^2$,
we have $\Pi_1(\Phi_u^{P_0}(\rho))\in I$,
so $|\Pi_1(\Phi_u^{P_0}(\rho))|<a$.
Then $|\Pi_1(\Phi_t^{P_0}(\rho))|\le|\Pi_1(\Phi_u^{P_0}(\rho))|<a$,
so $\Pi_1(\Phi_t^{P_0}(\rho))\in I$, as desired.
\end{proof}

We record the special cases $a=1$ and $a=200$ of \lref{lem-water-cant-leave-box-and-return}:

\begin{cor}\wrlab{cor-I0-water-cant-leave-box-and-return}
Let $s,u\in\R$ and say $s<u$.
Let $\rho\in\R^4$.
Assume that $\Phi_s^{P_0}(\rho),\Phi_u^{P_0}(\rho)\in I_0^2\times\R^2$.
Then $\Phi_{[s,u]}^{P_0}(\rho)\subseteq I_0^2\times\R^2$.
\end{cor}

\begin{cor}\wrlab{cor-J0-water-cant-leave-box-and-return}
Let $s,u\in\R$ and say $s<u$.
Let $\rho\in\R^4$.
Assume that $\Phi_s^{P_0}(\rho),\Phi_u^{P_0}(\rho)\in J_0^2\times\R^2$.
Then $\Phi_{[s,u]}^{P_0}(\rho)\subseteq J_0^2\times\R^2$.
\end{cor}

\begin{lem}\wrlab{lem-leave-home-go-up}
Let $s,u\in\R$ and assume that $s<u$.
Let $\rho\in\R^4$.
Assume that $\Phi_s^{P_0}(\rho),\Phi_u^{P_0}(\rho)\in\R^2\times J_0^2$.
Then $\Phi_{[s,u]}^{P_0}(\rho)\subseteq\R^2\times J_0^2$.
\end{lem}

\begin{proof}
Fix $t\in[s,u]$.
Say, for a contradiction, that
$\Phi_t^{P_0}(\rho)\notin\R^2\times J_0^2$.

Let $\sigma:=\Phi_s^{P_0}(\rho)$
and $\tau:=\Phi_t^{P_0}(\rho)$
and $\mu:=\Phi_u^{P_0}(\rho)$.
Then $\sigma\in\R^2\times J_0^2$
and $\tau\notin\R^2\times J_0^2$
and $\mu\in\R^2\times J_0^2$.
Let $q:=t-s$ and let $r:=u-t$.
Then $q\ge0$ and $r\ge0$.
Moreover, we have $\tau=\Phi_q^{P_0}(\sigma)$
and $\mu=\Phi_r^{P_0}(\tau)$.

As $\sigma\in\R^2\times J_0^2$,
we get $\Pi_3(\sigma)\in J_0$,
{\it i.e.}, $\sigma\in\Pi_3^{-1}(J_0)$.
By \cref{cor-Pi3inverse-of-J0-is-P0-invar},
the set $\Pi_3^{-1}(J_0)$ is $P_0$-invariant.
So, because $\sigma\in\Pi_3^{-1}(J_0)$,
we get $\Phi_q^{P_0}(\sigma)\in\Pi_3^{-1}(J_0)$,
{\it i.e.}, we get $\tau\in\Pi_3^{-1}(J_0)$.
Then $\Pi_3(\tau)\in J_0$.
So, since $\tau\notin\R^2\times J_0^2$,
we see that $\Pi_4(\tau)\notin J_0$.

As $\sigma\in\R^2\times J_0^2$,
we get $\Pi_4(\sigma)\in J_0=(-a_{J_0},a_{J_0})$,
so $\Pi_4(\sigma)>-a_{J_0}$.
So, by \cref{cor-above-or-below-means-nostop}(i),
$\Pi_4(\Phi_q^{P_0}(\sigma))>-a_{J_0}$,
{\it i.e.}, $\Pi_4(\tau)>-a_{J_0}$.
So, as $\Pi_4(\tau)\notin J_0=(-a_{J_0},a_{J_0})$,
we get $\Pi_4(\tau)\ge a_{J_0}$.
So, by \lref{lem-above-or-below-means-nostop}(i),
$\Pi_4(\Phi_r^{P_0}(\tau))\ge a_{J_0}$,
{\it i.e.}, $\Pi_4(\mu)\ge a_{J_0}$.
Then $\Pi_4(\mu)\notin(-a_{J_0},a_{J_0})=J_0$.
However $\mu\in\R^2\times J_0^2$, so $\Pi_4(\mu)\in J_0$,
contradiction.
\end{proof}

\begin{lem}\wrlab{lem-water-cant-leave-J0-box-and-return}
Let $s,u\in\R$ and assume that $s<u$.
Let $\rho\in\R^4$.
Assume that $\Phi_s^{P_0}(\rho),\Phi_u^{P_0}(\rho)\in J_0^4$.
Then $\Phi_{[s,u]}^{P_0}(\rho)\subseteq J_0^4$.
\end{lem}

\begin{proof}
As $J_0^4\subseteq J_0^2\times\R^2$ and $J_0^4\subseteq\R^2\times J_0^2$,
by \cref{cor-J0-water-cant-leave-box-and-return} and \lref{lem-leave-home-go-up},
we get $\Phi_{[s,u]}^{P_0}(\rho)\subseteq[J_0^2\times\R^2]\cap[\R^2\times J_0^2]=J_0^4$.
\end{proof}

\begin{lem}\wrlab{lem-P0-porous}
The map $P_0:\R^4\to\R^4$ is porous.
\end{lem}

\begin{proof}
As $Z$ is nowhere dense in $\R^4$,
it suffices to prove
$(\R^4)\backslash Z\subseteq\scru(P_0)$.
Fix $\rho\in(\R^4)\backslash Z$.
We wish to prove that $\rho\in\scru(P_0)$.

By \lref{lem-undeterred-criterion}(c\,$\Rightarrow$a),
it suffices to prove that
$\displaystyle{\lim_{t\to\infty}\,[\Pi_4(\Phi_t^{P_0}(\rho))]=\infty}$
and that
$\displaystyle{\lim_{t\to-\infty}\,[\Pi_4(\Phi_t^{P_0}(\rho))]=-\infty}$.
We will prove the former; the proof of the latter is similar.
By \lref{lem-never-home-fly-up}(i),
it suffices to show, for some~$t_0\in\R$,
that $\Phi_{[t_0,\infty)}^{P_0}(\rho)\subseteq[\R^4]\backslash[J_0^4]$.
We therefore assume, for a contradiction, that, for all $t\in\R$,
we have $[\Phi_{[t,\infty)}^{P_0}(\rho)]\cap[J_0^4]\ne\emptyset$.

We then fix a sequence $t_1,t_2,\ldots$ in $\R$
such that $t_1<t_2<\cdots$,
such that $t_1,t_2,\ldots\to\infty$,
and such that $\Phi_{t_1}^{P_0}(\rho),\Phi_{t_2}^{P_0}(\rho),\ldots\in J_0^4$.
Then, by \lref{lem-water-cant-leave-J0-box-and-return},
we have $\Phi_{[t_1,t_2]}^{P_0}(\rho),\Phi_{[t_2,t_3]}^{P_0}(\rho),\cdots\subseteq J_0^4$.
Taking the union, $\Phi_{[t_1,\infty)}^{P_0}(\rho)\subseteq J_0^4$.
Let $\sigma:=\Phi_{t_1}^{P_0}(\rho)$.
Then $\Phi_{[0,\infty)}^{P_0}(\sigma)=\Phi_{[t_1,\infty)}^{P_0}(\rho)\subseteq J_0^4$.

By \lref{lem-P0-calculations}(vii),
we see that $Z$ is $P_0$-invariant,
and it follows that $(\R^4)\backslash Z$ is $P_0$-invariant.
So, since $\rho\in(\R^4)\backslash Z$, we have $\sigma\in(\R^4)\backslash Z$.
Choose $u,v\in\R^2$ such that $\sigma=(u,v)\in\R^2\times\R^2=\R^4$.
Since $\sigma\notin Z$, we get $u\notin Z_0$.
Then $u\notin\{0\}\times\R$.
Let $\pi_1,\pi_2:\R^2\to\R$ be the coordinate projection maps
defined by $\pi_1(w,x)=w$ and $\pi_2(w,x)=x$.
By \lref{lem-asymptotic-coords-of-H}(i), we have $|\pi_1(\Phi_t^H(u))|\to\infty$ as $t\to\infty$.
Then fix a compact set $K\subseteq\R$ such that,
\begin{itemize}
\item[$(*)$]for all $t\in[0,\infty)\backslash K$, \qquad $|\pi_1(\Phi_t^H(u))|\ge a_{J_0}$.
\end{itemize}

Define $f:\R\to\R$ by $f(t)=\Omega_\sigma(t)$.
By \lref{lem-following-the-waterflow-of-P0}(ii and iv),
for all $t\ge0$, we have both $f'(t)\ge0$ and $f(t)\ge0$.
For all $t\ge0$, we have
$\Phi_t^{P_0}(\sigma)\in\Phi_{[0,\infty)}^{P_0}(\sigma)\subseteq J_0^4$,
so $\Pi_1(\Phi_t^{P_0}(\sigma))\in J_0$,
so $|\Pi_1(\Phi_t^{P_0}(\sigma))|<a_{J_0}$.
By \lref{lem-following-the-waterflow-of-P0}(v),
for all $t\in\R$, we have $\Pi_{12}(\Phi_t^{P_0}(\sigma))=\Phi_{f(t)}^H(u)$,
so $\Pi_1(\Phi_t^{P_0}(\sigma))=\pi_1(\Phi_{f(t)}^H(u))$.
Then, for all $t\ge0$, we have
$$|\,\pi_1(\Phi_{f(t)}^H(u))\,|\quad=\quad|\,\Pi_1(\Phi_t^{P_0}(\sigma))\,|\quad<\quad a_{J_0},$$
so, by $(*)$, $f(t)\in K$.
Then $f([0,\infty))\subseteq K$.
Then, by \lref{lem-fn-with-growth-is-proper} (with $a:=8$, $b:=1$),
fix $r\ge0$ such that both $f'(r)<1$ and $f'(r+8)<1$.

By \lref{lem-following-the-waterflow-of-P0}(ii),
for all $t\in\R$, we have $f'(t)=\omega(\Phi_t^{P_0}(\sigma))$.
Let $\sigma_0:=\Phi_r^{P_0}(\sigma)$ and
let $\sigma_1:=\Phi_8^{P_0}(\sigma_0)=\Phi_{r+8}^{P_0}(\sigma)$.
Then
\begin{itemize}
\item$\sigma_0,\,\sigma_1\quad\in\quad\Phi_{[0,\infty)}^{P_0}(\sigma)\quad\subseteq\quad J_0^4$ \qquad\qquad\qquad\qquad\qquad and
\item$\omega(\sigma_0)\quad=\quad\omega(\Phi_r^{P_0}(\sigma))\quad=\quad f'(r)\quad<\quad1$ \qquad\qquad\, and
\item$\omega(\sigma_1)\quad=\quad\omega(\Phi_{r+8}^{P_0}(\sigma))\quad=\quad f'(r+8)\quad<\quad1$.
\end{itemize}

Let $S_1:=J_0^2\times(4I_0)^2\subseteq\R^4$.
By \lref{lem-P0-calculations}(iv), we have $\omega=1$ on~$[J_0^4]\backslash[S_1]$.
So, since $\omega(\sigma_0)<1$ and $\sigma_0\in J_0^4$, we see that $\sigma_0\in S_1$.
Similarly, since $\omega(\sigma_1)<1$ and $\sigma_1\in J_0^4$, we see that $\sigma_1\in S_1$.

Let $v_0:=\Pi_{34}(\sigma_0)$ and $v_1:=\Pi_{34}(\sigma_1)$.
Since $\sigma_0,\sigma_1\in S_1$, we get $v_0,v_1\in(4I_0)^2$.
Fix $y,z\in4I_0$ such that $v_0=(y,z)\in\R\times\R=\R^2$.
We have $z\in4I_0=(-4,4)$.
Then $z>-4$, so $z+8>4$.
Then $z+8\notin4I_0$.
Then $\Phi_8^{Q_0}(v_0)=(y,z+8)\notin(4I_0)^2$.
Define $\lambda:\R\to\R^2$ by $\lambda(t)=\Pi_{34}(\Phi_t^{P_0}(\sigma_0))$,
and let $A:=\{t\in[0,8]\,|\,\lambda(t)=\Phi_t^{Q_0}(v_0)\}$.

Since $v_1\in(4I_0)^2$ and $\Phi_8^{Q_0}(v_0)\notin(4I_0)^2$,
we see that $v_1\ne\Phi_8^{Q_0}(v_0)$.
Then $\lambda(8)=\Pi_{34}(\Phi_8^{P_0}(\sigma_0))=\Pi_{34}(\sigma_1)=v_1\ne\Phi_8^{Q_0}(v_0)$,
and so $8\notin A$.
Let $s:=\inf[0,8]\backslash A$.
Then $[0,s)\subseteq A$ and, moreover,
\begin{itemize}
\item[$(**)$]for all $\delta>0$, we have $[s,s+\delta)\not\subseteq A$.
\end{itemize}
As $\lambda(0)=\Pi_{34}(\sigma_0)=v_0=\Phi_0^{Q_0}(v_0)$, we get $0\in A$.
So, since $A$ is closed in~$[0,8]$ and $[0,s)\subseteq A$,
we have $s\in A$, {\it i.e.}, we have $\lambda(s)=\Phi_s^{Q_0}(v_0)$.

As  $z\in(-4,4)$ and $s\in[0,8]$,
we get $z+s\in(-4,12)\subseteq(-12,12)$,
so $z+s\in12I_0$.
Then $\Phi_s^{Q_0}(v_0)=(y,z+s)\in(4I_0)\times(12I_0)=S$.
By definition of $\lambda$, we have $\lambda(s)=\Pi_{34}(\Phi_s^{P_0}(\sigma_0))$.
Then
$$\Pi_{34}(\Phi_s^{P_0}(\sigma_0))\quad=\quad\lambda(s)\quad=\quad\Phi_s^{Q_0}(v_0)\quad\in\quad S,$$
so $\Phi_s^{P_0}(\sigma_0)\in\Pi_{34}^{-1}(S)=\R^2\times S$.
Then, because $\R^2\times S$ is open in $\R^4$, by continuity, choose $\delta_1>0$
such that $\Phi_{[s,s+\delta_1)}^{P_0}(\sigma_0)\subseteq\R^2\times S$.

{\it Claim:}
For all $t\in[s,s+\delta_1)$, we have $\lambda(t)=\Phi_t^{Q_0}(v_0)$.
{\it Proof of claim:}
Since $\lambda(s)=\Phi_s^{Q_0}(v_0)$,
by uniqueness of solutions of ODEs,
it suffices to prove, for all $t\in[s,s+\delta_1)$,
that $\lambda'(t)=Q_0(\lambda(t))$.

Differentiating the definition of $\lambda$, and using linearity of $\Pi_{34}$,
we see, for all $t\in\R$, that $\lambda'(t)=\Pi_{34}(P_0(\Phi_t^{P_0}(\sigma_0)))$.
Let $w:=(0,1)\in\R^2$.
Then, by definition of~$Q_0$, for all $\tau\in\R^4$, we have $Q_0(\tau)=w$.
Fix $t\in[s,s+\delta_1)$.
We wish to prove that $\Pi_{34}(P_0(\Phi_t^{P_0}(\sigma_0)))=w$.

We have $\Phi_t^{P_0}(\sigma_0)\in\Phi_{[s,s+\delta_1)}^{P_0}(\sigma_0)\subseteq\R^2\times S$.
Then, by \lref{lem-P0-calculations}(v),
we see that $\Pi_{34}(P_0(\Phi_t^{P_0}(\sigma_0)))=w$, as desired.
{\it End of proof of claim.}

Since $s\in A$ and $8\notin A$, we see that $s\ne8$.
So, as $s\in[0,8]$, we get $s<8$.
Let $\delta_0:=\min\{\delta_1,8-s\}$.
Then $\delta_0>0$.
Also, $s+\delta_0\le8$.
Then $[s,s+\delta_0)\subseteq[0,8]$.
As $\delta_0\le\delta_1$, by the claim,
we see, for all $t\in[s,s+\delta_0)$, that $\lambda(t)=\Phi_t^{Q_0}(v_0)$.
Then $[s,s+\delta_0)\subseteq A$, contradicting $(**)$.
\end{proof}

Recall that $\displaystyle{T_\times:=m-\left[\int_{-4}^4\,(\zeta_{3I_0}(t))\,dt\right]}$.

\begin{lem}\wrlab{lem-following-the-runners-for-m}
Let $\rho\in\R^2\times(3I_0)^2$.
Assume $\Phi_{[0,m]}^{P_0}(\rho)\subseteq(\,\overline{100I_0}\,)^4$.
Then $\Pi_{12}(\Phi_m^{P_0}(\rho))=\Phi_{T_\times}^H(\Pi_{12}(\rho))$
and $\Pi_{34}(\Phi_m^{P_0}(\rho))=\Phi_m^Q(\Pi_{34}(\rho))$.
\end{lem}

\begin{proof}
Define $f:\R\to\R$ by $f(t)=\Omega_\rho(t)$.
Choose $u_0\in\R^2$ and $v_0\in(3I_0)^2$ such that $\rho=(u_0,v_0)\in\R^2\times\R^2=\R^4$.
Let $J:=[0,m]$.
Then we have $\Phi_J^{P_0}(\rho)=\Phi_{[0,m]}^{P_0}(\rho)\subseteq(\,\overline{100I_0}\,)^4$.

By \lref{lem-following-the-waterflow-of-P0}(v) and
\lref{lem-following-the-runners-of-P0-in-J04},
for all $t\in J$, we have both
$$\Pi_{12}(\Phi_t^{P_0}(\rho))=\Phi_{f(t)}^H(u_0)\qquad\hbox{and}\qquad
\Pi_{34}(\Phi_t^{P_0}(\rho))=\Phi_t^Q(v_0).$$
Therefore we have both $\Pi_{12}(\Phi_m^{P_0}(\rho))=\Phi_{f(m)}^H(u_0)=\Phi_{f(m)}^H(\Pi_{12}(\rho))$
and $\Pi_{34}(\Phi_m^{P_0}(\rho))=\Phi_m^Q(v_0)=\Phi_m^Q(\Pi_{34}(\rho))$.
We need only prove $f(m)=T_\times$.

For all $t\in J$, $\Phi_t^{P_0}(\rho)\in\Phi_J^{P_0}(\rho)\subseteq(\,\overline{100I_0}\,)^4$,
so, by \lref{lem-P0-calculations}(iii),
we conclude that $\omega(\Phi_t^{P_0}(\rho))=\alpha(\Pi_{34}(\Phi_t^{P_0}(\rho)))$.
Then, for all $t\in J$,
we have $\omega(\Phi_t^{P_0}(\rho))=\alpha(\Pi_{34}(\Phi_t^{P_0}(\rho)))=\alpha(\Phi_t^Q(v_0))$.
Then
$$f(m)\,\,=\,\,\Omega_\rho(m)\,\,=\,\,\int_0^m\,[\omega(\Phi_t^{P_0}(\rho))]\,dt\,\,=\,\,
\int_0^m\,[\alpha(\Phi_t^Q(v_0))]\,dt.$$
Define $g:\R\to\R$ by $g(t)=1-[\alpha(\Phi_t^Q(v_0))]$.
Then
$$f(m)\quad=\quad\int_0^m\,\left[1-(g(t))\right]\,dt\quad=\quad m\,\,-\,\,\left[\int_0^m\,(g(t))\,dt\right].$$

Choose $y_0,z_0\in3I_0$ such that $v_0=(y_0,z_0)\in\R\times\R=\R^2$.
Since $v_0\in(3I_0)^2\subseteq(4I_0)\times(12I_0)=S\subseteq R$,
by definition of $m$, we get $\Phi_m^Q(v_0)=v_0$.
Then, for all $t\in\R$, $\Phi_{t+m}^Q(v_0)=\Phi_t^Q(\Phi_m^Q(v_0))=\Phi_t^Q(v_0)$,
so $g(t+m)=g(t)$.
That is, $g:\R\to\R$ is $m$-periodic.
It follows that
\begin{itemize}
\item[]the integral of $g$ over any interval in $\R$ of length $m$
\end{itemize}
is equal to
\begin{itemize}
\item[]the integral of $g$ over any other interval in $\R$ of length $m$.
\end{itemize}
In particular,
$\displaystyle{\int_0^m\,(g(t))\,dt=\int_{-4-z_0}^{m-4-z_0}\,(g(t))\,dt}$.
By invariance of integration under translation,
$\displaystyle{\int_{-4-z_0}^{m-4-z_0}\,(g(t))\,dt=\int_{-4}^{m-4}\,(g(t-z_0))\,dt}$.

We have $y_0\in3I_0\subseteq4I_0$ and $z_0\in3I_0\subseteq12I_0$ and $v_0=(y_0,z_0)$.
So, for all $t\in[4,m-4]$,
by \cref{cor-reentry-time-to-4I02}, $\Phi_{t-z_0}^Q(v_0)\in(\R^2)\backslash[(4I_0)^2]$,
and so $\alpha(\Phi_{t-z_0}^Q(v_0))=1$,
and so $g(t-z_0)=1-1=0$.
We conclude that $\displaystyle{\int_4^{m-4}\,(g(t-z_0))\,dt=0}$.
Then $\displaystyle{\int_{-4}^{m-4}\,(g(t-z_0))\,dt=\int_{-4}^4\,(g(t-z_0))\,dt}$.

Recall that $\alpha:\R^2\to\R$ is defined by
$\alpha(y,z)=1-[\zeta_{3I_0}(y)][\zeta_{3I_0}(z)]$.
As $y_0\in3I_0$, we have $\zeta_{3I_0}(y_0)=1$.
For all $t\in[-4,4]$,
by \lref{lem-follow-Q-in-and-out-of-S}(i),
we have $\Phi_{t-z_0}^Q(v_0)=\Phi_{t-z_0}^{Q_0}(v_0)=(y_0,t)$,
and so
$$\alpha(\,\Phi_{t-z_0}^Q(v_0)\,)\,\,\,=\,\,\,1\,-\,[\zeta_{3I_0}(y_0)][\zeta_{3I_0}(t)]\,\,\,=\,\,\,
1\,-\,[1][\zeta_{3I_0}(t)],$$
and so $g(t-z_0)=\zeta_{3I_0}(t)$.
Then $\displaystyle{\int_{-4}^4\,(g(t-z_0))\,dt=\int_{-4}^4\,(\zeta_{3I_0}(t))\,dt}$.

Putting all these observations together, we conclude that
\begin{eqnarray*}
\int_0^m(g(t))\,dt&=&\int_{-4-z_0}^{m-4-z_0}\,(g(t))\,dt\,=\,\int_{-4}^{m-4}\,(g(t-z_0))\,dt\\
&=&\int_{-4}^4\,(g(t-z_0))\,dt\,=\,\int_{-4}^4\,(\zeta_{3I_0}(t))\,dt.
\end{eqnarray*}
Then
$\displaystyle{f(m)=m-\left[\int_0^m\,(g(t))\,dt\right]=m-\left[\int_{-4}^4\,(\zeta_{3I_0}(t))\,dt\right]=T_\times}$.
\end{proof}

\begin{lem}\wrlab{lem-P0-per}
$\Phi_m^{P_0}$ agrees with $\Id_4$ to all orders at $\xi_{I_0}$.
\end{lem}

\begin{proof}
It suffices to show both of the following:
\begin{itemize}
\item[(i)]$\Pi_{12}\circ\Phi_m^{P_0}$ agrees with $\Pi_{12}$ to all orders at $\xi_{I_0}$ \qquad and
\item[(ii)]$\Pi_{34}\circ\Phi_m^{P_0}$ agrees with $\Pi_{34}$ to all orders at $\xi_{I_0}$.
\end{itemize}

Recall that $R\subseteq\R^2$ is the racetrack of \secref{sect-racetrack}
and that $\bfzero:\R^2\to\R^2$ is the zero map defined by $\bfzero(w,x)=(0,0)$.
Recall that the racetrack $R$ is an open subset of $\R^2$
and that $(4I_0)\times(12I_0)=S\subseteq R\subseteq(50I_0)^2\subseteq\R^2$.

Let $R_1:=\{(0,0)\}\times R\subseteq\R^2\times\R^2=\R^4$.
Because $(0,-1)\in S\subseteq R$ and $\xi_{I_0}=(0,0,0,-1)$,
we get $\xi_{I_0}\in R_1$.
Also, because $R\subseteq(50I_0)^2$,
we conclude that $R_1\subseteq(50I_0)^4\subseteq(100I_0)^4\subseteq(\,\overline{100I_0}\,)^4$.

Define $Q_1:\R^4\to\R^4$ by: \qquad for all $\rho=(u,v)\in\R^2\times\R^2=\R^4$,
$$Q_1(u,v)\quad=\quad(\,\,\bfzero(u)\,\,,\,\,Q(v)\,\,)\quad\in\quad\R^2\,\times\,\R^2\quad=\quad\R^4.$$
From the construction in \secref{sect-racetrack}, $R$ is $Q$-invariant,
so $R_1$ is $Q_1$-invariant.
So, since $\xi_{I_0}\in R_1$, we conclude that $\Phi_\R^{Q_1}(\xi_{I_0})\subseteq R_1$.

For $\rho=(u,v)\in\R^2\times\R^2=\R^4$,
if $\rho\in R_1$, then $u=(0,0)$, so $H_0(u)=(0,0)$,
so $H(u)=(0,0)$, so $[\alpha(v)][H(u)]=(0,0)=\bfzero(u)$, so
$$P_*(\rho)\,\,\,=\,\,\,(\,\,[\alpha(v)][H(u)]\,\,,\,\,Q(v)\,\,)\,\,\,=\,\,\,
(\,\,\bfzero(u)\,\,,\,\,Q(v)\,\,)\,\,\,=\,\,\,Q_1(\rho).$$
Thus $P_*=Q_1$ on $R_1$.
Also, because $R_1\subseteq(\,\overline{100I_0}\,)^4\,$,
we see, by \lref{lem-P0-calculations}(i),
that $P_0=P_*$ on $R_1$.
Then $P_0=P_*=Q_1$ on $R_1$.
So, as $\Phi_\R^{Q_1}(\xi_{I_0})\subseteq R_1$,
by \lref{lem-orbits-agree}, we get $\Phi_{[0,m]}^{P_0}(\xi_{I_0})=\Phi_{[0,m]}^{Q_1}(\xi_{I_0})$.
Then $\Phi_{[0,m]}^{P_0}(\xi_{I_0})=\Phi_{[0,m]}^{Q_1}(\xi_{I_0})\subseteq R_1\subseteq(100I_0)^4$.
So, by continuity, fix an open neighborhood $N_0$ in $\R^4$ of $\xi_{I_0}$
such that $\Phi_{[0,m]}^{P_0}(N_0)\subseteq(100I_0)^4$.
Let $N:=N_0\cap[\R^2\times(3I_0)^2]$.
Then $N$ is also an open neighborhood in~$\R^4$ of $\xi_{I_0}$
and, moreover, $\Phi_{[0,m]}^{P_0}(N)\subseteq\Phi_{[0,m]}^{P_0}(N_0)\subseteq(100I_0)^4$.
Then, by \lref{lem-following-the-runners-for-m}, we see, for all $\rho\in N$,
that $\Pi_{12}(\Phi_m^{P_0}(\rho))=\Phi_{T_\times}^H(\Pi_{12}(\rho))$
and that $\Pi_{34}(\Phi_m^{P_0}(\rho))=\Phi_m^Q(\Pi_{34}(\rho))$.
That is, on $N$, we have both
$$\Pi_{12}\circ\Phi_m^{P_0}\,=\,\Phi_{T_\times}^H\circ\Pi_{12}
\qquad\hbox{and}\qquad
\Pi_{34}\circ\Phi_m^{P_0}\,=\,\Phi_m^Q\circ\Pi_{34}.$$
So, since $N$ is an open  neighborhood of $\xi_{I_0}$, we have both
\begin{itemize}
\item[$(*)$]$\Pi_{12}\,\circ\,\Phi_m^{P_0}$ agrees with $\Phi_{T_\times}^H\,\circ\,\Pi_{12}$ to all orders at $\xi_{I_0}$
\qquad and
\item[$(**)$]$\Pi_{34}\,\circ\,\Phi_m^{P_0}$ agrees with $\Phi_m^Q\,\circ\,\Pi_{34}$ to all orders at $\xi_{I_0}$.
\end{itemize}

By \lref{lem-H0-fixes-00-to-all-orders}, $\Phi_{T_\times}^H$ agrees with $\Id_2$ to all orders at~$(0,0)$.
So, since $\Pi_{12}(\xi_{I_0})=(0,0)$, we see that
$\Phi_{T_\times}^H\circ\Pi_{12}$ agrees with $\Pi_{12}$ to all orders at~$\xi_{I_0}$.
So (i) follows from $(*)$.
It remains to prove (ii).

Recall that $(0,-1)\in R$.
Then $R$ is an open neighborhood in $\R^2$ of~$(0,-1)$.
By definition of $m$, for all $v\in R$, we have $\Phi_m^Q(v)=v$.
That is, $\Phi_m^Q=\Id_2$ on $R$.
Then $\Phi_m^Q$ agrees with $\Id_2$ to all orders at $(0,-1)$.
So, since $\Pi_{34}(\xi_{I_0})=(0,-1)$, we conclude that
$\Phi_m^Q\circ\Pi_{34}$ agrees with $\Pi_{34}$ to all orders at~$\xi_{I_0}$.
So (ii) follows from $(**)$.
\end{proof}

\begin{lem}\wrlab{lem-no-early-return}
Let $\tau\in B_\circ(I_0)$ and let $s\in(0,m)$.
Assume that $\Phi_s^{P_0}(\tau)\in I_0^4$.
Then $s<2$.
\end{lem}

\begin{proof}
From the construction of $S$ and $R$ in \secref{sect-racetrack},
we have
$$(4I_0)\,\times\,(12I_0)\quad=\quad S\quad\subseteq\quad R\quad\subseteq\quad(50I_0)^2.$$
Then $I_0\times\{-1\}\subseteq(4I_0)\times(12I_0)\subseteq R$, so
$I_0^3\times\{-1\}\subseteq I_0^2\times R$.
Let $C$ be the closure in $\R^4$ of~$I_0^2\times R$.
Then $B_\circ(I_0)=I_0^3\times\{-1\}\subseteq I_0^2\times R\subseteq C$.
Then $\tau\in B_\circ(I_0)\subseteq C$.
Because $I_0^2\times R\subseteq I_0^2\times(50I_0)^2\subseteq(50I_0)^4$,
we get $C\subseteq(\,\overline{50I_0}\,)^4\subseteq(100I_0)^4$.
Then $C\cup[(100I_0)^4]=(100I_0)^4$.

By assumption, $\tau\in B_\circ(I_0)$ and $\Phi_s^{P_0}(\tau)\in I_0^4$.
Then
$$\Phi_0^{P_0}(\tau)\quad=\quad\tau\quad\in\quad
B_\circ(I_0)\quad\subseteq\quad I_0^2\,\times\,R\quad\subseteq\quad I_0^2\,\times\,\R^2$$
and $\Phi_s^{P_0}(\tau)\in I_0^4\subseteq I_0^2\times\R^2$.
So, by \cref{cor-I0-water-cant-leave-box-and-return},
$\Phi_{[0,s]}^{P_0}(\tau)\subseteq I_0^2\times\R^2$.

{\it Claim:} $\Phi_{[0,s]}^{P_0}(\tau)\subseteq C$.
{\it Proof of claim:} Define
$$A\quad:=\quad\{\,\,t\in[0,s]\,\,|\,\,\Phi_t^{P_0}(\tau)\in C\,\,\}.$$
Assume, for a contradiction, that $A\subsetneq[0,s]$.

Let $t_0:=\inf[0,s]\backslash A$.
Then $[0,t_0)\subseteq A$ and,
\begin{itemize}
\item[$(*)$]for all $\delta>0$, we have $[t_0,t_0+\delta)\not\subseteq A$.
\end{itemize}
Because $\tau\in C$, we see that $0\in A$.
So, since $A$ is closed in $[0,s]$ and since $[0,t_0)\subseteq A$,
we get $[0,t_0]\subseteq A$.
Then $[0,t_0]\subseteq A\subsetneq[0,s]$.
Then $t_0<s$.
As $t_0\in[0,t_0]\subseteq A$,
we conclude that $\Phi_{t_0}^{P_0}(\tau)\in C$.
Then $\Phi_{t_0}^{P_0}(\tau)\in C\subseteq(100I_0)^4$.
By continuity, fix $\delta_1>0$ such that
$\Phi_{[t_0,t_0+\delta_1)}^{P_0}(\tau)\subseteq(100I_0)^4$.
We have $[0,t_0]\subseteq A$, so $\Phi_{[0,t_0]}^{P_0}(\tau)\subseteq C$.
Then $\Phi_{[0,t_0+\delta_1)}^{P_0}(\tau)
=[\Phi_{[0,t_0)}^{P_0}(\tau)]\cup[\Phi_{[t_0,t_0+\delta_1)}^{P_0}(\tau)]
\subseteq C\cup[(100I_0)^4]=(100I_0)^4$.
By \lref{lem-P0-calculations}(i), we have $P_*=P_0$ on $(100I_0)^4$.
Then, by \lref{lem-orbits-agree}, we have
$\Phi_{[0,t_0+\delta_1)}^{P_*}(\tau)=\Phi_{[0,t_0+\delta_1)}^{P_0}(\tau)$.

By \lref{lem-P0-calculations}(viii),
we see that $\Phi_\R^{P_*}(\R^2\times R)=\R^2\times R$.
Also, $\tau\in I_0^2\times R\subseteq\R^2\times R$.
Let $\delta_0:=\min\{\delta_1,s-t_0\}$.
Then $\delta_0>0$ and
$\Phi_{[t_0,t_0+\delta_0)}^{P_0}(\tau)\subseteq\Phi_{[0,t_0+\delta_1)}^{P_0}(\tau)=
\Phi_{[0,t_0+\delta_1)}^{P_*}(\tau)\subseteq
\Phi_\R^{P_*}(\R^2\times R)=\R^2\times R$.

Recall that $\Phi_{[0,s]}^{P_0}(\tau)\subseteq I_0^2\times\R^2$.
Also, $[t_0,t_0+\delta_0)\subseteq[0,s]$.
Then $\Phi_{[t_0,t_0+\delta_0)}^{P_0}(\tau)\subseteq\Phi_{[0,s]}^{P_0}(\tau)\subseteq I_0^2\times\R^2$.
Then
$$\Phi_{[t_0,t_0+\delta_0)}^{P_0}(\tau)\,\,\,\subseteq\,\,\,
[I_0^2\times\R^2]\,\cap\,[\R^2\times R]\,\,\,=\,\,\,I_0^2\,\times\,R\,\,\,\subseteq\,\,\,C.$$
Then $[t_0,t_0+\delta_0)\subseteq A$, contradicting $(*)$.
{\it End of proof of claim.}

By assumption, $\tau\in B_\circ(I_0)$ and $\Phi_s^{P_0}(\tau)\in I_0^4$.
Fix $u,v\in\R^2$ such that $\tau=(u,v)\in\R^2\times\R^2=\R^4$.
By the claim, $\Phi_{[0,s]}^{P_0}(\tau)\subseteq C$.
Then $\Phi_{[0,s]}^{P_0}(\tau)\subseteq C\subseteq(100I_0\,)^4$.
Then, by \lref{lem-following-the-runners-of-P0-in-J04},
we have $\Pi_{34}(\Phi_s^{P_0}(\tau))=\Phi_s^Q(v)$.
Then $\Phi_s^Q(v)=\Pi_{34}(\Phi_s^{P_0}(\tau))\in\Pi_{34}(I_0^4)=I_0^2$.
Also, $v=\Pi_{34}(\tau)\in\Pi_{34}(B_\circ(I_0))=\Pi_{34}(I_0^3\times\{-1\})=I_0\times\{-1\}$.
So let $z:=-1$ and choose $y\in I_0$ such that $v=(y,z)\in\R\times\R=\R^2$.

Let $q:=s-1$.
Then $\Phi_{q-z}^Q(v)=\Phi_{q+1}^Q(v)=\Phi_s^Q(v)\in I_0^2$.
By \cref{cor-reentry-time-to-I02}, we know,
for all $t\in[1,m-1]$, that $\Phi_{t-z}^Q(v)\notin I_0^2$.
Then $q\notin[1,m-1]$,
so $s\notin[2,m]$.
So, since $s\in(0,m)$, we get $s<2$.
\end{proof}

\begin{lem}\wrlab{lem-P0K0-in-scrp}
We have $(P_0,K_0)\in\scrp_{I_0}$.
\end{lem}

\begin{proof}
Recall that $P_0:\R^4\to\R^4$ is $C^\infty$ and $P_0(\R^4)\subseteq\barI_0^4$.
Then $P_0\in\scrc$.
On $(\R^4)\backslash(K_0^4)$,
we have $\gamma=0$, so $P_0=\gamma P_++(1-\gamma)V_0=V_0$.
Then $(P_0,K_0)\in\scrd$.
By \lref{lem-P0-calculations}(ii),
$P_0=V_0$ on $(\,\overline{3I_0}\,)^4$.
Then $P_0\in\scrc_{I_0}$.
So, since $4I_0\subseteq400I_0=K_0$,
we get $(P_0,K_0)\in\scrd_{I_0}^\times$.
As $m>24$, we get $m>2=2a_{I_0}$.
By \lref{lem-P0-per}, $\Phi_m^{P_0}$ agrees with $\Id_4$ to all orders at~$\xi_{I_0}$.
It remains to show, for all $\tau\in B_\circ(I_0)$, for all $t\in(0,m)$, that
$$[\,\Phi_t^{P_0}(\tau)\in I_0^4\,]\qquad\Leftrightarrow\qquad[\,t<2\,].$$
\lref{lem-no-early-return} yields $\Rightarrow$,
and it remains to prove $\Leftarrow$.
Fix $\tau\in B_\circ(I_0)$.
We wish to prove that $\Phi_{(0,2)}^{P_0}(\tau)\subseteq I_0^4$.

As $\tau\in B_\circ(I_0)=I_0^3\times\{-1\}$, we get
$\Phi_{(0,2)}^{V_0}(\tau)\subseteq I_0^3\times(-1,1)=I_0^4$ and
$$\Phi_{[0,2]}^{V_0}(\tau)\quad\subseteq\quad I_0^3\,\times\,[-1,1]\quad\subseteq\quad
\barI_0^4\quad\subseteq\quad(\,\overline{3I_0}\,)^4.$$
So, since $P_0=V_0$ on $(\,\overline{3I_0}\,)^4$,
by \lref{lem-orbits-agree}, for all $t\in[0,2]$,
we have $\Phi_t^{P_0}(\tau)=\Phi_t^{V_0}(\tau)$.
Then $\Phi_{(0,2)}^{P_0}(\tau)=\Phi_{(0,2)}^{V_0}(\tau)\subseteq I_0^4$, as desired.
\end{proof}

\section{Results about the exchange operation $\scrx_I(P,V)$\wrlab{sect-exchange-results}}

Recall, from \secref{sect-dtuflow},
the definitions of $DF_I^V$, $UF_I^V$ and $TF_I^V$.
Recall, from \secref{sect-gettoD},
the definitions of $\scrd_+$, $\scrd^\#$, $\scrd_+^\#$ and $\scrd_*$.
The definitions of~ $\scrd_I^\times$, $\scrx_I(P,v)$ and $\scrp_I$
are all found in \secref{sect-notation}.

The following fact will be used repeatedly throughout this section:
If $(V,I)\in\scrd$, if $(P,K)\in\scrd_I^\times$ and if $X=\scrx_I(P,V)$, then
\begin{itemize}
\item$4I\subseteq K$,
\item$P=V_0$ on $(\R^4)\backslash(K^4)$,
\item$X=P$ on $(\R^4)\backslash(I^4)$,
\item$V=V_0$ on $(\R^4)\backslash(I^4)$,
\item$P=V_0$ on $(\,\overline{3I}\,)^4$ \qquad and
\item$X=V$ on $(\,\overline{3I}\,)^4$,
\end{itemize}
and, consequently,
\begin{itemize}
\item$X=P=V=V_0$ on $(\R^4)\backslash(K^4)$ \qquad and
\item$X=P=V=V_0$ on $[(\,\overline{3I}\,)^4]\backslash[I^4]$.
\end{itemize}

\begin{lem}\wrlab{lem-exch-gives-mod}
Let $(V,I)\in\scrd$.
Let $(P,K)\in\scrd_I^\times$.
Let $X:=\scrx_I(P,V)$.
Then $(X,K)\in\scrm(V,I)$.
\end{lem}

\begin{proof}
On $(\R^4)\backslash(K^4)$, we have $X=V_0$.
Then $(X,K)\in\scrd$.
We have $4I\subseteq K$, so $a_{4I}\le a_K$.
Then $a_I<4a_I=a_{4I}\le a_K$.
Also, we have $X=V$ on~$\barI^4$.
Then $(X,K)\in\scrm(V,I)$, as desired.
\end{proof}

\begin{lem}\wrlab{lem-make-por-per}
Let $(V,I)\in\scrd_+$
and let $(P,K)\in\scrd_I^\times$.
Assume that both~$V$ and $P$ are porous.
Let $X:=\scrx_I(P,V)$.
Then $X$ is porous.
\end{lem}

\begin{proof}
Let
\begin{itemize}
\item$S_1:=(\R^4)\backslash(\scru(P))$, \quad $U_1:=(\R^4)\backslash(\,\barI^4\,)$\qquad and
\item$S_2:=(\R^4)\backslash(\scru(V))$, \quad $U_2:=(2I)^4$.
\end{itemize}
Then
\begin{itemize}
\item(\,$S_1$ is $P$-invariant\,) \quad and \quad (\,$X=P$ on $U_1$\,) \qquad and
\item(\,$S_2$ is $V$-invariant\,) \quad and \quad (\,$X=V$ on $U_2$\,).
\end{itemize}
Then, by \lref{lem-locally-invar-criterion},
$Z_1:=S_1\cap U_1$ and $Z_2:=S_2\cap U_2$
are both locally $X$-invariant.
Since both $V$ and $P$ are porous, by \cref{cor-porous-dense-comeager},
we see that both $S_1$ and $S_2$ are meager in $\R^4$.
Then both $Z_1$ and $Z_2$ are meager in~$\R^4$.
Then, by \lref{lem-nw-dense-saturation},
both $Z'_1:=\Phi_\R^X(Z_1)$ and $Z'_2:=\Phi_\R^X(Z_2)$ are meager in~$\R^4$.
Then $Z':=Z'_1\cup Z'_2$ is $X$-invariant and meager in~$\R^4$.

By \cref{cor-porous-dense-comeager},
we wish to show that $\scru(X)$ is comeager in $\R^4$.
It therefore suffices to prove that $(\R^4)\backslash Z'\subseteq\scru(X)$.
Fix $\rho\in\R^4$, and assume that $\rho\notin Z'$.
We wish to show that $\rho\in\scru(X)$.

{\it Claim 1:}
$\Phi_\R^X(\rho)\not\subseteq\barI^4$.
{\it Proof of Claim 1:}
Assume, for a contradiction, that $\Phi_\R^X(\rho)\subseteq\barI^4$.
Then, since $V=X$ on $\barI^4$, by \lref{lem-orbits-agree},
we get $\Phi_\R^V(\rho)=\Phi_\R^X(\rho)$.
Then $\Phi_\R^V(\rho)\subseteq\barI^4$,
so $\Pi_4(\Phi_\R^V(\rho))\subseteq\Pi_4(\,\barI^4\,)=\barI$.

We have $\rho\in\Phi_\R^X(\rho)\subseteq\barI^4\subseteq(2I)^4=U_2$.
Because $\rho\notin Z'$, we know that $\rho\notin Z'_2$,
so $\rho\notin Z_2=S_2\cap U_2$.
Then $\rho\notin S_2$.
That is, $\rho\in\scru(V)$.
Then $\R=\Pi_4(\Phi_\R^V(\rho))\subseteq\barI$, contradiction.
{\it End of proof of Claim 1.}

By Claim 1, fix $\rho_0\in\Phi_\R^X(\rho)$ such that $\rho_0\notin\barI^4$.
Since $Z'$ is $X$-invariant, since $\rho\notin Z'$ and since $\rho_0\in\Phi_\R^X(\rho)$,
we see that $\rho_0\notin Z'$.
Since $\scru(X)$ is $X$-invariant,
it suffices to show that $\rho_0\in\scru(X)$.

On $(\R^4)\backslash(K^4)$, we have $X=V_0$.
Therefore $(X,K)\in\scrd$.
Then, by \lref{lem-undeterred-criterion}(b\,$\Rightarrow$a),
it suffices to show that
$$(-\infty,-a_K)\,\cap\,[\Pi_4(\Phi_\R^X(\rho_0))]
\,\,\ne\,\,\emptyset\,\,\ne\,\,(a_K,\infty)\,\cap\,[\Pi_4(\Phi_\R^X(\rho_0))].$$
Because $\Pi_4(I^4)=I\subseteq4I\subseteq K=(-a_K,a_K)$, we get
$$(-\infty,-a_K)\,\cap\,[\Pi_4(I^4)]\quad=\quad\emptyset\quad=\quad
(a_K,\infty)\,\cap\,[\Pi_4(I^4)].$$
Let $R:=(\Phi_\R^X(\rho_0))\cup(I^4)$.
Then $\Pi_4(R)=[\Pi_4(\Phi_\R^X(\rho_0))]\cup[\Pi_4(I^4)]$, so
\begin{eqnarray*}
(-\infty,-a_K)\,\cap\,[\Pi_4(R)]&=&(-\infty,-a_K)\,\cap\,[\Pi_4(\Phi_\R^X(\rho_0))]\qquad\hbox{and}\\
(a_K,\infty)\,\cap\,[\Pi_4(R)]&=&(a_K,\infty)\,\cap\,[\Pi_4(\Phi_\R^X(\rho_0))].
\end{eqnarray*}
It therefore suffices to prove that
$$(-\infty,-a_K)\,\cap\,[\Pi_4(R)]
\quad\ne\quad\emptyset\quad\ne\quad(a_K,\infty)\,\cap\,[\Pi_4(R)].$$

We have $\rho_0\notin Z'$, so $\rho_0\notin Z'_1$, so $\rho_0\notin Z_1=S_1\cap U_1$.
So, since $\rho_0\in(\R^4)\backslash(\,\barI^4\,)=U_1$, we see that $\rho_0\notin S_1$.
That is, $\rho_0\in\scru(P)$.
Then $\Pi_4(\Phi_\R^P(\rho_0))=\R$, and it follows that
$$(-\infty,-a_K)\,\cap\,[\Pi_4(\Phi_\R^P(\rho_0))]
\,\,\ne\,\,\emptyset\,\,\ne\,\,(a_K,\infty)\,\cap\,[\Pi_4(\Phi_\R^P(\rho_0))].$$
It therefore suffices to show that $\Phi_\R^P(\rho_0)\subseteq R$.
We will only prove that $\Phi_{[0,\infty)}^P(\rho_0)\subseteq R$;
a similar argument proves that $\Phi_{(-\infty,0]}^P(\rho_0)\subseteq R$.
We define $Q:=\{t\ge0\,|\,\Phi_t^P(\rho_0)\in R\}$.
We wish to show that $Q=[0,\infty)$.
Assume, for a contradiction, that $Q\ne[0,\infty)$.

Since $Q\subsetneq[0,\infty)$,
we get $[0,\infty)\backslash Q\ne\emptyset$.
Let $t_1:=\inf\,[0,\infty)\backslash Q$.

{\it Claim 2:} $t_1>0$.
{\it Proof of Claim 2:}
Since $\rho_0\in(\R^4)\backslash(\,\barI^4\,)$,
by continuity of $\Phi_\bullet^P(\rho_0)$,
fix $\gamma>0$ such that
$\Phi_{[0,\gamma)}^P(\rho_0)\subseteq(\R^4)\backslash(I^4)$.
Then, because $X=P$ on~$(\R^4)\backslash(I^4)$,
by \lref{lem-orbits-agree},
$\Phi_{[0,\gamma)}^X(\rho_0)=\Phi_{[0,\gamma)}^P(\rho_0)$.

Then $\Phi_{[0,\gamma)}^P(\rho_0)\subseteq\Phi_\R^X(\rho_0)\subseteq R$.
It follows that $[0,\gamma)\subseteq Q$.
Therefore $t_1=\inf\,[0,\infty)\backslash Q\ge\gamma>0$.
{\it End of proof of Claim 2.}

Because $\inf[0,\infty)\backslash Q=t_1$,
we get $[0,t_1)\subseteq Q$ and, for all $\delta>0$, we have $[t_1,t_1+\delta)\not\subseteq Q$.
Because $[0,t_1)\subseteq Q$, it follows that
\begin{itemize}
\item[$(+)$]$\Phi_{[0,t_1)}^P(\rho_0)\,\,\subseteq\,\,R\,\,=\,\,(\Phi_\R^X(\rho_0))\,\cup\,(I^4)$.
\end{itemize}
Let $\rho_1:=\Phi_{t_1}^P(\rho_0)$.
For all $\delta>0$, we have $[t_1,t_1+\delta)\not\subseteq Q$,
from which it follows that $\Phi_{[t_1,t_1+\delta)}^P(\rho_0)\not\subseteq R$.
That is,
\begin{itemize}
\item[$(*)$] for all $\delta>0$, \qquad $\Phi_{[0,\delta)}^P(\rho_1)\,\,\not\subseteq\,\,R\,\,=\,\,(\Phi_\R^X(\rho_0))\,\cup\,(I^4)$.
\end{itemize}

{\it Claim 3:} $\rho_1\notin I^4$.
{\it Proof of Claim 3:}
Suppose that $\rho_1\in I^4$.
We wish to obtain a contradiction.

By continuity of $\Phi_\bullet^P(\rho_1)$,
fix $\delta>0$ such that
$\Phi_{[0,\delta)}^P(\rho_1)\subseteq I^4$.
Then $\Phi_{[0,\delta)}^P(\rho_1)\subseteq I^4\subseteq R$.
This contradicts $(*)$.
{\it End of proof of Claim 3.}

{\it Claim 4:} Let $N$ be an interval in $\R$
such that $\inf N<0$ and $0\in N$.
Assume that $\Phi_N^P(\rho_1)\subseteq(\R^4)\backslash(I^4)$.
Then $\Phi_N^P(\rho_1)\subseteq\Phi_\R^X(\rho_0)$.
{\it Proof of Claim 4:}
Because $X=P$ on $(\R^4)\backslash(I^4)$,
it follows, from \lref{lem-orbits-agree},
that, for all $t\in N$,
we have $\Phi_t^X(\rho_1)=\Phi_t^P(\rho_1)$.
Then $\Phi_N^X(\rho_1)=\Phi_N^P(\rho_1)$.
By Claim 2, $t_1>0$.
Choose $t_*\in(0,t_1]$ such that $-t_*\in N$.
Let $\rho_*:=\Phi_{-t_*}^X(\rho_1)=\Phi_{-t_*}^P(\rho_1)$.
Then $\Phi_\R^X(\rho_*)=\Phi_\R^X(\rho_1)$.
It follows that $\Phi_N^P(\rho_1)=\Phi_N^X(\rho_1)\subseteq\Phi_\R^X(\rho_1)=\Phi_\R^X(\rho_*)$.

We have $t_1-t_*\in[0,t_1)$.
So, by $(+)$, we get $\Phi_{t_1-t_*}^P(\rho_0)\in R$.
Since $\rho_1=\Phi_{t_1}^P(\rho_0)$,
we see that $\Phi_{-t_*}^P(\rho_1)=\Phi_{t_1-t_*}^P(\rho_0)$.
Then
$$\rho_*\,\,\,=\,\,\,\Phi_{-t_*}^P(\rho_1)\,\,\,=\,\,\,
\Phi_{t_1-t_*}^P(\rho_0)\,\,\,\in\,\,\,R\,\,\,=\,\,\,(\Phi_\R^X(\rho_0))\,\cup\,(I^4).$$
So, since $\rho_*=\Phi_{-t_*}^P(\rho_1)\in\Phi_N^P(\rho_1)\subseteq(\R^4)\backslash(I^4)$,
we see that $\rho_*\in\Phi_\R^X(\rho_0)$.
Then $\Phi_\R^X(\rho_*)=\Phi_\R^X(\rho_0)$.
Then $\Phi_N^P(\rho_1)\subseteq\Phi_\R^X(\rho_*)=\Phi_\R^X(\rho_0)$.
{\it End of proof of Claim 4.}

{\it Claim 5:} $\rho_1\in\barI^4$.
{\it Proof of Claim 5:}
Suppose that $\rho_1\notin\barI^4$.
We wish to obtain a contradiction.

By continuity of $\Phi_\bullet^P(\rho_1)$,
fix $\delta>0$ such that $\Phi_{(-\delta,\delta)}^P(\rho_1)\subseteq(\R^4)\backslash(I^4)$.
Then, by Claim 4, we have
$\Phi_{(-\delta,\delta)}^P(\rho_1)\subseteq\Phi_\R^X(\rho_0)$.
Then
$$\Phi_{[0,\delta)}^P(\rho_1)\quad\subseteq\quad\Phi_{(-\delta,\delta)}^P(\rho_1)\quad\subseteq\quad
\Phi_\R^X(\rho_0)\quad\subseteq\quad R.$$
This contradicts $(*)$.
{\it End of proof of Claim 5.}

{\it Claim 6:} $\rho_1\notin B_\circ(I)$.
{\it Proof of Claim 6:}
Assume, for a contradiction, that $\rho_1\in B_\circ(I)$.
Then, by \lref{lem-omnibus-V0}(v), $\Phi_{[-1,1]}^{V_0}(\rho_1)\subseteq(\,\overline{3I}\,)^4$.
Since $P=V_0$ on $(\,\overline{3I}\,)^4$,
it follows, from \lref{lem-orbits-agree},
that, for all $t\in[-1,1]$,
we have $\Phi_t^P(\rho_1)=\Phi_t^{V_0}(\rho_1)$.

Because $\rho_1\in B_\circ(I)$,
by \lref{lem-omnibus-V0}(iv),
$\Phi_{[-1,0]}^{V_0}(\rho_1)\subseteq[(\,\overline{3I}\,)^4]\backslash[I^4]$.
Then $\Phi_{[-1,0]}^P(\rho_1)=\Phi_{[-1,0]}^{V_0}(\rho_1)\subseteq[(\,\overline{3I}\,)^4]\backslash[I^4]
\subseteq(\R^4)\backslash(I^4)$.
Then, by Claim 4, we get
$\Phi_{[-1,0]}^P(\rho_1)\subseteq\Phi_\R^X(\rho_0)$.
Then $\Phi_{[-1,0]}^P(\rho_1)\subseteq\Phi_\R^X(\rho_0)\subseteq R$.

Because $\rho_1\in B_\circ(I)$,
by \lref{lem-omnibus-V0}(i), $\Phi_{(0,1]}^{V_0}(\rho_1)\subseteq I^4$.
Then
$$\Phi_{(0,1]}^P(\rho_1)\quad=\quad\Phi_{(0,1]}^{V_0}(\rho_1)\quad\subseteq\quad I^4\quad\subseteq\quad R.$$

Then $\Phi_{[0,1)}^P(\rho_1)\subseteq\Phi_{[-1,1]}^P(\rho_1)=(\Phi_{[-1,0]}^P(\rho_1))\cup(\Phi_{(0,1]}^P(\rho_1))\subseteq R$.
This contradicts $(*)$.
{\it End of proof of Claim 6.}

{\it Claim 7:} $\rho_1\in T_\circ(I)$.
{\it Proof of Claim 7:}
By Claim 5 and Claim 3, we have $\rho_1\in(\,\barI^4\,)\backslash(I^4)$.
Let $S:=(\,\barI^3\,)\backslash(I^3)$.
Then
$$\rho_1\quad\in\quad(\,\barI^4\,)\backslash(I^4)\quad=\quad(B_\circ(I))\,\cup\,(S\times\barI\,)\,\cup\,(T_\circ(I)).$$
So, by Claim 6, we have
$\rho_1\in(S\times\barI\,)\cup(T_\circ(I))$,
and it suffices to show that $\rho_1\notin S\times\barI$.
Assume, for a contradiction, that $\rho_1\in S\times\barI$.

By \lref{lem-omnibus-V0}(vi),
$\Phi_{[-1,1]}^{V_0}(\rho_1)\subseteq[(\,\overline{3I}\,)^4]\backslash[I^4]\subseteq(\,\overline{3I}\,)^4$.
Then, since $P=V_0$ on~$(\,\overline{3I}\,)^4$,
by \lref{lem-orbits-agree},
$\Phi_{[-1,1]}^P(\rho_1)=\Phi_{[-1,1]}^{V_0}(\rho_1)$.
Then $\Phi_{[-1,1]}^P(\rho_1)=\Phi_{[-1,1]}^{V_0}(\rho_1)\subseteq[(\,\overline{3I}\,)^4]\backslash[I^4]\subseteq(\R^4)\backslash(I^4)$.
Then, by Claim~4, $\Phi_{[-1,1]}^P(\rho_1)\subseteq\Phi_\R^X(\rho_0)$.
Then $\Phi_{[0,1)}^P(\rho_1)\subseteq\Phi_{[-1,1]}^P(\rho_1)\subseteq\Phi_\R^X(\rho_0)\subseteq R$.
This contradicts $(*)$.
{\it End of proof of Claim 7.}

By Claim 7 and \lref{lem-omnibus-V0}(ii),
we see that $\Phi_{(-2a_I,0)}^{V_0}(\rho_1)\subseteq I^4$.
So, since $P=V_0$ on $I^4$,
it follows, from \cref{cor-vect-agree-implies-flow-agree}, that
\begin{itemize}
\item[$(**)$]for all $t\in[-2a_I,0]$, we have $\Phi_t^P(\rho_1)=\Phi_t^{V_0}(\rho_1)$.
\end{itemize}
Then, for all $t\in(-2a_I,0)$,
$\Phi_t^P(\rho_1)=\Phi_t^{V_0}(\rho_1)\in \Phi_{(-2a_I,0)}^{V_0}(\rho_1)\subseteq I^4\subseteq\barI^4$.
Also, $\Phi_{-t_1}^P(\rho_1)=\rho_0\notin\barI^4$.
Then $-t_1\notin(-2a_I,0)$, {\it i.e.}, $t_1\notin(0,2a_I)$.
So, by Claim 2, it follows that $t_1\ge2a_I$.
Then $0\le t_1-2a_I<t_1$.

By Claim 7, $\rho_1\in T_\circ(I)$.
Let $\rho':=\Phi_{-2a_I}^{V_0}(\rho_1)$.
By \lref{lem-drop-from-top},
we have $\rho'\in B_\circ(I)$ and $\rho_1=SU_I(\rho')$.
By $(**)$, $\Phi_{-2a_I}^P(\rho_1)=\Phi_{-2a_I}^{V_0}(\rho_1)$.
Then
$$\rho'\,\,=\,\,\Phi_{-2a_I}^{V_0}(\rho_1)\,\,=\,\,\Phi_{-2a_I}^P(\rho_1)\,\,=\,\,
\Phi_{-2a_I}^P(\Phi_{t_1}^P(\rho_0))\,\,=\,\,\Phi_{t_1-2a_I}^P(\rho_0).$$
So, since $t_1-2a_I\in[0,t_1)$, by $(+)$,
we see that $\rho'\in(\Phi_\R^X(\rho_0))\cup(I^4)$.
Since $\rho'\in B_\circ(I)$ and $(B_\circ(I))\cap(I^4)=\emptyset$,
we have $\rho'\notin I^4$.
Then $\rho'\in\Phi_\R^X(\rho_0)$.
Then, because $Z'$ is $X$-invariant and $\rho_0\notin Z'$,
we see that $\rho'\notin Z'$.
Then $\rho'\notin Z'_2$, so $\rho'\notin Z_2=S_2\cap U_2$.
So, since $\rho'\in B_\circ(I)\subseteq\barI^4\subseteq(2I)^4=U_2$,
we conclude that $\rho'\notin S_2$.
That is, $\rho'\in\scru(V)$.
Then
$$\rho'\,\,\in\,\,(B_\circ(I))\,\,\cap\,\,(\scru(V))\,\,\subseteq\,\,
(B(I))\cap(\scru(V))\,\,=\,\,\scru_B(V,I).$$
So, since $(V,I)\in\scrd_+$, we get $SU_I(\rho')=UF_I^V(\rho')$.
Let $t_0:=TF_I^V(\rho')$.
We have $\rho'\in\scru_B(V,I)$, so $\rho'=DF_I^V(\rho')$.
Then $\Phi_{t_0}^V(\rho')=UF_I^V(\rho')$.
By \lref{lem-undeterred-orbit},
$\Phi_{(0,t_0)}^V(\rho')\subseteq I^4$.
Also, $X=V$ on $I^4$.
Then, by \cref{cor-vect-agree-implies-flow-agree},
$\Phi_{t_0}^X(\rho')=\Phi_{t_0}^V(\rho')$.
Then $UF_I^V(\rho')=\Phi_{t_0}^V(\rho')=\Phi_{t_0}^X(\rho')$.
Then $\rho_1=SU_I(\rho')=UF_I^V(\rho')=\Phi_{t_0}^X(\rho')\in\Phi_\R^X(\rho')$.
Also, recall that $\rho'\in\Phi_\R^X(\rho_0)$.
We conclude that $\Phi_\R^X(\rho_1)=\Phi_\R^X(\rho')=\Phi_\R^X(\rho_0)$.

By Claim 7 and \lref{lem-omnibus-V0}(iii),
$\Phi_{[0,1]}^{V_0}(\rho_1)\subseteq[(\,\overline{3I}\,)^4]\backslash[I^4]$.
So, since $X=V_0=P$ on $[(\,\overline{3I}\,)^4]\backslash[I^4]$,
by \lref{lem-orbits-agree}, we have
$$\Phi_{[0,1]}^X(\rho_1)\quad=\quad\Phi_{[0,1]}^{V_0}(\rho_1)\quad=\quad\Phi_{[0,1]}^P(\rho_1).$$
Then
$\Phi_{[0,1)}^P(\rho_1)\subseteq\Phi_{[0,1]}^P(\rho_1)=
\Phi_{[0,1]}^X(\rho_1)\subseteq\Phi_\R^X(\rho_1)=\Phi_\R^X(\rho_0)\subseteq R$.
This contradicts $(*)$.
\end{proof}

\begin{lem}\wrlab{lem-merge-is-periodic}
Let $(V,I)\in\scrd_*$.
Let $(P,K)\in\scrp_I$.
Let $X:=\scrx_I(P,V)$.
Then $(X,\xi_I)$ is periodic to all orders.
\end{lem}

\begin{proof}
Since $(V,I)\in\scrd_*$,
fix an integer $j\ge1$ and an open neighborhood~$N$ in $\scru_B^\circ(V,I)$
of $\xi_I$ such that $TF_I^V=j$ on~$N$.
Then, by definition of~$UF_I^V$, we see that $UF_I^V=\Phi_j^V$ on $N$.
Because $(V,I)\in\scrd_*\subseteq\scrd_+$, it follows that $UF_I^V=SU_I$ on $\scru_B(V,I)$.
We have $P=V_0$ on~$I^4$.
By \lref{lem-omnibus-V0}(i), we have $\Phi_{(0,2a_I)}^{V_0}(B_\circ(I))=I^4$.
Then, by \cref{cor-vect-agree-implies-flow-agree},
on $B(I)$, we have $\Phi_{2a_I}^P=\Phi_{2a_I}^{V_0}$.
Also, by \lref{lem-SU-and-Phi-V0}(ii), on $B(I)$, we have $SU_I=\Phi_{2a_I}^{V_0}$.
Because $X=V$ on $I^4$,
by \lref{lem-agree-inside-I4},
$\Phi_j^X=\Phi_j^V$ on~$\scru_B^\circ(V,I)$.
We have $N\subseteq\scru_B^\circ(V,I)\subseteq\scru_B(V,I)\subseteq B(I)$.
Then, on $N$, we have
$$\Phi_j^X\,\,=\,\,\Phi_j^V\,\,=\,\,UF_I^V\,\,=\,\,SU_I\,\,=\,\,\Phi_{2a_I}^{V_0}\,\,=\,\,\Phi_{2a_I}^P.$$

For all $t\in(0,2a_I)$, let $L_t:=\Phi_{-t}^{V_0}(N)$.

{\it Claim:} Let $t\in(0,2a_I)$. Then, on $L_t$,
$\Phi_{j+2a_I}^X=\Phi_{4a_I}^{V_0}=\Phi_{4a_I}^P$.
{\it Proof of claim:}
We have $X=V_0=P$ on $[(\,\overline{3I}\,)^4]\backslash[I^4]$
and $2a_I-t\in(0,2a_I)$.
By \lref{lem-omnibus-V0}(iii), we have
$\Phi_{(0,2a_I)}^{V_0}(T_\circ(I))\subseteq[(\,\overline{3I}\,)^4]\backslash[I^4]$.
Therefore, by \lref{lem-orbits-agree},
on $T_\circ(I)$, we have $\Phi_{2a_I-t}^X=\Phi_{2a_I-t}^{V_0}=\Phi_{2a_I-t}^P$.

We have $X=V_0=P$ on $[(\,\overline{3I}\,)^4]\backslash[I^4]$.
By \lref{lem-omnibus-V0}(iv), we have
$\Phi_{(-2a_I,0)}^{V_0}(B_\circ(I))\subseteq[(\,\overline{3I}\,)^4]\backslash[I^4]$.
Also, $N\subseteq\scru_B^\circ(V,I)\subseteq B_\circ(I)$.
Then
$$\Phi_{(0,t)}^{V_0}(L_t)\,\,=\,\,\Phi_{(-t,0)}^{V_0}(N)\subseteq\,\,
\Phi_{(-2a_I,0)}^{V_0}(B_\circ(I))\,\,\subseteq\,\,[(\,\overline{3I}\,)^4]\backslash[I^4].$$
So, by \cref{cor-vect-agree-implies-flow-agree},
on $L_t$, we have $\Phi_t^X=\Phi_t^{V_0}=\Phi_t^P$.
We also have $\Phi_t^{V_0}(L_t)=N$.
Then, since $\Phi_j^X=\Phi_{2a_I}^{V_0}=\Phi_{2a_I}^P$ on $N$,
we see that, on $L_t$,
$$\Phi_j^X\circ\Phi_t^X\quad=\quad\Phi_{2a_I}^{V_0}\circ\Phi_t^{V_0}\quad=\quad\Phi_{2a_I}^P\circ\Phi_t^P.$$
By \lref{lem-SU-and-Phi-V0}(ii),
$\Phi_{2a_I}^{V_0}(B_\circ(I))=SU_I(B_\circ(I))$.
Then
$$\Phi_{2a_I}^{V_0}(\Phi_t^{V_0}(L_t))\,=\,\Phi_{2a_I}^{V_0}(N)\,\subseteq\,
\Phi_{2a_I}^{V_0}(B_\circ(I))\,=\,SU_I(B_\circ(I))\,=\,T_\circ(I).$$
So, since $\Phi_{2a_I-t}^X=\Phi_{2a_I-t}^{V_0}=\Phi_{2a_I-t}^P$ on $T_\circ(I)$,
we see that, on $L_t$,
$$\Phi_{2a_I-t}^X\circ\Phi_j^X\circ\Phi_t^X\,\,=\,\,\Phi_{2a_I-t}^{V_0}\circ\Phi_{2a_I}^{V_0}\circ\Phi_t^{V_0}
\,\,=\,\,\Phi_{2a_I-t}^P\circ\Phi_{2a_I}^P\circ\Phi_t^P.$$
That is, on $L_t$, we have $\Phi_{j+2a_I}^X=\Phi_{4a_I}^{V_0}=\Phi_{4a_I}^P$.
{\it End of proof of claim.}

Since $(P,K)\in\scrp_I$,
fix an integer $m>2a_I$ such that
 \begin{itemize}
 \item[$(*)$]$\Phi_m^P$ agrees with $\Id_4$ at $\xi_I$ to all orders \qquad and
 \item[$(**)$]for all $\tau\in B_\circ(I)$, \, for all $t\in(0,m)$,
       $$[\,\Phi_t^P(\tau)\in I^4\,]\qquad\Leftrightarrow\qquad[\,t<2a_I\,].$$
 \end{itemize}

Let $U:=\displaystyle{\bigcup_{t\in(0,2a_I)}L_t=\Phi_{(-2a_I,0)}^{V_0}(N)}$.
The claim shows that, on $U$, we have
$\Phi_{j+2a_I}^X=\Phi_{4a_I}^{V_0}=\Phi_{4a_I}^P$.

We have $P=V_0$ on $(\,\overline{3I}\,)^4$ and,
from \lref{lem-omnibus-V0}(v),
we conclude that
$\Phi_{[0,4a_I]}^{V_0}(B_\circ(I))\subseteq(\,\overline{3I}\,)^4$.
Then, by \cref{cor-vect-agree-implies-flow-agree},
for all $t\in[0,4a_I]$, on~$B(I)$, $\Phi_t^P=\Phi_t^{V_0}$.
Let $U_1:=\Phi_{(2a_I,4a_I)}^P(B_\circ(I))=\Phi_{(2a_I,4a_I)}^{V_0}(B_\circ(I))$.
We have
$$\Phi_{(0,m-4a_I)}^P(U_1)\qquad\subseteq\qquad\Phi_{(2a_I,m)}^P(B_\circ(I)),$$
and, by $(**)$, we have $\Phi_{(2a_I,m)}^P(B_\circ(I))\subseteq(\R^4)\backslash(I^4)$.
Then
$$\Phi_{(0,m-4a_I)}^P(U_1)\quad\subseteq\quad(\R^4)\backslash(I^4).$$
Also, $X=P$ on $(\R^4)\backslash(I^4)$.
Then, by \cref{cor-vect-agree-implies-flow-agree},
on $U_1$, we have $\Phi_{m-4a_I}^X=\Phi_{m-4a_I}^P$.
Recall that, on $U$, $\Phi_{j+2a_I}^X=\Phi_{4a_I}^P$.
Also,
$$\Phi_{4a_I}^{V_0}(U)\,\,=\,\,\Phi_{(2a_I,4a_I)}^{V_0}(N)\,\,\subseteq\,\,
\Phi_{(2a_I,4a_I)}^{V_0}(B_\circ(I))\,\,=\,\,U_1.$$
Then, on $U$, we have
$\Phi_{m-4a_I}^X\circ\Phi_{j+2a_I}^X=\Phi_{m-4a_I}^P\circ\Phi_{4a_I}^P$.
That is, on $U$, we have $\Phi_{m-2a_I+j}^X=\Phi_m^P$.

We have $X=V_0=P$ on $[(\,\overline{3I}\,)^4]\backslash[I^4]$.
Also, $\xi_I\in B_\circ(I)$, so, and, by \lref{lem-omnibus-V0}(iv),
we get $\Phi_{[-2a_I,0]}^{V_0}(\xi_I)\subseteq[(\,\overline{3I}\,)^4]\backslash[I^4]$.
Then, by \lref{lem-orbits-agree}, we have
$\Phi_{-a_I}^X(\xi_I)=\Phi_{-a_I}^{V_0}(\xi_I)=\Phi_{-a_I}^P(\xi_I)$.

Let $\rho:=\Phi_{-a_I}^{V_0}(\xi_I)=\Phi_{-a_I}^X(\xi_I)=\Phi_{-a_I}^P(\xi_I)$.
Because $\xi_I\in N$, it follows that $\rho\in\Phi_{(-2a_I,0)}^{V_0}(N)=U$.
By \cref{cor-U-open}(iii), $\scru_B^\circ(V,I)$ is open in $B_\circ(I)$.
So, as $N$ is open in $\scru_B^\circ(V,I)$, we see that $N$ is open in~$B_\circ(I)$.
Then, by \lref{lem-nbd-in-B-saturates-to-nbd}, $U$ is an open subset of $\R^4$.
Then $U$  is an open neighborhood of $\rho$.
So, since $\Phi_{m-2a_I+j}^X=\Phi_m^P$ on $U$, we see that
$\Phi_{m-2a_I+j}^X$ agrees to all orders with $\Phi_m^P$ at~$\rho$.
Since $\rho\in\Phi^P_\R(\xi_I)$, it follows, from $(*)$,
that $\Phi_m^P$ agrees to all orders with $\Id_4$ at~$\rho$.
Then $\Phi_{m-2a_I+j}^X$ agrees to all orders with $\Id_4$ at~$\rho$.
We have $m>2a_I$ and $j\ge1$, so $m-2a_I+j\ne0$.
Then $(X,\rho)$ is periodic to all orders.
So, as $\xi_I=\Phi_{a_I}^X(\rho)\in\Phi_\R^X(\rho)$, we see that $(X,\xi_I)$ is periodic to all orders.
\end{proof}

\section{The iteration\wrlab{sect-iteration}}

Recall, from \secref{sect-dtuflow},
the definitions of $DF_I^V$, $UF_I^V$ and $TF_I^V$.

Recall, from \secref{sect-gettoD},
the definitions of $\scrd_+$, $\scrd^\#$, $\scrd_+^\#$ and $\scrd_*$.
The definition of $\scrd_I^\times$ appears in \secref{sect-notation}.

Let $\rho\in\R^4$.
We define $\scrt_\rho:\R^4\to\R^4$ by $\scrt_\rho(\sigma)=\sigma+\rho$.
For any $V\in\scrc$,
we define $\scrt_\rho V:\R^4\to\R^4$ by $(\scrt_\rho V)(\sigma)=V(\sigma-\rho)$;
then $\scrt_\rho V\in\scrc$ and $\scru(\scrt_\rho V)=\scrt_\rho(\scru(V))$,
so: \qquad [\,$V$ is porous\,] \,\, iff \,\, [\,$\scrt_\rho V$ is porous\,].

For all $\rho\in\R^4$, we have $\scrt_\rho V_0=V_0$.
Also, for all $\rho\in\R^4$, for all $a\in\R$, for all $V\in\scrv(a)$,
we have $\scrt_\rho V\in\scrv(a+[\Pi_4(\rho)])$.

Let $J\in\scri$.
Let $\rho\in J^4$.
Then, for all $I\in\scri$, we have $\scrt_\rho(I^4)\subseteq(I+J)^4$.
Moreover, for all $(V,I)\in\scrd$, we have $(\scrt_\rho V,I+J)\in\scrd$.

For all $V\in\scrc$, for all $\rho,\sigma\in\R^4$, we have:
if $(V,\sigma)$ is periodic to all orders,
then $(\scrt_\rho V,\sigma+\rho)$ is periodic to all orders.

\begin{lem}\wrlab{lem-base-lem}
Let $(V,I)\in\scrd$.
Assume $V$ is porous.
Let $\sigma'\in\scru(V)$.
Then there exists $(V',I')\in\scrm(V,I)$
such that $V'$ is porous and
such that $(V',\sigma')$ is periodic to all orders.
\end{lem}

\begin{proof}
Choose $J_0\in\scri$ such that $\barI\subseteq J_0$ and such that $\sigma'\in J_0^4$.
Then $a_I<a_{J_0}$ and $(V,J_0)\in\scrd$.
Also, $\sigma'\in(\scru(V))\cap(J_0^4)=\scru_\circ(V,J_0)$.
Since $(V,J_0)\in\scrd$, it follows that $V\in\scrv(a_{J_0})$.

Let $\cksigma:=DF_{J_0}^{V}(\sigma')$.
Then, by \lref{lem-interior-dnflow}, we have $\cksigma\in B_\circ(J_0)$.
Let $\tau_1:=\xi_{J_0}-\cksigma$.
Then $\scrt_{\tau_1}(\cksigma)=\xi_{J_0}$
and $\scrt_{-\tau_1}(\xi_{J_0})=\cksigma$.

For all $\rho\in B_\circ(J_0)$, $\xi_{J_0}-\rho\in J_0^3\times\{0\}$.
Then $\tau_1\in J_0^3\times\{0\}$,
so $\tau_1\in J_0^4$ and $\Pi_4(\tau_1)=0$.
Let $W:=\scrt_{\tau_1}V$.
Then $V=\scrt_{-\tau_1}W$.
Let $J:=J_0+J_0=2J_0$.
Then $(W,J)\in\scrd$ and $W$ is porous.

Let $s_0:=a_{J_0}$.
Then $(0,0,0,-s_0)=\xi_{J_0}$ and $\Phi_{-s_0}^{V_0}(\cksigma)=\cksigma+\xi_{J_0}$.
We have $\cksigma\in B_\circ(J_0)\subseteq\R^3\times\{-a_{J_0}\}\subseteq\R^3\times(-\infty,-a_{J_0}]$.
By \lref{lem-V-V0-agreement-for-D}(iii), we have
$\Phi_{-s_0}^{V}(\cksigma)=\Phi_{-s_0}^{V_0}(\cksigma)$.
Then $\Phi_{-s_0}^{V}(\cksigma)=\cksigma+\xi_{J_0}$.

We have $\cksigma=DF_{J_0}^{V}(\sigma')\in\Phi_\R^{V}(\sigma')$,
and $\cksigma+\xi_{J_0}=\Phi_{-s_0}^{V}(\cksigma)\in\Phi_\R^{V}(\cksigma)$.
Then $\cksigma+\xi_{J_0}\in\Phi_\R^{V}(\sigma')$.
So, since $\sigma'\in\scru(V)$
and $\scru(V)$ is $V$-invariant,
we see that $\cksigma+\xi_{J_0}\in\scru(V)$.
Then $\scrt_{\tau_1}(\cksigma+\xi_{J_0})\in\scru(W)$.
So, since
$\scrt_{\tau_1}(\cksigma+\xi_{J_0})=2\xi_{J_0}=\xi_{2J_0}=\xi_J$,
we get $\xi_J\in\scru(W)$.
Then
$$\xi_J\quad\in\quad(\scru(W))\,\cap\,(B(J))\quad=\quad\scru_B(W,J).$$
Then $(W,J)\in\scrd^\#$.
By \lref{lem-upflow-const},
fix $(V_*,I_*)\in(\scrm_*(W,J))\cap(\scrd_*)$
such that $V_*$ is porous.
Then $(V_*,I_*)\in\scrd_*\subseteq\scrd$.
Also,
$$(V_*,I_*)\quad\in\quad\scrm_*(W,J)\quad\subseteq\quad\scrm(W,J),$$
so $a_J<a_{I_*}$ and $V_*=W$ on $\barJ^4$.
As $a_J<a_{I_*}$, we get $\barJ\subseteq I_*$.
Moreover, because $(V_*,I_*)\in\scrm_*(W,J)$,
it follows that $V_*\in\scrv(a_J)$.

Let $P_0$ and $K_0$ be as in \secref{sect-results-scrp}.
By \lref{lem-P0-porous}, $P_0$ is porous.
By \lref{lem-P0K0-in-scrp}, $(P_0,K_0)\in\scrp_{I_0}$
Let $K:=a_{I_*}K_0$ and
define $P:\R^4\to\R^4$ by $P(\sigma)=P_0(\sigma/a_{I_*})$.
Then $P$ is porous and $(P,K)\in\scrp_{I_*}$.

Since $(P,K)\in\scrp_{I_*}\subseteq\scrd_{I_*}^\times$,
it follows that $4I_*\subseteq K$,
so $a_{4I_*}\le a_K$.
Then $a_{I_*}<4a_{I_*}=a_{4I_*}\le a_K$.

Let $X:=\scrx_{I_*}(P,V_*)$.
Then $X=V_*$ on $I_*^4$, and, therefore, on $\barJ^4$.
Then $X=V_*=W$ on $\barJ^4$.
Since
$$(V_*,I_*)\,\,\in\,\,\scrd
\qquad\hbox{and}\qquad
(P,K)\,\,\in\,\,\scrd_{I_*}^\times,$$
we see, by \lref{lem-exch-gives-mod}, that $(X,K)\in\scrm(V_*,I_*)$.
Therefore we have $(X,K)\in\scrm(V_*,I_*)\subseteq\scrd$.
Because
$$(V_*,I_*)\,\,\in\,\,\scrd_*\,\,\subseteq\,\,\scrd_+
\qquad\hbox{and}\qquad
(P,K)\,\,\in\,\,\scrd_{I_*}^\times,$$
and because $V_*$ and $P$ are porous,
we see, by \lref{lem-make-por-per}, that $X$ is porous.
Because
$$(V_*,I_*)\,\,\in\,\,\scrd_*\qquad\hbox{and}\qquad(P,K)\,\,\in\,\,\scrp_{I_*},$$
we see, by \lref{lem-merge-is-periodic},
that $(X,\xi_{I_*})$ is periodic to all orders.

Let $V':=\scrt_{-\tau_1}X$.
Then $V'$ is porous.
Since $\tau_1\in J_0^4$, we get $-\tau_1\in J_0^4$.
Let $I':=K+J_0$.
Then $(V',I')\in\scrd$.

Because $(X,K)\in\scrm(V_*,I_*)$ and $(V_*,I_*)\in\scrm(W,J)$,
it follows that $(X,K)\in\scrm(W,J)$.
Then $X=W$ on $\barJ^4$.
Then, on $\scrt_{-\tau_1}(\,\barJ^4\,)$, we have
$$V'\,\,=\,\,\scrt_{-\tau_1}X\,\,=\,\,\scrt_{-\tau_1}W\,\,=\,\,V.$$
We have $\scrt_{\tau_1}(J_0^4)\subseteq(J_0+J_0)^4=J^4\subseteq\barJ^4$,
so $J_0^4\subseteq\scrt_{-\tau_1}(\,\barJ^4\,)$.
Therefore $V'=V$ on $J_0^4$.
So, since $\barI\subseteq J_0$, we see that $V'=V$ on $\barI^4$.
Also,
$$a_I<a_{J_0}<2a_{J_0}=a_{2J_0}=a_J<a_{I_*}<a_K<a_K+a_{J_0}=a_{K+J_0}=a_{I'}.$$
Then $(V',I')\in\scrm(V,I)$.
It remains only to prove that $(V',\sigma')$ is periodic to all orders.

Because $V_*\in\scrv(a_J)$,
we have $V_*=V_0$ on $\R^3\times(-\infty,-a_J]$.
Since $X=\scrx_{I_*}(P,V_*)$, we have $X=V_*$ on $\barI_*^4$.
Then
\begin{itemize}
\item[(i)]$X=V_*=V_0$ on $\{(0,0,0)\}\times[-a_{I_*},-a_J]$.
\end{itemize}

We have $W=\scrt_{\tau_1}V$ and $\Pi_4(\tau_1)=0$.
So, since $V\in\scrv(a_{J_0})$, we have $W\in\scrv(a_{J_0})$ as well.
That is, $W=V_0$ on $\R^3\times(-\infty,-a_{J_0}]$.
Then, because $X=W$ on $\barJ^4$,
we see that
\begin{itemize}
\item[(ii)]$X=W=V_0$ on $\{(0,0,0)\}\times[-a_J,-a_{J_0}]$.
\end{itemize}

Let $s_1:=a_{I_*}-a_{J_0}$.
By \lref{lem-midpoint-connection}(i), $\Phi_{s_1}^{V_0}(\xi_{I_*})=\xi_{J_0}$.
By (i) and~(ii), we have $X=V_0$ on $\{(0,0,0)\}\times[-a_{I_*},-a_{J_0}]$.
By \lref{lem-midpoint-connection}(ii),
$\Phi_{[0,s_1]}^{V_0}(\xi_{I_*})\subseteq\{(0,0,0)\}\times[-a_{I_*},-a_{J_0}]$.
Then, by \lref{lem-orbits-agree}, we have
$\Phi_{s_1}^X(\xi_{I_*})=\Phi_{s_1}^{V_0}(\xi_{I_*})$.
Then $\Phi_{s_1}^X(\xi_{I_*})=\xi_{J_0}$.
Then $\xi_{J_0}\in\Phi_\R^X(\xi_{I_*})$, so,
because $(X,\xi_{I_*})$ is periodic to all orders,
we see that $(X,\xi_{J_0})$ is periodic to all orders.
Then, because $\scrt_{-\tau_1}X=V'$ and $\scrt_{-\tau_1}(\xi_{J_0})=\cksigma$,
we conclude that $(V',\cksigma)$ is periodic to all orders.

Since $\cksigma=DF_{J_0}^{V}(\sigma')$, it follows that $\sigma'\in\Phi_\R^{V}(\cksigma)$.
So choose $s_0\in\R$ such that $\Phi_{s_0}^{V}(\cksigma)=\sigma'$.
Recall that $V'=V$ on $J_0^4$.
Also, $(V,J_0)\in\scrd$ and $J_0\subseteq2J_0=J$.
Then $(V,J)\in\scrd$.
Also,
$$\cksigma\in B_\circ(J_0)\subseteq\barJ_0^4\subseteq(2J_0)^4=J^4
\quad\hbox{and}\quad
\Phi_{s_0}^{V}(\cksigma)=\sigma'\in J_0^4\subseteq J^4.$$
Then, by \lref{lem-endpts-in-I4-implies-coincidence},
$\Phi_{s_0}^{V'}(\cksigma)=\Phi_{s_0}^{V}(\cksigma)$.
It follows that
$$\sigma'\quad=\quad\Phi_{s_0}^{V}(\cksigma)\quad=\quad\Phi_{s_0}^{V'}(\cksigma)\quad\in\quad\Phi_\R^{V'}(\cksigma).$$
So, since $(V',\cksigma)$ is periodic to all orders,
it follows that $(V',\sigma')$ is periodic to all orders as well.
\end{proof}

\section{Results about cyclic groups\wrlab{sect-res-cyc}}

\begin{lem}\wrlab{lem-cyc-precpt-or-closed}
Let $C$ be a cyclic subgroup of a locally compact (Hausdorff) topological group $P$.
Then either $C$ is a closed subset of $P$
or the closure in $P$ of $C$ is compact.
\end{lem}

\begin{proof}
Let $\barC$ denote the closure in $P$ of $C$.
Give $\barC$  the relative topology inherited from $P$.
Assume that $\barC$ is noncompact.
We wish to show that $C$ is closed in $P$.

By Theorem 2.3.2, p.~39 of \cite{rud:fouriergps},
since $\barC$ is monothetic and noncompact,
it follows that $\barC$ is discrete.
Then every subset of $\barC$ is closed, and
in particular, $C$ is closed in $\barC$.
So, since $\barC$ is closed in $P$,
we see that $C$ is closed in $P$.
\end{proof}

\begin{lem}\wrlab{lem-cyc-lift}
Let $P$ and $Q$ be locally compact (Hausdorff) topological groups
and let $h:P\to Q$ be a surjective (continuous) homomorphism.
Let $Q_0$ be an infinite cyclic closed subgroup of $Q$.
Then there exists an infinite cyclic closed subgroup $P_0$ of $P$
such that $h(P_0)=Q_0$.
\end{lem}

\begin{proof}
Let $q_0$ be a generator of $Q_0$.
Since $h:P\to Q$ is surjective, fix $p_0\in P$ such that $h(p_0)=q_0$.
Let $P_0$ be the cyclic subgroup of~$P$ generated by $p_0$.
Then $P_0$ is a cyclic subgroup of $P$ and $h(P_0)=Q_0$.
Since $Q_0$ is infinite and since $h(P_0)=Q_0$, it follows that $P_0$ is infinite as well.
It remains to prove that $P_0$ is a closed subset of $P$.

Let $\barP_0$ be the closure in $P$ of $P_0$.
Because $Q_0$ is closed in $Q$ and because $h:P\to Q$ is continuous,
it follows that $h^{-1}(Q_0)$ is closed in~$P$.
So, since $P_0\subseteq h^{-1}(Q_0)$,
we conclude that $\barP_0\subseteq h^{-1}(Q_0)$,
{\it i.e.}, that $h(\,\barP_0)\subseteq Q_0$.
Then $Q_0=h(P_0)\subseteq h(\,\barP_0)\subseteq Q_0$.
Then $h(\,\barP_0)=Q_0$.

Give $Q_0$ its relative topology inherited from $Q$.
Then $Q_0$ is a countable, locally compact, Hausdorff topological space.
Every point in a $T_1$ topological space is either open or nowhere dense.
Then, by the Baire Category Theorem, $Q_0$ has an open point.
Then $Q_0$ is a topological group with an open point, so $Q_0$ is discrete.
So, since $Q_0$ is infinite, $Q_0$ is noncompact.
Then, since $h(\,\barP_0\,)=Q_0$, we see that $\barP_0$ is also noncompact.
Then, by \lref{lem-cyc-precpt-or-closed},
$P_0$ is a closed subset of $P$.
\end{proof}

\begin{lem}\wrlab{lem-noncpt-center-gives-inf-cyclic}
If $G$ is a connected real Lie group whose center $Z(G)$ is noncompact,
then $Z(G)$ has an infinite cyclic closed subgroup.
\end{lem}

\begin{proof}
Let $Z:=Z(G)$ and let $Z_\circ$ be the identity component of $Z$.
Let $Z_*:=Z/(Z_\circ)$ have the discrete topology.

By Lemma 3.1 of \cite{ao:prolgpactsfree},
$Z_*$ is finitely generated.
If $Z_*$ is infinite,
then the Structure Theorem for Finitely Generated Abelian Groups
implies that $Z_*$ contains an infinite cyclic subgroup,
so, by \lref{lem-cyc-lift}, we are done.
We therefore assume that $Z_*=Z/(Z_\circ)$ is finite.
Then, since $Z$ is noncompact, we see that $Z_\circ$ is noncompact.

By Theorem~1 of \S2.21.1
on p.~104 of \cite{mg:toptransgps},
let $K$ be a maximal compact subgroup of $Z_\circ$.
Let $A:=(Z_\circ)/K$ have its quotient topology
from the canonical map $Z_\circ\to A$.
Because $Z_\circ$ is noncompact, $K\ne Z_\circ$,
so fix $a\in A\backslash\{1_A\}$.
Let $C$ be the cyclic subgroup of $A$ generated by $a$.
Let $\barC$ be the closure in $A$ of $C$.
By maximality of $K$, any nontrivial closed subgroup of $A$ is noncompact,
so $\barC$ is noncompact.
It follows, from \lref{lem-cyc-precpt-or-closed},
that $C$ is a closed subgroup of $A$.
Moreover, as $\barC$ is noncompact,
we know that $C$ is infinite.
Thus $C$ is an infinite cyclic closed subgroup of $A=(Z_\circ)/K$.
By \lref{lem-cyc-lift}, there exists an infinite cyclic closed subgroup $C_1$ of $Z_\circ$.
As $Z_\circ$ is an open subgroup of $Z$, $Z_\circ$ is a closed subset of $Z$.
Then $C_1$ is a closed subset of $Z$.
\end{proof}

\section{The counterexamples\wrlab{sect-counterex}}

Note that \tref{thm-induction-gives-ctrx} below applies
when $G$ is discrete and isomorphic to the additive group $\Z$.
\lref{lem-noncpt-center-gives-inf-cyclic} shows that \tref{thm-induction-gives-ctrx} also applies
when $G$ is a connected real Lie group whose center $Z(G)$ is noncompact.
Also note, by (ii) below that, if
$$\hbox{either}\qquad G\hbox{ is connected }\qquad\hbox{or}\qquad G=Z,$$
then $M$ is connected.

\begin{thm}\wrlab{thm-induction-gives-ctrx}
Let $G$ be a real Lie group.
Assume the center~$Z(G)$ of~$G$
admits an infinite cyclic closed subgroup $Z$.
Then there is a $C^\infty$ manifold~$M$ and
a fixpoint rare $C^\infty$ action of $G$ on $M$
such that:
\begin{itemize}
\item[(i)]for any integer $k\ge0$,
there is a dense subset $D$ of $F_kM$ such that,
for all $\delta\in D$, the stabilizer $\Stab_Z(\delta)$ is infinite \qquad and
\item[(ii)]the number of connected components of $M$ and $G/Z$ are equal.
\end{itemize}
\end{thm}

\begin{proof}
We have $(V_0,I_0)\in\scrd$.
Also, $\scru(V_0)=\R^4$, so $V_0$ is porous.

Let $|\bullet|:\R^4\to[0,\infty)$ be a norm on $\R^4$.
Let $\{\omega_1,\omega_2,\omega_3,\ldots\}$ be a countable dense subset of $\R^4$.
Choose $\sigma_1\in\R^4$ such that $|\sigma_1-\omega_1|<1$.
Then $\sigma_1\in\R^4=\scru(V_0)$.

By \lref{lem-base-lem},
choose $(V_1,I_1)\in\scrm(V_0,I_0)$
such that $V_1$ is porous and
such that $(V_1,\sigma_1)$ is periodic to all orders.
Because $V_1$ is porous, $\scru(V_1)$ is dense in $\R^4$,
so fix $\sigma_2\in\scru(V_1)$ such that $|\sigma_2-\omega_2|<1/2$.

By \lref{lem-base-lem},
choose $(V_2,I_2)\in\scrm(V_1,I_1)$
such that $V_2$ is porous and
such that $(V_2,\sigma_2)$ is periodic to all orders.
Because $V_2$ is porous, $\scru(V_2)$ is dense in $\R^4$,
so fix $\sigma_3\in\scru(V_2)$ such that $|\sigma_3-\omega_3|<1/3$.

By \lref{lem-base-lem},
choose $(V_3,I_3)\in\scrm(V_2,I_2)$
such that $V_3$ is porous and
such that $(V_3,\sigma_3)$ is periodic to all orders.
Because $V_3$ is porous, $\scru(V_3)$ is dense in $\R^4$,
so fix $\sigma_4\in\scru(V_3)$ such that $|\sigma_4-\omega_4|<1/4$.

Continuing yields a countable dense subset
$\{\sigma_1,\sigma_2,\sigma_3,\ldots\}$ of~$\R^4$,
and a sequence $(V_1,I_1),(V_2,I_2),(V_3,I_3),\ldots$ in $\scrd$.
For each integer $j\ge1$,
\begin{itemize}
\item$a_{I_{j+1}}>a_{I_j}$,
\item$V_{j+1}=V_j$ on $I_j^4$,
\item$V_j$ is porous \qquad and
\item$(V_j,\sigma_j)$ is periodic to all orders.
\end{itemize}

We have $a_{I_1}<a_{I_2}<a_{I_3}<\cdots$ and $a_{I_1},a_{I_2},a_{I_3},\ldots\in\N$.
It follows both that $I_1\subseteq I_2\subseteq I_3\subseteq\cdots$
and that $I_1\cup I_2\cup I_3\cup\cdots=\R$.
Then $I_1^4\subseteq I_2^4\subseteq I_3^4\subseteq\cdots$
and $I_1^4\cup I_2^4\cup I_3^4\cup\cdots=\R^4$.
Define $V_\infty\in\scrc$ by the rule:
For all integers $j\ge1$, $V_\infty=V_j$ on $I_j^4$.
Let $\displaystyle{X:=\bigcap_{k=1}^\infty(\scru(V_k))}$.

{\it Claim 1:} For all $t\in\R\backslash\{0\}$,
for all $\sigma\in X$, we have
$\Phi_t^{V_\infty}(\sigma)\ne\sigma$.
{\it Proof of Claim~1:}
Let $t\in\R\backslash\{0\}$ and let $\sigma\in X$.
Assume, for a contradiction, that we have $\Phi_t^{V_\infty}(\sigma)=\sigma$.

Since $\Phi_\R^{V_\infty}(\sigma)=\Phi_{[0,t]}^{V_\infty}(\sigma)$,
we see that $\Phi_\R^{V_\infty}(\sigma)$ is compact.
Then, since $I_1^4\subseteq I_2^4\subseteq I_3^4\subseteq\cdots$
and $I_1^4\cup I_2^4\cup I_3^4\cup\cdots=\R^4$,
fix an integer $j\ge1$ such that
$\Phi_\R^{V_\infty}(\sigma)\subseteq I_j^4$.
We have $V_j=V_\infty$ on $I_j^4$.
Then, by \lref{lem-orbits-agree},
$\Phi_\R^{V_j}(\sigma)=\Phi_\R^{V_\infty}(\sigma)$.
Then
$\Pi_4(\Phi_\R^{V_j}(\sigma))=\Pi_4(\Phi_\R^{V_\infty}(\sigma))\subseteq\Pi_4(I_j^4)=I_j$.
However, because $\sigma\in X\subseteq\scru(V_j)$,
it follows that $\Pi_4(\Phi_\R^{V_j}(\sigma))=\R$.
Thus we have $\R\subseteq I_j$, contradiction.
{\it End of proof of Claim~1.}

{\it Claim 2:} For all integers $j\ge1$, $(V_\infty,\sigma_j)$ is periodic to all orders.
{\it Proof of Claim 2:}
Fix an integer $j\ge1$.
We wish to show that $(V_\infty,\sigma_j)$ is periodic to all orders.

Because $(V_j,I_j)\in\scrd$ and $(V_j,\sigma_j)$ is periodic,
we conclude, from \lref{lem-trap-per-orb},
that $\Phi_\R^{V_j}(\sigma_j)\subseteq I_j^4$.
For all $\tau\in I_j^4$,
because $V_\infty$ and $V_j$ agree on $I_j^4$,
which is an open neighborhood of~$\tau$,
it follows that they agree to all orders at $\tau$.
So, for all $t\in\R$,
$V_\infty$ and $V_j$ agree to all orders at $\Phi_t^{V_j}(\sigma_j)$.
Then, by \lref{lem-orbits-agree-to-all-orders}, for all $t\in\R$,
$\Phi_t^{V_\infty}:\R^4\to\R^4$ and $\Phi_t^{V_j}:\R^4\to\R^4$ agree to all orders at $\sigma_j$.
So, as $(V_j,\sigma_j)$ is periodic to all orders,
it follows $(V_\infty,\sigma_j)$~is periodic to all orders as well.
{\it End of proof of Claim~2.}

Because $Z$ is infinite cyclic,
it follows that $Z$ is isomorphic to the additive discrete group $\Z$.
Let $f:Z\to\Z$ be an isomorphism.
Define a $Z$-action on $\R^4$ by:
for all $z\in Z$, for all $\sigma\in\R^4$, $z\sigma=\Phi_{f(z)}^{V_\infty}(\sigma)$.
Let $M:=G\times_Z\R^4$.
Because the $Z$-action on~$\R^4$ is $C^\infty$,
it follows that the $G$-action on $M$ is $C^\infty$ as well.
By construction, $M$ is (the total space of) a fiber bundle over $G/Z$ with fiber $\R^4$,
so, because $\R^4$ is connected,
$M$ has the same number of connected components as does $G/Z$.

By \cref{cor-U-open}(i), $\scru(V)$ is an open subset of $\R^4$.
For all integers $k\ge1$, $V_k$ is porous.
Then $\scru(V_k)$ is a dense open subset of~$\R^4$.
Then, because
$\displaystyle{X=\bigcap_{k=1}^\infty(\scru(V_k))}$,
we see, by the Baire Category Theorem,
that $X$ is dense in~$\R^4$.
By Claim 1, for all $z\in Z\backslash\{1_Z\}$, for all $\sigma\in X$,
we have $\Phi_{f(z)}^{V_\infty}(\sigma)\ne\sigma$,
{\it i.e.}, we have $z\sigma\ne\sigma$.
Thus the $Z$-action on $\R^4$ is fixpoint rare.
Then the $G$-action on $M$ is also fixpoint rare.

Let $p:G\times\R^4\to G\times_Z\R^4=M$ be the canonical map.
Define an injection $\iota:\R^4\to M$ by $\iota(\sigma)=p(1_G,\sigma)$.
Let $\Sigma:=\iota(\{\sigma_1,\sigma_2,\sigma_3,\ldots\})$.
Because $\{\sigma_1,\sigma_2,\sigma_3,\ldots\}$ is dense in $\R^4$, $A:=G\Sigma$ is dense in $M$.

Fix an integer $k\ge0$,
and let $\pi:=\pi_k^M:F_kM\to M$ be the $k$th order frame bundle of $M$.
Let $D:=\pi^{-1}(A)$.
Since $A$ is dense in $M$, and since $\pi:F_kM\to M$ is open,
it follows that $D$ is dense in $F_kM$.
Fix $\delta\in D$.
We wish to show that $\Stab_Z(\delta)$ is infinite.

Let $S:=\Stab_G(\delta)$.
Then $\Stab_Z(\delta)=S\cap Z$.
We therefore wish to show that $S\cap Z$ is infinite.

Since $\pi(\delta)\in A=G\Sigma$,
fix $g_0\in G$, $\tau_0\in\Sigma$ such that $\pi(\delta)=g_0\tau_0$.
Let $\delta_0:=g_0^{-1}\delta$.
Then $\pi(\delta_0)=\tau_0$,
{\it i.e.}, $\delta_0\in\pi^{-1}(\tau_0)$.
Let $S_0:=\Stab_G(\delta_0)$.
Then $S=g_0S_0g_0^{-1}$.
As $Z\subseteq Z(G)$, we have $Z=g_0Zg_0^{-1}$.
Then $S\cap Z=g_0(S_0\cap Z)g_0^{-1}$,
so it suffices to show that $S_0\cap Z$ is infinite.

Recall that $\Id_4:\R^4\to\R^4$ is the identity map defined by $\Id_4(\sigma)=\sigma$.
Let $\Id:M\to M$ be the identity map defined by $\Id(\rho)=\rho$.

Since $\tau_0\in\Sigma=\iota(\{\sigma_1,\sigma_2,\sigma_3,\ldots\})$, fix an integer $j\ge1$
such that $\tau_0=\iota(\sigma_j)$.
By Claim~2, fix $n_0\in\Z\backslash\{0\}$
such that $\Phi_{n_0}^{V_\infty}:\R^4\to\R^4$ agrees with the identity $\Id_4:\R^4\to\R^4$
to all orders at $\sigma_j$.

Let $z_0:=f^{-1}(n_0)$.
Then $z_0\in Z\backslash\{1_Z\}$,
and, for all $\sigma\in\R^4$, we have $z_0\sigma=\Phi_{n_0}^{V_\infty}(\sigma)$.
Then the map $\sigma\mapsto z_0\sigma:\R^4\to\R^4$
is equal to $\Phi_{n_0}^{V_\infty}:\R^4\to\R^4$,
and therefore agrees with $\Id_4:\R^4\to\R^4$ to all orders at $\sigma_j$.
Then, since $Z\subseteq Z(G)$ and since $\iota(\sigma_j)=\tau_0$,
it follows that the map $\rho\mapsto z_0\rho:M\to M$
agrees with $\Id:M\to M$ at $\tau_0$ to all orders.
In particular, $\rho\mapsto z_0\rho:M\to M$
agrees with $\Id:M\to M$ at $\tau_0$ to order $k$.
Then, for all $\rho\in\pi^{-1}(\tau_0)$, we have $z_0\rho=\rho$.
So, since $\delta_0\in\pi^{-1}(\tau_0)$, we get $z_0\delta_0=\delta_0$.
That is, $z_0\in\Stab_Z(\delta_0)=S_0\cap Z$.
Let $C_0$ be the cyclic subgroup of $S_0\cap Z$ generated by~$z_0$.
Every nontrivial subgroup of an infinite cyclic group is infinite,
so $C_0$ is infinite.
So, because $C_0\subseteq S_0\cap Z$,
it follows that $S_0\cap Z$ is infinite, as desired.
\end{proof}


\bibliography{list}

\end{document}